\documentclass[10pt]{amsart}

\usepackage[final]{microtype}

\usepackage{mathtools}
\mathtoolsset{mathic}

\usepackage{amsmath,amsfonts,mathscinet,amssymb,amsthm,mathscinet}
\usepackage{bm,extarrows,upgreek,tensor,mathrsfs,textcomp,comment,braket,tensor}
\usepackage{eucal} 
\usepackage{stmaryrd} 
\usepackage{ulem} 
\usepackage[shortlabels]{enumitem}
\setlist[itemize]{parsep=0pt,itemsep=3pt,topsep=2pt}

 \usepackage{tabularx}

\usepackage[usenames,dvipsnames]{xcolor}

\usepackage[left=3cm,right=3cm]{geometry}

\usepackage{etoolbox}

\DeclareSymbolFont{usualmathcal}{OMS}{cmsy}{m}{n}
\DeclareSymbolFontAlphabet{\mathucal}{usualmathcal}

\makeatletter
\newcommand{\raisemath}[1]{\mathpalette{\raisem@th{#1}}}
\newcommand{\raisem@th}[3]{\raisebox{#1}{$#2#3$}}
\makeatother

\usepackage{tikz}
\usetikzlibrary{patterns,matrix,through,arrows,decorations.pathreplacing,shapes,decorations.markings,shadows,shapes.geometric,positioning,calc,backgrounds,fit}
\usetikzlibrary{decorations.markings}
\tikzstyle directed=[postaction={decorate,decoration={markings,
    mark=at position #1 with {\arrow{>}}}}]
\tikzstyle rdirected=[postaction={decorate,decoration={markings,
    mark=at position #1 with {\arrow{<}}}}]
\tikzset{anchorbase/.style={baseline={([yshift=-0.5ex]current bounding box.center)}},
cross line/.style={preaction={draw=white,line width=4pt,-}},
}

 \usepackage[all]{xy} \SelectTips{cm}{}

\usepackage[final=true]{hyperref,bookmark}
\hypersetup{
  colorlinks = true,
  urlcolor = blue,
  pdfkeywords = {},
  pdftitle = {},
  pdfsubject = {},
  pdfpagemode = UseNone,
  bookmarksopen = true,
  bookmarksopenlevel = 3,
  pdfdisplaydoctitle = true }

\newtheorem{theorem}{Theorem}[section]

\newtheorem{conjecture}[theorem]{Conjecture}
\newtheorem{corollary}[theorem]{Corollary}

\newtheorem{definition}[theorem]{Definition}
\newtheorem{example}[theorem]{Example}

\newtheorem{lemma}[theorem]{Lemma}

\newtheorem{proposition}[theorem]{Proposition}

\theoremstyle{remark}
\newtheorem{remark}[theorem]{Remark}

\numberwithin{equation}{section}

\def\cal#1{\mathcal{#1}}
\newcommand{\cat}[1]{\ensuremath{\mbox{\bfseries {\upshape {#1}}}}}
\renewcommand{\to}{\rightarrow}
\newcommand{\del}{\partial}
\newcommand{\fns}{\footnotesize}
\newcommand{\scs}{\scriptsize}

\let\epsilon=\varepsilon
\newcommand{\rmi}{\mathrm{i}}
\newcommand{\multt}{t}

\DeclareMathOperator{\Hom}{Hom}
\DeclareMathOperator{\HOM}{HOM}
\DeclareMathOperator{\twoHOM}{2HOM}
\DeclareMathOperator{\End}{End}

\DeclareMathOperator{\twoEnd}{2End}
\DeclareMathOperator{\twoEND}{2END}

\newcommand{\Rep}{\cat{Rep}}

\newcommand{\Tot}{\mathrm{Tot}}
\newcommand{\Vect}{\mathrm{Vect}}

\newcommand{\gr}{\mathrm{gr}}

\newcommand{\bV}{\textstyle{\bigwedge}}

\newcommand{\Sym}{\mathsf{Sym}}

\newcommand{\Hot}{\mathrm{K}}
\newcommand{\HH}{\mathrm{HH}}
\newcommand{\HHH}{\mathrm{HHH}}
\newcommand{\KhR}{\mathsf{KhR}}

\newcommand{\U}{\dot{{\mathrm{U}}}}
\newcommand{\UU}{\mathucal{U}}
\newcommand{\Udot}{\dot{\mathucal{U}}}
\newcommand{\Ucheck}{\check{\mathucal{U}}}

\newcommand{\one}{{\mathbf 1}}
\newcommand{\id}{\mathrm{id}}
\newcommand{\onel}{{\mathbf 1}_{\lambda}}
\newcommand{\onek}{{\mathbf 1}_{{\bf k}}}
\def\1{\mathrm{1}}

\newcommand{\vTr}{\operatorname{vTr}}
\newcommand{\hTr}{\operatorname{hTr}}

\newcommand{\sln}{\mathfrak{sl}_n} 
\newcommand{\slv}{\mathfrak{sl}_\infty} 
\newcommand{\slm}{\mathfrak{sl}_m}
\newcommand{\slnn}[1]{\mathfrak{sl}_{#1}}

\newcommand{\glm}{\mathfrak{gl}_m}
\newcommand{\glmn}{\mathfrak{gl}_{m|n}}
\newcommand{\glnn}[1]{\mathfrak{gl}_{#1}}
\newcommand{\slnegn}{\glnn{-n}}
\newcommand{\slzn}{\glnn{0|n}}

\newcommand{\nFoam}{n\cat{Foam}}
\newcommand{\vFoam}{\infty\cat{Foam}}

\newcommand{\vAFoam}{\infty\cal{A}\cat{Foam}}

\newcommand{\C}{{\mathbb C}}
\newcommand{\Z}{{\mathbb Z}}
\newcommand{\N}{{\mathbb N}}

\def\emph#1{{\sl #1\/}}
\def\mm{{\sl mutatis mutandis \/}}
\def\ap{{\sl a priori \/}}
\def\ie{{\sl i.e.\ \/}}
\def\eg{{\sl e.g.\ \/}}

\begin{document}
%

\title{Annular Evaluation and Link Homology}

\author{Hoel Queffelec}
\address{Mathematical Sciences Institute, Australian National University, Canberra, ACT 0200, Australia}
\email{hoel.queffelec@anu.edu.au}

\author{David E. V. Rose}
\address{Department of Mathematics, University of North Carolina, Phillips Hall CB \#3250, UNC-CH, Chapel Hill, NC 27599-3250, USA}
\email{davidrose@unc.edu}
  
\author {Antonio Sartori}
\address{Mathematisches Institut Albert-Ludwigs-Universit\"{a}t Freiburg, Eckerstra{\ss}e 1,
79104 Freiburg im Breisgau,
Germany}
\email{antonio.sartori@math.uni-freiburg.de}

\begin{abstract}
We use \textit{categorical annular evaluation} 
to give a uniform construction of both $\sln$ and HOMFLYPT Khovanov-Rozansky link homology,
as well as annular versions of these theories. 
Variations on our construction yield $\slnegn$ link homology, 
\ie a link homology theory associated to the Lie superalgebra $\slzn$, 
both for links in $S^3$ and in the thickened annulus.
In the $n=2$ case, this produces a categorification of the Jones polynomial 
that we show is \textit{distinct} from Khovanov homology, 
and gives a \textit{finite-dimensional} categorification of the colored Jones polynomial. 
This behavior persists for general $n$.
Our approach yields simple constructions of spectral sequences relating these theories, 
and emphasizes the roles of super vector spaces, categorical traces, 
and current algebras in link homology.
\end{abstract}

\maketitle

%
\section{Introduction}\label{sec:intro}
%

The years since Khovanov's pioneering work categorifying the Jones polynomial \cite{Kh1} 
have seen a medley of categorifications of type A link polynomials. 
Indeed, a large cast of authors have constructed link homology theories categorifying the 
$\sln$ Reshetikhin-Turaev and HOMFLYPT link polynomials using a variety of means 
\cite{SeSm,Man,CK01,CK02,Sussan,MS2,Webster}.
Despite using differing tools, all of these constructions follow a similar procedure: 
first, link crossings are formally resolved to give a chain complex whose terms are curves or webs (certain planar graphs);
next, this complex is interpreted as specifying a chain complex in a certain abelian (or additive) category, 
depending on the particular construction being used \textit{and} the link polynomial being categorified (either $\sln$ or HOMFLYPT);
and, finally, the complexes for the various crossings are pieced together, to yield a complex of graded vector spaces, 
whose homology is the link invariant.
Despite the variety of choices at the second step, for a \textit{fixed} type A link polynomial, 
almost all\footnote{One exception is the ``odd'' categorification of the Jones polynomial, given in \cite{ORS}. 
Another is the recent finite-rank invariant given in \cite{Cautis2}, which we will show surprisingly doesn't agree with 
Khovanov-Rozansky homology, but rather with the $\slnegn$ link homology constructed in this paper.} 
known categorifications have been shown to agree with the original constructions given by Khovanov and Rozansky \cite{KhR,KhR2} 
or the colored variants thereof \cite{Wu, Yon,WebWil}.
This begs the question: could we also remove the ``other variable'' in the second step, 
\ie can we find a uniform construction of all type A link homology theories that recovers the $\sln$ (for every $n$) or HOMFLYPT version 
in a simple way?

In this paper, we answer this question in the affirmative: given the input data of a graded super vector space $\cal{V}$ with an endomorphism $T$,
we construct a homology theory $\cal{H}_{\cal{V},T}$ that, for specific choices of $\cal{V}$ and $T$, recovers $\sln$ and HOMFLYPT homology. 
In fact, we show more: by varying our input data, we are able to explicitly construct 
$\slnegn$ link homology, a theory long-conjectured to exist. 
This produces an example of a homology theory for links in $S^3$ that categorifies the $\sln$ Reshetikhin-Turaev 
polynomial\footnote{Note: the $\sln$ and $\slnegn$ link \textit{polynomials} agree, up to an overall $\pm1$ factor.}, 
but, as we show, is \textit{distinct} from Khovanov-Rozansky homology.
Additionally, the colored variant of this theory gives a categorification of the $\Sym$-colored $\sln$ Reshetikhin-Turaev polynomial 
that is manifestly \textit{finite-dimensional}.
Our unifying construction clarifies various structural properties of link homology in type A, 
and gives simple constructions of spectral sequences relating the various theories.
We also show that the ``annular'' versions of $\sln$ homology can also be recovered from our framework, 
and we produce new annular invariants, corresponding to the HOMFLYPT and $\slnegn$ theories. 

Our construction utilizes \textit{annular evaluation}, a technique introduced in \cite{QR2} by the first two named authors
for studying homological invariants of links in the thickened annulus. The present paper shows that this tool is also useful for the study of links in $S^3$.
The core idea is as follows: the formal complex of webs assigned to a braid $\beta$ is identical (\ie independent of $n$) in all definitions 
of type A link homologies. Viewing this complex in the appropriate setting, it is possible to take its ``trace'', which produces an invariant 
of braid conjugacy, \ie an invariant of the braid closure, viewed as a link in the thickened annulus. 
Annular evaluation then provides a means to simplify this complex into a particularly nice form, again independent of our choice of $n$. 
Using the input data $\cal{V}$ and $T$, we are able to pass to a complex of graded super vector spaces, 
whose homology is an invariant of braid conjugacy. 
For special values of $\cal{V}$ and $T$, we obtain an invariant of the link $\cal{L}_\beta \subset S^3$ determined by the braid $\beta$.

\subsection{Annular evaluation for link polynomials}\label{sec:Decat}

In fact, this procedure is interesting even in the decategorified setting of link polynomials, 
which demonstrates many of the features of the general theory.
We now discuss this case in some detail, as a warm-up for the categorified results in the remainder of the paper. 

Given a link $\cal{L}$, we can compute its $\sln$ Reshetikhin-Turaev polynomial \cite{RT_ribbon} using the work of 
Cautis-Kamnitzer-Morrison \cite{CKM}, which generalizes the Kauffman bracket formulation of the Jones polynomial \cite{Kauff}, 
and the MOY formulation of the Reshetikhin-Turaev invariant \cite{MOY}. 
Each crossing is assigned a linear combination of $\sln$ webs via the rules:
\begin{equation}\label{eq:DecatCrossing}
\left\langle \;
\xy
(0,0)*{
\begin{tikzpicture}[scale=.5,rotate=270]
	\draw[very thick, <-] (0,1) to [out=0,in=180] (2,0);
	\draw[very thick, <-] (0,0) to [out=0,in=225] (.9,.4);
	\draw[very thick] (1.1,.6) to [out=45,in=180] (2,1);
\end{tikzpicture}
};
\endxy 
\; \right\rangle
=
q^{-1}
\xy
(0,0)*{
\begin{tikzpicture}[scale=.5,rotate=270]
	\draw[very thick, <-] (0,1) to [out=340,in=200] (2,1);
	\draw[very thick, <-] (0,0) to [out=20,in=160] (2,0);
\end{tikzpicture}
};
\endxy
-
\xy
(0,0)*{
\begin{tikzpicture}[scale=.2]
	\draw [very thick, directed=.65] (0,-.75) to (0,.75);
	\draw [very thick,->] (0,.75) to [out=30,in=270] (1.25,2.75);
	\draw [very thick,->] (0,.75) to [out=150,in=270] (-1.25,2.75); 
	\draw [very thick, directed=.5] (1.25,-2.75) to [out=90,in=330] (0,-.75);
	\draw [very thick, directed=.5] (-1.25,-2.75) to [out=90,in=210] (0,-.75);
	\node at (.875,0) {\tiny $2$};
\end{tikzpicture}
};
\endxy
\quad , \quad
\left\langle \;
\xy
(0,0)*{
\begin{tikzpicture}[scale=.5, rotate=270]
	\draw[very thick, <-] (0,0) to [out=0,in=180] (2,1);
	\draw[very thick, <-] (0,1) to [out=0,in=135] (.9,.6);
	\draw[very thick] (1.1,.4) to [out=315,in=180] (2,0);
\end{tikzpicture}
};
\endxy 
\; \right\rangle
=
- \xy
(0,0)*{
\begin{tikzpicture}[scale=.2]
	\draw [very thick, directed=.65] (0,-.75) to (0,.75);
	\draw [very thick,->] (0,.75) to [out=30,in=270] (1.25,2.75);
	\draw [very thick,->] (0,.75) to [out=150,in=270] (-1.25,2.75); 
	\draw [very thick, directed=.5] (1.25,-2.75) to [out=90,in=330] (0,-.75);
	\draw [very thick, directed=.5] (-1.25,-2.75) to [out=90,in=210] (0,-.75);
	\node at (.875,0) {\tiny $2$};
\end{tikzpicture}
};
\endxy
+
q
\xy
(0,0)*{
\begin{tikzpicture}[scale=.5, rotate=270]
	\draw[very thick, <-] (0,1) to [out=340,in=200] (2,1);
	\draw[very thick, <-] (0,0) to [out=20,in=160] (2,0);
\end{tikzpicture}
};
\endxy
\end{equation}
and these rules extend linearly to assign a $\Z[q,q^{-1}]$-linear combination $\langle \cal{L} \rangle$ of $\sln$ webs to 
a diagram for $\cal{L}$. 
Next, the rules\footnote{In these equations, $[k]=\frac{q^k-q^{-k}}{q-q^{-1}}$ are the quantum integers, 
and ${k \brack l}$ are quantum binomial coefficients.}:
\begin{equation}\label{eq:UpwardRel}
\xy %
(0,0)*{
\begin{tikzpicture}[scale=.25]
	\draw [very thick, directed=.55] (0,.75) to [out=90,in=210] (1,2.5);
	\draw [very thick, directed=.55] (1,-1) to [out=90,in=330] (0,.75);
	\draw [very thick, directed=.55] (-1,-1) to [out=90,in=210] (0,.75);
	\draw [very thick, directed=.55] (3,-1) to [out=90,in=330] (1,2.5);
	\draw [very thick, directed=.55] (1,2.5) to (1,4.25);
	\node at (-1,-1.5) {\tiny $k$};
	\node at (1,-1.5) {\tiny $l$};
	\node at (-1.375,1.5) {\tiny $k{+}l$};
	\node at (3,-1.5) {\tiny $p$};
	\node at (1,4.75) {\tiny $k{+}l{+}p$};
\end{tikzpicture}
};
\endxy
=
\xy
(0,0)*{
\begin{tikzpicture}[scale=.25]
	\draw [very thick, directed=.55] (0,.75) to [out=90,in=330] (-1,2.5);
	\draw [very thick, directed=.55] (-1,-1) to [out=90,in=210] (0,.75);
	\draw [very thick, directed=.55] (1,-1) to [out=90,in=330] (0,.75);
	\draw [very thick, directed=.55] (-3,-1) to [out=90,in=210] (-1,2.5);
	\draw [very thick, directed=.55] (-1,2.5) to (-1,4.25);
	\node at (1,-1.5) {\tiny $p$};
	\node at (-1,-1.5) {\tiny $l$};
	\node at (1.25,1.5) {\tiny $l{+}p$};
	\node at (-3,-1.5) {\tiny $k$};
	\node at (-1,4.75) {\tiny $k{+}l{+}p$};
\end{tikzpicture}
};
\endxy
\;\; , \;\;
\xy
(0,0)*{
\begin{tikzpicture}[scale=.2]
	\draw [very thick, directed=.55] (0,.75) to (0,2.5);
	\draw [very thick, directed=.55] (0,-2.75) to [out=30,in=330] (0,.75);
	\draw [very thick, directed=.55] (0,-2.75) to [out=150,in=210] (0,.75);
	\draw [very thick, directed=.55] (0,-4.5) to (0,-2.75);
	\node at (0,-5.1) {\tiny $k{+}l$};
	\node at (0,3.1) {\tiny $k{+}l$};
	\node at (-1.75,-1) {\tiny $k$};
	\node at (1.75,-1) {\tiny $l$};
\end{tikzpicture}
};
\endxy
={k+l \brack l}
\xy
(0,0)*{
\begin{tikzpicture}[scale=.2]
	\draw [very thick, directed=.55] (0,-4.5) to (0,2.5);
	\node at (0,-5.1) {\tiny $k{+}l$};
	\node at (0,3.1) {\tiny $k{+}l$};
\end{tikzpicture}
};
\endxy
\;\; , \;\;
\xy 
(0,0)*{
\begin{tikzpicture}[scale=.65]
	\draw[very thick,directed=.6] (.25,0) to (.25,.5);
	\draw[very thick,directed=.6] (.25,.5) to [out=45,in=225] (.75,1);
	\draw[very thick,directed=.6] (.75,1) to (.75,1.5);
	\draw[very thick,directed=.6] (.75,1.5) to [out=135,in=315] (.25,2);
	\draw[very thick,directed=.55] (.25,.5) to [out=135,in=225] (.25,2);
	\draw[very thick,directed=.6] (.25,2) to (.25,2.5);
	\draw[very thick,directed=.55] (1,0) to [out=90,in=330] (.75,1); 
	\draw[very thick,directed=.55] (.75,1.5) to [out=30,in=270] (1,2.5);
	\node at (.25,-.2) {\tiny $k$};
	\node at (1,-.2) {\tiny $l$};
	\node at (.25,2.7) {\tiny $k$};
	\node at (1,2.7) {\tiny $l$};
	\node at (.625,.5) {\tiny $1$};
	\node at (.625,2) {\tiny $1$};
\end{tikzpicture}
};
\endxy
-
\xy %
(0,0)*{
\begin{tikzpicture}[scale=.65,xscale=-1]
	\draw[very thick,directed=.6] (.25,0) to (.25,.5);
	\draw[very thick,directed=.6] (.25,.5) to [out=45,in=225] (.75,1);
	\draw[very thick,directed=.6] (.75,1) to (.75,1.5);
	\draw[very thick,directed=.6] (.75,1.5) to [out=135,in=315] (.25,2);
	\draw[very thick,directed=.55] (.25,.5) to [out=135,in=225] (.25,2);
	\draw[very thick,directed=.6] (.25,2) to (.25,2.5);
	\draw[very thick,directed=.55] (1,0) to [out=90,in=330] (.75,1);
	\draw[very thick,directed=.55] (.75,1.5) to [out=30,in=270] (1,2.5);
	\node at (.25,-.2) {\tiny $l$};
	\node at (1,-.2) {\tiny $k$};
	\node at (.25,2.7) {\tiny $l$};
	\node at (1,2.7) {\tiny $k$};
	\node at (.625,.5) {\tiny $1$};
	\node at (.625,2) {\tiny $1$};
\end{tikzpicture}
};
\endxy
=
[k-l]
\xy
(0,0)*{
\begin{tikzpicture}[scale=.6,xscale=-1]
	\draw[very thick,directed=.6] (.25,0) to (.25,2.5);
	\draw[very thick,directed=.55] (1,0) to (1,2.5);
	\node at (.25,-.2) {\tiny $l$};
	\node at (1,-.2) {\tiny $k$};
	\node at (.25,2.7) {\tiny $l$};
	\node at (1,2.7) {\tiny $k$};
\end{tikzpicture}
};
\endxy
\end{equation}
\begin{equation}\label{eq:DownwardRel}
\xy
(0,0)*{
\begin{tikzpicture}[scale=.2]
	\draw [very thick, directed=.55] (0,.75) to (0,2.5);
	\draw [very thick, rdirected=.45] (0,-2.75) to [out=30,in=330] (0,.75);
	\draw [very thick, directed=.55] (0,-2.75) to [out=150,in=210] (0,.75);
	\draw [very thick, directed=.55] (0,-4.5) to (0,-2.75);
	\node at (0,-5.1) {\tiny $k$};
	\node at (0,3.1) {\tiny $k$};
	\node at (-2.625,-1) {\tiny $k{+}l$};
	\node at (1.75,-1) {\tiny $l$};
\end{tikzpicture}
};
\endxy
={n-k \brack l}
\xy
(0,0)*{
\begin{tikzpicture}[scale=.2]
	\draw [very thick, directed=.55] (0,-4.5) to (0,2.5);
	\node at (0,-5.1) {\tiny $k{+}l$};
	\node at (0,3.1) {\tiny $k{+}l$};
\end{tikzpicture}
};
\endxy
\;\; , \;\;
\xy
(0,0)*{
\begin{tikzpicture}[scale=.25]
	\draw[very thick] (0,0) circle (1);
	\draw[very thick,->] (1,0) to (1,-.1);
	\node at (1,-1) {\tiny$k$};
\end{tikzpicture}
};
\endxy
=
{n \brack k}
\;\; , \;\;
\end{equation}
together with those obtained from these via symmetry or reversing all arrows 
(and others involving ``tags,'' which are not relevant to the current paper)
are used to evaluate each web appearing in $\langle \cal{L} \rangle$, 
and hence express $\langle\cal{L}\rangle$ as an element $\tilde{P}_n(\cal{L}) \in \Z[q,q^{-1}]$.
Rescaling\footnote{Here $w_\cal{L}$ denotes the writhe of the chosen diagram for $\cal{L}$.} 
then gives the $\sln$ Reshetikhin-Turaev polynomial 
$P_n(\cal{L}) = q^{n w_{\cal{L}}} \tilde{P}_n(\cal{L})$, which is independent of the chosen diagram for $\cal{L}$.

Note that equations \eqref{eq:DecatCrossing} and \eqref{eq:UpwardRel} are both independent of $n$. 
We can thus attempt to wait ``as long as possible'' before applying the relations in equation \eqref{eq:DownwardRel} in 
expressing $\langle\cal{L}\rangle$ as a Laurent polynomial.
To most effectively do so, we restrict the types of link diagrams we consider. 
Recall the following well known result:
\begin{theorem}[Alexander's Theorem]\label{thm:Alexander}
Given any link $\cal{L} \subset S^3$, there exists a braid $\beta$ so that $\cal{L}$ is isotopic to the closure of $\beta$.
\end{theorem}
To fix notation, if $\beta$ is a braid, we denote its closure, viewed as a link in $S^3$, by $\cal{L}_\beta$. 
We can also consider the closure of $\beta$ as a link in the thickened annulus $\cal{A} \times [0,1]$, which we denote by $\widehat{\beta}$. 
The latter maps to $\cal{L}_\beta$ via the standard inclusion of the thickened annulus into $S^3$. 
In summary:
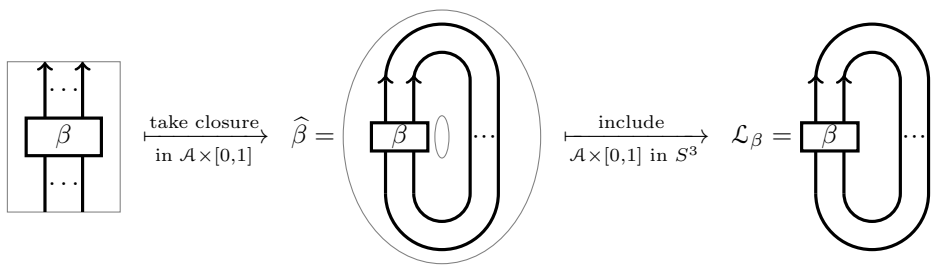
\[
\begin{tikzpicture}[anchorbase, scale=1,rotate=270]
\draw [gray] (0,0) rectangle (2,1.5);
\draw[very thick,<-] (0,.5) to (.75,.5);
\draw[very thick,<-] (0,1) to (.75,1);
\draw[very thick] (.75,.25) rectangle (1.25,1.25);
\node at (1,.75) {$\beta$};
\draw[very thick] (1.25,.5) to (2,.5);
\draw[very thick] (1.25,1) to (2,1);
\node at (.375,.75) {$\hdots$};
\node at (1.625,.75) {$\hdots$};
\end{tikzpicture}
\;\; \xmapsto[{\text{in } \cal{A} \times [0,1]}]{\text{take closure}} \;\;
\widehat{\beta} = 
\begin{tikzpicture}[anchorbase,scale=.75,rotate=270]
\draw [gray]  (1,1.5) ellipse (.375 and .125);
\draw [gray] (1,1.5) ellipse (2.25 and 1.75);
\draw[very thick] (0,.5) to (.75,.5);
\draw[very thick] (0,1) to (.75,1);
\draw[very thick, ->] (0,.5) to (-.1,.5);
\draw[very thick, ->] (0,1) to (-.1,1);
\draw[very thick] (.75,.25) rectangle (1.25,1.25);
\node at (1,.75) {$\beta$};
\draw[very thick] (1.25,.5) to (2,.5);
\draw[very thick] (1.25,1) to (2,1);
\draw [very thick] (0,1) arc (270:90:.5) -- (2,2) arc (90:-90:.5);
\draw [very thick] (0,.5) arc (270:90:1) -- (2,2.5) arc (90:-90:1);
\node at (1,2.1) {$\cdot$};
\node at (1,2.25) {$\cdot$};
\node at (1,2.4) {$\cdot$};
\end{tikzpicture}
\;\; \xmapsto[{\cal{A} \times [0,1] \text{ in } S^3}]{\text{include }} \;\;
\cal{L}_\beta = 
\begin{tikzpicture}[anchorbase,scale=.75,rotate=270]
\draw[very thick] (0,.5) to (.75,.5);
\draw[very thick] (0,1) to (.75,1);
\draw[very thick, ->] (0,.5) to (-.1,.5);
\draw[very thick, ->] (0,1) to (-.1,1);
\draw[very thick] (.75,.25) rectangle (1.25,1.25);
\node at (1,.75) {$\beta$};
\draw[very thick] (1.25,.5) to (2,.5);
\draw[very thick] (1.25,1) to (2,1);
\draw [very thick] (0,1) arc (270:90:.5) -- (2,2) arc (90:-90:.5);
\draw [very thick] (0,.5) arc (270:90:1) -- (2,2.5) arc (90:-90:1);
\node at (1,2.1) {$\cdot$};
\node at (1,2.25) {$\cdot$};
\node at (1,2.4) {$\cdot$};
\end{tikzpicture}
\]

Given $\beta$, we can hence consider $\langle \widehat{\beta} \rangle$, which is a linear combination of 
webs embedded in the thickened annulus, \ie an element of the $\sln$ web skein module of the 
annulus (given by webs modulo isotopy and the relations above). 
We have that $\langle \widehat{\beta} \rangle \mapsto \langle \cal{L}_{\beta} \rangle$ under the inclusion 
$\cal{A} \times [0,1] \hookrightarrow S^3$. 
The following result, 
proved in \cite{QR2} (or as the decategorification of Theorem \ref{thm:foamObjDec} below) allows us to simplify 
$\langle \widehat{\beta} \rangle$, \textit{independent of our choice of $n$}, before passing to $\langle \cal{L}_{\beta} \rangle$.

\begin{theorem}[Decategorified annular evaluation]\label{thm:DecatAnnEval}
The annular closure of any braid-like web can be expressed as a $\Z[q,q^{-1}]$-linear combination of essential, labeled, concentric circles, 
using only the web relations in equation \eqref{eq:UpwardRel}, and those obtained from these via symmetry/reversing all orientations.
\end{theorem}

Using this, we can express $\langle \widehat{\beta} \rangle$ as a $\Z[q,q^{-1}]$-linear combination of essential, labeled circles, 
which we denote\footnote{ 
Moreover, $\langle \widehat{\beta} \rangle_\circ$ is unique once we insist that the labelings on the essential circles are radially non-decreasing, 
hence is an invariant of the conjugacy class of $\beta$ (\ie of the annular link $\widehat{\beta}$).}
by $\langle \widehat{\beta} \rangle_\circ$. 
Finally, we can recover the $\sln$ link polynomials $P_n(\cal{L}_\beta)$ by using the appropriate version of 
the second relation in equation \eqref{eq:DownwardRel} to ``evaluate'' the labeled circles. 
Since this formula is polynomial in $q^n$, we can also recover the HOMFLYPT polynomial \cite{HOMFLY,PT}
by using the value of the latter on $k$-colored unknots in this evaluation. To be precise:

\begin{proposition}\label{prop:WedgePoly}
The link invariant $P_n(\cal{L}_\beta)$ is recovered from 
the braid conjugacy invariant $\langle \widehat{\beta} \rangle_\circ$
by multiplying by $q^{n w_{\beta}}$ and assigning the value ${n \brack k}$ to $k$-labeled circles. 
The same is true for the HOMFLYPT polynomial $P_\infty(\cal{L}_\beta)$ if we instead multiply by $a^{w_{\beta}}$ 
and use the assignment
\[
\xy
(0,0)*{
\begin{tikzpicture}[scale=.25]
	\draw[very thick] (0,0) circle (1);
	\draw[very thick,->] (1,0) to (1,-.1);
	\node at (1,-1) {\tiny$k$};
\end{tikzpicture}
};
\endxy \mapsto q^{k^2} a^{-k} \prod_{i=1}^k \frac{1-a^2q^{2-2i}}{1-q^{2i}} 
\]
\end{proposition}

Surprisingly, there are \textit{different} values for $k$-labeled circles that also produce the invariants $P_n$ and $P_\infty$ 
when specializing $\langle \widehat{\beta} \rangle_\circ$. 
The $\sln$ web formulation of link polynomials presented above is based on the Cautis-Kamnitzer-Morrison description 
of the category of finite-dimensional $U_q(\sln)$ representations, 
in which webs describe intertwiners between tensor products of the fundamental $U_q(\sln)$ representations $\bV^k(\C_q^n)$.
In particular, the values for circles above are the dimensions of the (quantum) exterior powers.
However, in \cite{RTub} it is shown that the relations in equation \eqref{eq:UpwardRel} also hold in a diagrammatic description 
of $\Rep(U_q(\sln))$ in which webs describe intertwiners between tensor products of the symmetric representations $\Sym^k(\C_q^n)$.
It follows that we can obtain the \textit{same} $\sln$ Reshetikhin-Turaev invariants $P_n(\cal{L}_\beta)$ via the annular evaluation 
procedure, but instead using the evaluation
\begin{equation}\label{eq:SymCircle}
\xy
(0,0)*{
\begin{tikzpicture}[scale=.25]
	\draw[very thick] (0,0) circle (1);
	\draw[very thick,->] (1,0) to (1,-.1);
	\node at (1,-1) {\tiny$k$};
\end{tikzpicture}
};
\endxy
=
{n+k-1 \brack k}
\end{equation}
at the final step.
To state things precisely, 
we need the symmetric webs version of equation \eqref{eq:DecatCrossing}.
In this setting, the crossing formulae become:
\begin{equation}\label{eq:DecatSymCrossing}
\left\{ \;
\xy
(0,0)*{
\begin{tikzpicture}[scale=.5,rotate=270]
	\draw[very thick, <-] (0,1) to [out=0,in=180] (2,0);
	\draw[very thick, <-] (0,0) to [out=0,in=225] (.9,.4);
	\draw[very thick] (1.1,.6) to [out=45,in=180] (2,1);
\end{tikzpicture}
};
\endxy 
\; \right\}
=
-q
\xy
(0,0)*{
\begin{tikzpicture}[scale=.5,rotate=270]
	\draw[very thick, <-] (0,1) to [out=340,in=200] (2,1);
	\draw[very thick, <-] (0,0) to [out=20,in=160] (2,0);
\end{tikzpicture}
};
\endxy
+
\xy
(0,0)*{
\begin{tikzpicture}[scale=.2]
	\draw [very thick, directed=.65] (0,-.75) to (0,.75);
	\draw [very thick,->] (0,.75) to [out=30,in=270] (1.25,2.75);
	\draw [very thick,->] (0,.75) to [out=150,in=270] (-1.25,2.75); 
	\draw [very thick, directed=.5] (1.25,-2.75) to [out=90,in=330] (0,-.75);
	\draw [very thick, directed=.5] (-1.25,-2.75) to [out=90,in=210] (0,-.75);
	\node at (.875,0) {\tiny $2$};
\end{tikzpicture}
};
\endxy
\quad , \quad
\left\{ \;
\xy
(0,0)*{
\begin{tikzpicture}[scale=.5, rotate=270]
	\draw[very thick, <-] (0,0) to [out=0,in=180] (2,1);
	\draw[very thick, <-] (0,1) to [out=0,in=135] (.9,.6);
	\draw[very thick] (1.1,.4) to [out=315,in=180] (2,0);
\end{tikzpicture}
};
\endxy 
\; \right\}
=
\xy
(0,0)*{
\begin{tikzpicture}[scale=.2]
	\draw [very thick, directed=.65] (0,-.75) to (0,.75);
	\draw [very thick,->] (0,.75) to [out=30,in=270] (1.25,2.75);
	\draw [very thick,->] (0,.75) to [out=150,in=270] (-1.25,2.75); 
	\draw [very thick, directed=.5] (1.25,-2.75) to [out=90,in=330] (0,-.75);
	\draw [very thick, directed=.5] (-1.25,-2.75) to [out=90,in=210] (0,-.75);
	\node at (.875,0) {\tiny $2$};
\end{tikzpicture}
};
\endxy
-q^{-1}
\xy
(0,0)*{
\begin{tikzpicture}[scale=.5, rotate=270]
	\draw[very thick, <-] (0,1) to [out=340,in=200] (2,1);
	\draw[very thick, <-] (0,0) to [out=20,in=160] (2,0);
\end{tikzpicture}
};
\endxy
\end{equation}
where the $2$-labeled edge now corresponds to the representation $\Sym^2(\C_q^n)$.
We can again apply Theorem \ref{thm:DecatAnnEval} to the braid conjugacy invariant $\{ \widehat{\beta} \}$ to 
obtain $\{ \widehat{\beta} \}_\circ$, which is again a linear combination of labeled circles. 
We can then recover $P_n(\cal{L}_\beta)$ as before by rescaling and evaluating circles at the dimensions of the 
corresponding symmetric power representations, \ie using equation \eqref{eq:SymCircle}. 

\begin{proposition}\label{prop:SymPoly}
The link invariant $P_n(\cal{L}_\beta)$ is recovered from 
the braid conjugacy invariant $\{ \widehat{\beta} \}_\circ$
by multiplying by $q^{n w_{\beta}}$ and assigning the value ${n+k-1 \brack k}$ to $k$-labeled circles. 
The same is true for the HOMFLYPT polynomial $P_\infty(\cal{L}_\beta)$ if we instead multiply by $a^{w_{\beta}}$ 
and use the assignment
\[
\xy
(0,0)*{
\begin{tikzpicture}[scale=.25]
	\draw[very thick] (0,0) circle (1);
	\draw[very thick,->] (1,0) to (1,-.1);
	\node at (1,-1) {\tiny$k$};
\end{tikzpicture}
};
\endxy \mapsto q^{k} a^{-k} \prod_{i=1}^k \frac{1-a^2q^{2i-2}}{1-q^{2i}} =: \mathsf{c}_{\Sym^k}
\]
\end{proposition}

In fact, we can re-state Proposition \ref{prop:WedgePoly} purely in terms of the bracket 
$\{ \widehat{\beta} \}_\circ$ as well. 
The values for crossings in equation \eqref{eq:DecatCrossing} can be obtained 
from those in equation \eqref{eq:DecatSymCrossing} by using the substitution $q \leftrightarrow q^{-1}$ and multiplying by $-1$.
It follows that 
\begin{equation}\label{eq:WedgeToSym}
\langle \widehat{\beta} \rangle_\circ = (-1)^{w_\beta} \{ \widehat{\beta} \}_\circ \big|_{q= q^{-1}}
\end{equation}

\begin{proposition}\label{prop:WedgePoly2}
The link invariant $P_n(\cal{L}_\beta)$ is recovered from 
the braid conjugacy invariant $\{ \widehat{\beta} \}_\circ$ by first using the substitution $q\leftrightarrow q^{-1}$,
then multiplying by $(-1)^{w_\beta} q^{n w_{\beta}}$ and assigning the value ${n \brack k}$ to $k$-labeled circles. 
The same is true for the HOMFLYPT polynomial $P_\infty(\cal{L}_\beta)$ if we instead multiply by $ (-1)^{w_\beta} a^{w_{\beta}}$ 
and use the assignment
\[
\xy
(0,0)*{
\begin{tikzpicture}[scale=.25]
	\draw[very thick] (0,0) circle (1);
	\draw[very thick,->] (1,0) to (1,-.1);
	\node at (1,-1) {\tiny$k$};
\end{tikzpicture}
};
\endxy \mapsto q^{k^2} a^{-k} \prod_{i=1}^k \frac{1-a^2q^{2-2i}}{1-q^{2i}} =: \mathsf{c}_{\wedge^k}
\]
\end{proposition}

Taken together, Propositions \ref{prop:SymPoly} and \ref{prop:WedgePoly2} show that using only one formula for the crossings, 
we can obtain the $\sln$ and HOMFLYPT link polynomials in two distinct ways, 
by using different values for our $k$-labeled circles and rescalings.

\begin{remark}\label{rem:Mirror}
In the HOMFLYPT case, we can rewrite the values of $k$-labeled circles in Proposition \ref{prop:WedgePoly2} as
\[
\mathsf{c}_{\wedge^k} 
= (-1)^k q^{-k} a^{-k} \prod_{i=1}^k \frac{1-a^2q^{2-2i}}{1-q^{-2i}}
\]
which equals $(-1)^k$ times the value $\mathsf{c}_{\Sym^k}$ of the $\Sym^k$-colored unknot, 
after applying the substitution $q \leftrightarrow q^{-1}$.
The sum of the labelings on the circles in each term in $\{ \widehat{\beta} \}_\circ$ equals 
the number of strands in the braid $\beta$, which we denote\footnote{More generally, given a colored braid $\beta$, 
we let $i_\beta$ denote the sum of the labelings of the strands in $\beta$.} 
by $i_\beta$. 
Combining this with Propositions \ref{prop:SymPoly} and \ref{prop:WedgePoly2} we find that
\[
P_\infty(\cal{L}_\beta)\big|_{q=q^{-1}} = \left( a^{w_\beta} \{ \widehat{\beta} \}_\circ\big|_{\mathsf{c}_{\Sym^k}} \right)\Big|_{q=q^{-1}}
= (-1)^{i_\beta} a^{w_\beta} \left( \{ \widehat{\beta} \}_\circ\big|_{q=q^{-1}} \right)\Big|_{\mathsf{c}_{\wedge^k}} 
= (-1)^{i_\beta + w_\beta} P_\infty(\cal{L}_\beta).
\]
It is easy to see that 
$i_\beta + w_\beta$ has the same parity as the number of components $\#\cal{L}_\beta$ of the link $\cal{L}_\beta$.
Thus,
\begin{equation}\label{eq:ourMirror}
P_\infty(\cal{L}_\beta)\big|_{q=q^{-1}} = (-1)^{\#\cal{L}_\beta} P_\infty(\cal{L}_\beta).
\end{equation}
We will categorify this fact in Section \ref{sec:Mirror}, 
where we interpret it as a manifestation of the ``mirror symmetry'' property of the HOMFLYPT polynomial.
\end{remark}

\begin{remark}\label{rem:Coloring}
More generally, throughout this section we could have worked with braids/links whose strands are colored by 
positive integers. Using colored versions\footnote{Colored crossings can be defined in terms of uncolored crossings  
by iteratively using the second relation in equation \eqref{eq:UpwardRel} to ``explode'' $k$-labeled strands into $1$-labeled strands, 
provided we require the ``pitchfork'' relation that slides trivalent vertices through crossings.
Alternatively, in the symmetric case, 
these are given \eg as the decategorification of the crossing formulae in equation \eqref{eq:slnColoredCrossing}.
The antisymmetric case is obtained from the symmetric case by taking $q \leftrightarrow q^{-1}$ and multiplying 
by $(-1)^{kl}$, where $k$ and $l$ are the labelings of the involved strands; this generalizes equation \eqref{eq:WedgeToSym}.}
of our crossing formulae,
and the same annular evaluation process, 
the circle evaluations in Proposition \ref{prop:WedgePoly2} produce the $\bV$-colored link polynomials, 
while those in Proposition \ref{prop:SymPoly} recover the $\Sym$-colored invariants 
(in both the $\sln$ Reshetikhin-Turaev and HOMFLYPT setting).

To be precise, we must additionally rescale in order for the link polynomials to be independent of framing. 
For the choice of circle evaluations $\mathsf{c}_{\Sym^\bullet}$, a computation shows that: 
\[
\left\{
\xy
(0,0)*{
\begin{tikzpicture}[scale=.5]
	\draw[very thick,->] (1,1) to [out=270,in=0] (.75,.75) to [out=180,in=270] (0,2);
	\draw[line width = 4, white] (0,0) to [out=90,in=180] (.75,1.25) to [out=0,in=90] (1,1);	
	\draw[very thick] (0,0) to [out=90,in=180] (.75,1.25) to [out=0,in=90] (1,1);
	\node at (0,-.25) {\tiny$k$};
\end{tikzpicture}
};
\endxy
\right\}
= a^{-k}q^{k-k^2}
\left\{
\xy
(0,0)*{
\begin{tikzpicture}[scale=.5]
	\draw[very thick,->] (0,0) to (0,2);
	\node at (0,-.25) {\tiny$k$};
\end{tikzpicture}
};
\endxy
\right\}
\;\; , \;\;
\left\{
\xy
(0,0)*{
\begin{tikzpicture}[scale=.5]
	\draw[very thick] (0,0) to [out=90,in=180] (.75,1.25) to [out=0,in=90] (1,1);
	\draw[line width = 4, white] (1,1) to [out=270,in=0] (.75,.75) to [out=180,in=270] (0,2);
	\draw[very thick,->] (1,1) to [out=270,in=0] (.75,.75) to [out=180,in=270] (0,2);
	\node at (0,-.25) {\tiny$k$};
\end{tikzpicture}
};
\endxy
\right\}
= a^{k}q^{k^2-k}
\left\{
\xy
(0,0)*{
\begin{tikzpicture}[scale=.5]
	\draw[very thick,->] (0,0) to (0,2);
	\node at (0,-.25) {\tiny$k$};
\end{tikzpicture}
};
\endxy
\right\}
\]
If we define the writhe of a colored braid $\beta$ by
\[
w_\beta := \sum_{k>0} k\left(\# \bigg(
\xy
(0,0)*{
\begin{tikzpicture}[scale=.375,rotate=270]
	\draw[very thick, <-] (0,1) to [out=0,in=180] (2,0);
	\draw[very thick, <-] (0,0) to [out=0,in=225] (.9,.4);
	\draw[very thick] (1.1,.6) to [out=45,in=180] (2,1);
	\node at (2.25,0) {\tiny$k$};
	\node at (2.25,1) {\tiny$k$};
\end{tikzpicture}
};
\endxy 
\bigg)
- \# \bigg(
\xy
(0,0)*{
\begin{tikzpicture}[scale=.375, rotate=270]
	\draw[very thick, <-] (0,0) to [out=0,in=180] (2,1);
	\draw[very thick, <-] (0,1) to [out=0,in=135] (.9,.6);
	\draw[very thick] (1.1,.4) to [out=315,in=180] (2,0);
	\node at (2.25,0) {\tiny$k$};
	\node at (2.25,1) {\tiny$k$};
\end{tikzpicture}
};
\endxy
\bigg)
\right)
\]
and a ``quadratic'' writhe by
\[
W_\beta := \sum_{k>0} k^2 \left(\# \bigg(
\xy
(0,0)*{
\begin{tikzpicture}[scale=.375,rotate=270]
	\draw[very thick, <-] (0,1) to [out=0,in=180] (2,0);
	\draw[very thick, <-] (0,0) to [out=0,in=225] (.9,.4);
	\draw[very thick] (1.1,.6) to [out=45,in=180] (2,1);
	\node at (2.25,0) {\tiny$k$};
	\node at (2.25,1) {\tiny$k$};
\end{tikzpicture}
};
\endxy 
\bigg)
- \# \bigg(
\xy
(0,0)*{
\begin{tikzpicture}[scale=.375, rotate=270]
	\draw[very thick, <-] (0,0) to [out=0,in=180] (2,1);
	\draw[very thick, <-] (0,1) to [out=0,in=135] (.9,.6);
	\draw[very thick] (1.1,.4) to [out=315,in=180] (2,0);
	\node at (2.25,0) {\tiny$k$};
	\node at (2.25,1) {\tiny$k$};
\end{tikzpicture}
};
\endxy
\bigg)
\right)
\]
then $a^{w_\beta}q^{W_\beta - w_\beta} \{ \widehat{\beta} \}_\circ \big|_{\mathsf{c}_{\Sym^\bullet}}$ is a colored link invariant, 
which agrees with the $\Sym$-colored HOMFLYPT polynomial.
Similarly, using $\mathsf{c}_{\wedge^\bullet}$ for our circle evaluations gives
\[
\left\langle
\xy
(0,0)*{
\begin{tikzpicture}[scale=.5]
	\draw[very thick,->] (1,1) to [out=270,in=0] (.75,.75) to [out=180,in=270] (0,2);
	\draw[line width = 4, white] (0,0) to [out=90,in=180] (.75,1.25) to [out=0,in=90] (1,1);	
	\draw[very thick] (0,0) to [out=90,in=180] (.75,1.25) to [out=0,in=90] (1,1);
	\node at (0,-.25) {\tiny$k$};
\end{tikzpicture}
};
\endxy
\right\rangle
= a^{-k}(-q)^{k^2-k}
\left\langle
\xy
(0,0)*{
\begin{tikzpicture}[scale=.5]
	\draw[very thick,->] (0,0) to (0,2);
	\node at (0,-.25) {\tiny$k$};
\end{tikzpicture}
};
\endxy
\right\rangle
\;\; , \;\;
\left\langle
\xy
(0,0)*{
\begin{tikzpicture}[scale=.5]
	\draw[very thick] (0,0) to [out=90,in=180] (.75,1.25) to [out=0,in=90] (1,1);
	\draw[line width = 4, white] (1,1) to [out=270,in=0] (.75,.75) to [out=180,in=270] (0,2);
	\draw[very thick,->] (1,1) to [out=270,in=0] (.75,.75) to [out=180,in=270] (0,2);
	\node at (0,-.25) {\tiny$k$};
\end{tikzpicture}
};
\endxy
\right\rangle
= a^{k}(-q)^{k-k^2}
\left\langle
\xy
(0,0)*{
\begin{tikzpicture}[scale=.5]
	\draw[very thick,->] (0,0) to (0,2);
	\node at (0,-.25) {\tiny$k$};
\end{tikzpicture}
};
\endxy
\right\rangle
\]
so $a^{w_\beta}(-q)^{w_\beta - W_\beta} \langle \widehat{\beta} \rangle_\circ \big|_{\mathsf{c}_{\wedge^\bullet}}$ is 
a colored link invariant that agrees with the $\bV$-colored HOMFLYPT polynomial. 
In both cases, the corresponding colored $\sln$ Reshetikhin-Turaev polynomials can be recovered by setting $a=q^n$.
In Section \ref{sec:Mirror}, 
we will show that the colored analogue of Remark \ref{rem:Mirror} corresponds to the ``mirror symmetry'' between $\Sym$- and $\bV$-colored 
HOMFLYPT polynomials.
\end{remark}

\subsection{The categorified story (\ie the contents of the remainder of this paper)}

At the decategorified level, we've seen that all link invariants of type A, 
\ie the $\sln$ and HOMFLYPT link polynomials,
can be recovered from the braid conjugacy invariant $\{ \beta \}_\circ$. 
In the remainder of this paper, we will promote this result to the categorical level. 
Using the annular foam technology developed in previous work of the first two named authors \cite{QR2}, 
we give a uniform construction of both $\sln$ and HOMFLYPT Khovanov-Rozansky homology. 
As in the decategorified setting, we do so by first building a ``universal'' type A braid conjugacy invariant, 
and then recover Khovanov-Rozansky homology by assigning certain data to essential circles. 
Our results do much more than simply re-package known invariants: 
our approach shows that there is freedom in the choice of data assigned at the final step, 
and for novel choices of this data we construct homological link invariants distinct from the Khovanov-Rozansky theory. 
Moreover, our approach provides a simple criteria for showing that a given link homology theory agrees with those we construct; 
this should prove useful \eg in showing that recent categorifications of the HOMFLYPT polynomial \cite{ObRoz, GNR} agree with 
that of Khovanov-Rozansky. 

The remainder of the paper is organized as follows:

In Section \ref{sec:Basics} we quickly review some basic notions from knot theory, 
and establish conventions and notation for graded vector spaces and categories.

In Section \ref{sec:Foams}, we discuss the category of annular foams, 
which is the natural setting for a chain complex $\llbracket \widehat{\beta} \rrbracket$ that categorifies the 
braid conjugacy invariant $\{\widehat{\beta} \}$. 
We then give an updated perspective on the categorical annular evaluation from \cite{QR2} (including a streamlined proof), 
which allows us to pass from $\llbracket \widehat{\beta} \rrbracket$ to a homotopy equivalent complex $\llbracket \widehat{\beta} \rrbracket_\circ$ 
that, similarly to $\{\widehat{\beta} \}_\circ$, is given in terms of labeled, essential circles in the annulus. 
The complex $\llbracket \widehat{\beta} \rrbracket_\circ$ can be interpreted as a universal homological braid conjugacy invariant in type A.

In Section \ref{sec:CatUqTrace}, we cast our results in terms of traces of categorified quantum groups,
or equivalently, the current algebras $\glm[t]$. 
Using this, we are able to prove a result that characterizes monoidal functors from the ``circular" subcategory 
of annular foams to the category of graded, super vector spaces 
in terms of the data of a single graded, super vector space $\cal{V}$, together with an endomorphism $T:\cal{V} \to \cal{V}$. 
Using this, we see that one obtains a homology theory for each choice of $\cal{V}$ and $T$, which is an invariant of braid conjugacy.

In Section \ref{sec:identifying}, we investigate particular choices of these parameters, and show that we can recover $\sln$ and HOMFLYPT 
Khovanov-Rozansky link homology for special choices.
In particular, this shows that for certain choices of $\cal{V}$ and $T$ our braid conjugacy invariants 
descend to give invariants of links in $S^3$. 
More surprisingly, we show that different choices, 
which correspond to passing from the ``skew-symmetric'' evaluation in Proposition \ref{prop:WedgePoly} to the ``symmetric'' evaluation in Proposition \ref{prop:SymPoly}, 
produce a homology theory for links $\cal{L} \subset S^3$ that categorifies the $\sln$ link polynomials, but is \textit{distinct} from the Khovanov-Rozansky theory.
This invariant can be interpreted as the link homology theory associated to the Lie superalgebra $\slnegn := \slzn$, 
so we denote it $\cal{H}_{-n}(\cal{L})$.
We also show how to recover the annular $\sln$ homology defined in \cite{QR}, 
and also construct HOMFLYPT and $\slnegn$ (or symmetric) versions of this theory, the latter of which carries an action of the Lie algebra $\sln$. 
Finally, we briefly discuss deformations of these theories.

In Section \ref{sec:Diff}, we build spectral sequences relating HOMFLYPT link homology to both 
$\sln$ and $\slnegn$ link homology.
In doing so, we also provide a categorification of the ``mirror symmetry'' in Remark \ref{rem:Mirror} above.

Finally, in Section \ref{sec:examples} we compute some examples.

\subsection{Acknowledgements}

The authors would like to thank 
Christian Blanchet, 
Sergei Gukov, 
Mikhail Khovanov, 
Aaron Lauda, 
Hugh Morton, 
Lev Rozansky, 
Daniel Tubbenhauer, 
and 
Paul Wedrich
for useful discussions.
Special thanks to Tony Licata for his many helpful suggestions in the early stages of this work, 
and his continuous support throughout its completion.
This material was presented at the workshop ``Quantum topology and categorified representation theory'' at the Isaac Newton Institute, 
as part of the program ``Homology Theories in Low Dimensional Topology.''
Some of the writing of this document was done during this program, and we thank the Isaac Newton Institute for their support.
D.R. was partially supported by an NSA Young Investigator Grant and a Simons Collaboration Grant.

\begin{remark}
After this material was presented 
at the Isaac Newton Institute, 
we were informed by Emmanuel Wagner that he and Louis-Hadrien Robert planned to combine our annular technology 
from \cite{QR2}
with a ``symmetric foam evaluation'' to give a different construction of our invariant $\cal{H}_{-n}(\cal{L})$.
Due to the unusually prolonged preparation of this manuscript, 
they slightly beat us to the presses; 
find their recent (and interesting) work here \cite{RobWag}, 
which they call symmetric Khovanov-Rozansky homology.
\end{remark}

%
\section{Basic notions and conventions}\label{sec:Basics}
%

\subsection{Braids and links}

Throughout, we build our link invariants using Alexander's Theorem \ref{thm:Alexander}, 
\ie starting with a braid $\beta$ whose closure $\cal{L}_\beta$ in $S^3$ is our desired link. 
Recall that we also denote the closure of $\beta$ in the thickened annulus $\cal{A} \times [0,1]$ by $\widehat{\beta}$. 
We will generally refer to links in $\cal{A} \times [0,1]$ as \textit{annular links}.

Complementary to Alexander's Theorem is the following result, 
which shows that link invariants can be constructed by considering braid invariants satisfying certain relations.

\begin{theorem}[Markov's Theorem]\label{thm:Markov}
Given two braids $\beta_1,\beta_2$, we have that $\cal{L}_{\beta_1} \sim \cal{L}_{\beta_2}$ if and only if 
$\beta_1$ and $\beta_2$ are related by a sequence of the following ``Markov moves'':
\begin{enumerate}
	\item[I.] Conjugation: $\beta \leftrightarrow \gamma \beta \gamma^{-1}$
	\item[II.] Stabilization: $\beta \leftrightarrow b_n^{\pm1} \beta$ for $\beta \in B_n$. 
\end{enumerate}
\end{theorem}

All of our invariants will manifestly be invariant under conjugation, 
as they will intrinsically be invariants of braid closures in the thickened annulus,
but only some will be invariant under the second Markov move.
Those that are not invariant under this move are still topologically relevant, 
as indicated by the following version of Markov's Theorem for braid closures in the thickened annulus. 
A proof can be found in \cite{MortonBraid}; see also \cite{KasselTuraev} for a thorough exposition.

\begin{theorem}\label{thm:MarkovAnnulus}
Given two braids $\beta,\beta'$, we have that $\widehat{\beta} \sim \widehat{\beta'}$ if and only if 
$\beta$ and $\beta'$ are conjugate braids.
\end{theorem}

More generally, we consider \emph{colored braids}, that is, braids in which each strand carries a labeling by a non-negative integer. 
Such braids no longer form a group, but rather a groupoid, in which objects correspond to the sequences of positive integers 
labeling the boundaries of the colored braids. 
We call a colored braid \emph{balanced} provided it is an endomorphism in the groupoid of colored braids. 
We can take the (annular) closure of a balanced, colored braid to obtain a colored (annular) link.
The Alexander and (annular) Markov theorems extend immediately to this setting.

\subsection{Graded linear categories}\label{sec:Graded}

We will make significant use of the concept of graded linear category. 
The standard definition of such a category is one enriched in $\Z$-graded vector spaces.
Given such a category $\mathbf{C}$, we can introduce formal grading 
shifts $q^k X$ of objects $X$ in $\mathbf{C}$, and restrict to morphisms $q^kX \xrightarrow{\alpha} q^lY$ that satisfy:
\[
\deg(\alpha) + l - k = 0.
\]
It is this latter notion that we call a \textit{graded linear category}. 

\begin{remark}
As a convention, throughout the paper we will denote shift functors for various gradings
as above, via powers of formal variables. For example, a shift of a graded vector space $V$ is denoted by 
$q^kV$, where $(q^k V)_t := V_{t-k}$.
Similarly, a shift on a bi-graded vector space $W$ is denoted $q^k a^l W$, 
a shift in homological degree on a chain complex is denoted by $h^k$, 
and a shift in super-degree on a super vector space is denoted by $\Pi$.
\end{remark}

Returning to our discussion, given a graded linear category $\mathbf{C}$, we can recover a linear category enriched in $\Z$-graded vector spaces, 
which we denote $\mathbf{C}^*$, by identifying all shifts of a given object, and taking
\[
\Hom_{\mathbf{C}^*}(X,Y) := \HOM_\mathbf{C}(X,Y) := \bigoplus_{t \in \Z} \Hom_{\mathbf{C}}(q^tX,Y)
\]
where the right-hand side gives the decomposition as a graded vector 
space\footnote{With this convention, we have that $\HOM_{\mathbf{C}}(X,q^kY) \cong q^k \Hom_{\mathbf{C}^*}(X,Y)$ as graded vector spaces.}.  
This shows that both notions are equivalent, but we will find it necessary to distinguish between them for various constructions.
Further, at times we will deal with bi-graded and/or super versions of these notions as well.

A linear 2-category can be viewed as a category enriched in linear categories. 
It follows that we immediately obtain analogous notions for 2-categories, 
\ie we can consider graded linear 2-categories.
Again, we formally distinguish these from 2-categories with $\Hom$-categories 
enriched in $\Z$-graded vector spaces.

Throughout, we also tacitly assume that all linear categories are additive unless otherwise specified. 
That is, we will assume that we have a notion of direct sum. 
If a linear category is not additive, we can pass to the additive closure by introducing finite formal direct sums of objects and by working with matrices of morphisms. 
We do the same for 2-categories enriched in linear categories.

%
\section{Annular foams and link invariants}\label{sec:Foams}
%

In this section we give a uniform description of all type A link homologies, 
using annular foam categories. We will relegate some proofs to subsequent sections, 
as they rely on the relation between foam categories and categorified quantum groups, 
which we review in Section \ref{sec:CatUqTrace}.

%
\subsection{The annular foam category}
%

We begin by quickly recalling the $2$-category of $\sln$ foams; 
see \cite{QR} for full details, and \cite{MSV} for an earlier incarnation of $\sln$ foams.
Objects in $\nFoam$ are finite sequences $\mathbf{k} = [k_1,\ldots,k_m]$ with $k_i \in \{1,\ldots,n\}$, 
and 1-morphisms are graded formal direct sums of ``left-directed'' $\sln$ webs, trivalent graphs with edges labeled by 
elements in $\{1,\ldots,n\}$, with vertices locally modeled by
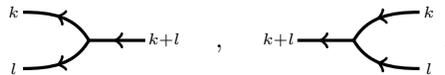
\[
\xy
(0,0)*{
\begin{tikzpicture}[scale=.5]
	\draw [very thick, directed=.55] (2.25,0) to (.75,0);
	\draw [very thick,directed=.55] (.75,0) to [out=120,in=0] (-1,.75);
	\draw [very thick,directed=.55] (.75,0) to [out=240,in=0] (-1,-.75);
	\node at (2.75,0) {\tiny $k{+}l$};
	\node at (-1.25,.75) {\tiny $k$};
	\node at (-1.25,-.75) {\tiny $l$};
\end{tikzpicture}
};
\endxy
\quad , \quad
\xy
(0,0)*{
\begin{tikzpicture}[scale=.5]
	\draw [very thick, rdirected=.55] (-2.25,0) to (-.75,0);
	\draw [very thick,rdirected=.55] (-.75,0) to [out=60,in=180] (1,.75);
	\draw [very thick,rdirected=.55] (-.75,0) to [out=300,in=180] (1,-.75);
	\node at (-2.75,0) {\tiny $k{+}l$};
	\node at (1.25,.75) {\tiny $k$};
	\node at (1.25,-.75) {\tiny $l$};
\end{tikzpicture}
};
\endxy .
\]
We use powers of the formal variable $q$ to denote shifts in grading.
The 2-morphisms in $\nFoam$ are matrices of linear combinations of degree-zero $\sln$ foams, 
certain singular cobordisms between $\sln$ webs, with decorated facets, locally modeled by the pieces:
\begin{equation}\label{eq:foamgens}
\begin{gathered}
\xy
(0,0)*{
\begin{tikzpicture} [scale=.6,fill opacity=0.2]
	\path[fill=blue] (2.25,3) to (.75,3) to (.75,0) to (2.25,0);
	\path[fill=red] (.75,3) to [out=225,in=0] (-.5,2.5) to (-.5,-.5) to [out=0,in=225] (.75,0);
	\path[fill=red] (.75,3) to [out=135,in=0] (-1,3.5) to (-1,.5) to [out=0,in=135] (.75,0);	
	\draw [very thick,directed=.55] (2.25,0) to (.75,0);
	\draw [very thick,directed=.55] (.75,0) to [out=135,in=0] (-1,.5);
	\draw [very thick,directed=.55] (.75,0) to [out=225,in=0] (-.5,-.5);
	\draw[very thick, red, directed=.55] (.75,0) to (.75,3);
	\draw [very thick] (2.25,3) to (2.25,0);
	\draw [very thick] (-1,3.5) to (-1,.5);
	\draw [very thick] (-.5,2.5) to (-.5,-.5);
	\draw [very thick,directed=.55] (2.25,3) to (.75,3);
	\draw [very thick,directed=.55] (.75,3) to [out=135,in=0] (-1,3.5);
	\draw [very thick,directed=.55] (.75,3) to [out=225,in=0] (-.5,2.5);
	\node [blue, opacity=1]  at (1.5,2.5) {\tiny{$_{k+l}$}};
	\node[red, opacity=1] at (-.75,3.25) {\tiny{$l$}};
	\node[red, opacity=1] at (-.25,2.25) {\tiny{$k$}};		
\end{tikzpicture}
};
\endxy
\;\; , \;\;
\xy
(0,0)*{
\begin{tikzpicture} [scale=.6,fill opacity=0.2]
	\path[fill=blue] (-2.25,3) to (-.75,3) to (-.75,0) to (-2.25,0);
	\path[fill=red] (-.75,3) to [out=45,in=180] (.5,3.5) to (.5,.5) to [out=180,in=45] (-.75,0);
	\path[fill=red] (-.75,3) to [out=315,in=180] (1,2.5) to (1,-.5) to [out=180,in=315] (-.75,0);	
	\draw [very thick,rdirected=.55] (-2.25,0) to (-.75,0);
	\draw [very thick,rdirected=.55] (-.75,0) to [out=315,in=180] (1,-.5);
	\draw [very thick,rdirected=.55] (-.75,0) to [out=45,in=180] (.5,.5);
	\draw[very thick, red, rdirected=.55] (-.75,0) to (-.75,3);
	\draw [very thick] (-2.25,3) to (-2.25,0);
	\draw [very thick] (1,2.5) to (1,-.5);
	\draw [very thick] (.5,3.5) to (.5,.5);
	\draw [very thick,rdirected=.55] (-2.25,3) to (-.75,3);
	\draw [very thick,rdirected=.55] (-.75,3) to [out=315,in=180] (1,2.5);
	\draw [very thick,rdirected=.55] (-.75,3) to [out=45,in=180] (.5,3.5);
	\node [blue, opacity=1]  at (-1.5,2.5) {\tiny{$_{k+l}$}};
	\node[red, opacity=1] at (.25,3.25) {\tiny{$l$}};
	\node[red, opacity=1] at (.75,2.25) {\tiny{$k$}};		
\end{tikzpicture}
};
\endxy
\;\; , \;\;
\xy
(0,0)*{
\begin{tikzpicture} [scale=.6,fill opacity=0.2]
	\path[fill=red] (-2.5,4) to [out=0,in=135] (-.75,3.5) to [out=270,in=90] (.75,.25)
		to [out=135,in=0] (-2.5,1);
	\path[fill=blue] (-.75,3.5) to [out=270,in=125] (.29,1.5) to [out=55,in=270] (.75,2.75) 
		to [out=135,in=0] (-.75,3.5);
	\path[fill=blue] (-.75,-.5) to [out=90,in=235] (.29,1.5) to [out=315,in=90] (.75,.25) 
		to [out=225,in=0] (-.75,-.5);
	\path[fill=red] (-2,3) to [out=0,in=225] (-.75,3.5) to [out=270,in=125] (.29,1.5)
		to [out=235,in=90] (-.75,-.5) to [out=135,in=0] (-2,0);
	\path[fill=red] (-1.5,2) to [out=0,in=225] (.75,2.75) to [out=270,in=90] (-.75,-.5)
		to [out=225,in=0] (-1.5,-1);
	\path[fill=red] (2,3) to [out=180,in=0] (.75,2.75) to [out=270,in=55] (.29,1.5)
		to [out=305,in=90] (.75,.25) to [out=0,in=180] (2,0);
	\draw[very thick, directed=.55] (2,0) to [out=180,in=0] (.75,.25);
	\draw[very thick, directed=.55] (.75,.25) to [out=225,in=0] (-.75,-.5);
	\draw[very thick, directed=.55] (.75,.25) to [out=135,in=0] (-2.5,1);
	\draw[very thick, directed=.55] (-.75,-.5) to [out=135,in=0] (-2,0);
	\draw[very thick, directed=.55] (-.75,-.5) to [out=225,in=0] (-1.5,-1);
	\draw[very thick, red, rdirected=.85] (-.75,3.5) to [out=270,in=90] (.75,.25);
	\draw[very thick, red, rdirected=.75] (.75,2.75) to [out=270,in=90] (-.75,-.5);	
	\draw[very thick] (-1.5,-1) -- (-1.5,2);	
	\draw[very thick] (-2,0) -- (-2,3);
	\draw[very thick] (-2.5,1) -- (-2.5,4);	
	\draw[very thick] (2,3) -- (2,0);
	\draw[very thick, directed=.55] (2,3) to [out=180,in=0] (.75,2.75);
	\draw[very thick, directed=.55] (.75,2.75) to [out=135,in=0] (-.75,3.5);
	\draw[very thick, directed=.65] (.75,2.75) to [out=225,in=0] (-1.5,2);
	\draw[very thick, directed=.55]  (-.75,3.5) to [out=225,in=0] (-2,3);
	\draw[very thick, directed=.55]  (-.75,3.5) to [out=135,in=0] (-2.5,4);
	\node[red, opacity=1] at (-2.25,3.375) {\tiny$p$};
	\node[red, opacity=1] at (-1.75,2.75) {\tiny$l$};	
	\node[red, opacity=1] at (-1.25,1.75) {\tiny$k$};
	\node[blue, opacity=1] at (0,2.75) {\tiny$_{l+p}$};
	\node[blue, opacity=1] at (0,.25) {\tiny$_{k+l}$};
	\node[red, opacity=1] at (1.37,2.5) {\tiny$_{k+l+p}$};	
\end{tikzpicture}
};
\endxy
\;\; , \;\;
\xy
(0,0)*{
\begin{tikzpicture} [scale=.6,fill opacity=0.2]
	\path[fill=red] (-2.5,4) to [out=0,in=135] (.75,3.25) to [out=270,in=90] (-.75,.5)
		 to [out=135,in=0] (-2.5,1);
	\path[fill=blue] (-.75,2.5) to [out=270,in=125] (-.35,1.5) to [out=45,in=270] (.75,3.25) 
		to [out=225,in=0] (-.75,2.5);
	\path[fill=blue] (-.75,.5) to [out=90,in=235] (-.35,1.5) to [out=315,in=90] (.75,-.25) 
		to [out=135,in=0] (-.75,.5);	
	\path[fill=red] (-2,3) to [out=0,in=135] (-.75,2.5) to [out=270,in=125] (-.35,1.5) 
		to [out=235,in=90] (-.75,.5) to [out=225,in=0] (-2,0);
	\path[fill=red] (-1.5,2) to [out=0,in=225] (-.75,2.5) to [out=270,in=90] (.75,-.25)
		to [out=225,in=0] (-1.5,-1);
	\path[fill=red] (2,3) to [out=180,in=0] (.75,3.25) to [out=270,in=45] (-.35,1.5) 
		to [out=315,in=90] (.75,-.25) to [out=0,in=180] (2,0);				
	\draw[very thick, directed=.55] (2,0) to [out=180,in=0] (.75,-.25);
	\draw[very thick, directed=.55] (.75,-.25) to [out=135,in=0] (-.75,.5);
	\draw[very thick, directed=.55] (.75,-.25) to [out=225,in=0] (-1.5,-1);
	\draw[very thick, directed=.45]  (-.75,.5) to [out=225,in=0] (-2,0);
	\draw[very thick, directed=.35]  (-.75,.5) to [out=135,in=0] (-2.5,1);	
	\draw[very thick, red, rdirected=.75] (-.75,2.5) to [out=270,in=90] (.75,-.25);
	\draw[very thick, red, rdirected=.85] (.75,3.25) to [out=270,in=90] (-.75,.5);
	\draw[very thick] (-1.5,-1) -- (-1.5,2);	
	\draw[very thick] (-2,0) -- (-2,3);
	\draw[very thick] (-2.5,1) -- (-2.5,4);	
	\draw[very thick] (2,3) -- (2,0);
	\draw[very thick, directed=.55] (2,3) to [out=180,in=0] (.75,3.25);
	\draw[very thick, directed=.55] (.75,3.25) to [out=225,in=0] (-.75,2.5);
	\draw[very thick, directed=.55] (.75,3.25) to [out=135,in=0] (-2.5,4);
	\draw[very thick, directed=.55] (-.75,2.5) to [out=135,in=0] (-2,3);
	\draw[very thick, directed=.55] (-.75,2.5) to [out=225,in=0] (-1.5,2);
	\node[red, opacity=1] at (-2.25,3.75) {\tiny$p$};
	\node[red, opacity=1] at (-1.75,2.75) {\tiny$l$};	
	\node[red, opacity=1] at (-1.25,1.75) {\tiny$k$};
	\node[blue, opacity=1] at (-.125,2.25) {\tiny$_{k+l}$};
	\node[blue, opacity=1] at (-.125,.75) {\tiny$_{l+p}$};
	\node[red, opacity=1] at (1.35,2.75) {\tiny$_{k+l+p}$};	
\end{tikzpicture}
};
\endxy \\
\xy
(0,0)*{
\begin{tikzpicture} [scale=.5,fill opacity=0.2]
	\path [fill=red] (4.25,-.5) to (4.25,2) to [out=165,in=15] (-.5,2) to (-.5,-.5) to 
		[out=0,in=225] (.75,0) to [out=90,in=180] (1.625,1.25) to [out=0,in=90] 
			(2.5,0) to [out=315,in=180] (4.25,-.5);
	\path [fill=red] (3.75,.5) to (3.75,3) to [out=195,in=345] (-1,3) to (-1,.5) to 
		[out=0,in=135] (.75,0) to [out=90,in=180] (1.625,1.25) to [out=0,in=90] 
			(2.5,0) to [out=45,in=180] (3.75,.5);
	\path[fill=blue] (.75,0) to [out=90,in=180] (1.625,1.25) to [out=0,in=90] (2.5,0);
	\draw [very thick,directed=.55] (2.5,0) to (.75,0);
	\draw [very thick,directed=.55] (.75,0) to [out=135,in=0] (-1,.5);
	\draw [very thick,directed=.55] (.75,0) to [out=225,in=0] (-.5,-.5);
	\draw [very thick,directed=.55] (3.75,.5) to [out=180,in=45] (2.5,0);
	\draw [very thick,directed=.55] (4.25,-.5) to [out=180,in=315] (2.5,0);
	\draw [very thick, red, directed=.75] (.75,0) to [out=90,in=180] (1.625,1.25);
	\draw [very thick, red] (1.625,1.25) to [out=0,in=90] (2.5,0);
	\draw [very thick] (3.75,3) to (3.75,.5);
	\draw [very thick] (4.25,2) to (4.25,-.5);
	\draw [very thick] (-1,3) to (-1,.5);
	\draw [very thick] (-.5,2) to (-.5,-.5);
	\draw [very thick,directed=.55] (4.25,2) to [out=165,in=15] (-.5,2);
	\draw [very thick, directed=.55] (3.75,3) to [out=195,in=345] (-1,3);
	\node [blue, opacity=1]  at (1.625,.5) {\tiny{$_{k+l}$}};
	\node[red, opacity=1] at (3.5,2.65) {\tiny{$l$}};
	\node[red, opacity=1] at (4,1.85) {\tiny{$k$}};		
\end{tikzpicture}
};
\endxy
\;\; , \;\;
\xy
(0,0)*{
\begin{tikzpicture} [scale=.5,fill opacity=0.2]
	\path [fill=red] (4.25,2) to (4.25,-.5) to [out=165,in=15] (-.5,-.5) to (-.5,2) to
		[out=0,in=225] (.75,2.5) to [out=270,in=180] (1.625,1.25) to [out=0,in=270] 
			(2.5,2.5) to [out=315,in=180] (4.25,2);
	\path [fill=red] (3.75,3) to (3.75,.5) to [out=195,in=345] (-1,.5) to (-1,3) to [out=0,in=135]
		(.75,2.5) to [out=270,in=180] (1.625,1.25) to [out=0,in=270] 
			(2.5,2.5) to [out=45,in=180] (3.75,3);
	\path[fill=blue] (2.5,2.5) to [out=270,in=0] (1.625,1.25) to [out=180,in=270] (.75,2.5);
	\draw [very thick,directed=.55] (4.25,-.5) to [out=165,in=15] (-.5,-.5);
	\draw [very thick, directed=.55] (3.75,.5) to [out=195,in=345] (-1,.5);
	\draw [very thick, red, directed=.75] (2.5,2.5) to [out=270,in=0] (1.625,1.25);
	\draw [very thick, red] (1.625,1.25) to [out=180,in=270] (.75,2.5);
	\draw [very thick] (3.75,3) to (3.75,.5);
	\draw [very thick] (4.25,2) to (4.25,-.5);
	\draw [very thick] (-1,3) to (-1,.5);
	\draw [very thick] (-.5,2) to (-.5,-.5);
	\draw [very thick,directed=.55] (2.5,2.5) to (.75,2.5);
	\draw [very thick,directed=.55] (.75,2.5) to [out=135,in=0] (-1,3);
	\draw [very thick,directed=.55] (.75,2.5) to [out=225,in=0] (-.5,2);
	\draw [very thick,directed=.55] (3.75,3) to [out=180,in=45] (2.5,2.5);
	\draw [very thick,directed=.55] (4.25,2) to [out=180,in=315] (2.5,2.5);
	\node [blue, opacity=1]  at (1.625,2) {\tiny{$_{k+l}$}};
	\node[red, opacity=1] at (3.5,2.65) {\tiny{$l$}};
	\node[red, opacity=1] at (4,1.85) {\tiny{$k$}};		
\end{tikzpicture}
};
\endxy
\;\; , \;\;
\xy
(0,0)*{
\begin{tikzpicture} [scale=.6,fill opacity=0.2]
	\path[fill=blue] (-.75,4) to [out=270,in=180] (0,2.5) to [out=0,in=270] (.75,4) .. controls (.5,4.5) and (-.5,4.5) .. (-.75,4);
	\path[fill=red] (-.75,4) to [out=270,in=180] (0,2.5) to [out=0,in=270] (.75,4) -- (2,4) -- (2,1) -- (-2,1) -- (-2,4) -- (-.75,4);
	\path[fill=blue] (-.75,4) to [out=270,in=180] (0,2.5) to [out=0,in=270] (.75,4) .. controls (.5,3.5) and (-.5,3.5) .. (-.75,4);
	\draw[very thick, directed=.55] (2,1) -- (-2,1);
	\path (.75,1) .. controls (.5,.5) and (-.5,.5) .. (-.75,1); 
	\draw [very thick, red, directed=.65] (-.75,4) to [out=270,in=180] (0,2.5) to [out=0,in=270] (.75,4);
	\draw[very thick] (2,4) -- (2,1);
	\draw[very thick] (-2,4) -- (-2,1);
	\draw[very thick,directed=.55] (2,4) -- (.75,4);
	\draw[very thick,directed=.55] (-.75,4) -- (-2,4);
	\draw[very thick,directed=.55] (.75,4) .. controls (.5,3.5) and (-.5,3.5) .. (-.75,4);
	\draw[very thick,directed=.55] (.75,4) .. controls (.5,4.5) and (-.5,4.5) .. (-.75,4);
	\node [red, opacity=1]  at (1.5,3.5) {\tiny{$_{k+l}$}};
	\node[blue, opacity=1] at (.25,3.4) {\tiny{$k$}};
	\node[blue, opacity=1] at (-.25,4.1) {\tiny{$l$}};	
\end{tikzpicture}
};
\endxy
\;\; , \;\;
\xy
(0,0)*{
\begin{tikzpicture} [scale=.6,fill opacity=0.2]
	\path[fill=blue] (-.75,-4) to [out=90,in=180] (0,-2.5) to [out=0,in=90] (.75,-4) .. controls (.5,-4.5) and (-.5,-4.5) .. (-.75,-4);
	\path[fill=red] (-.75,-4) to [out=90,in=180] (0,-2.5) to [out=0,in=90] (.75,-4) -- (2,-4) -- (2,-1) -- (-2,-1) -- (-2,-4) -- (-.75,-4);
	\path[fill=blue] (-.75,-4) to [out=90,in=180] (0,-2.5) to [out=0,in=90] (.75,-4) .. controls (.5,-3.5) and (-.5,-3.5) .. (-.75,-4);
	\draw[very thick, directed=.55] (2,-1) -- (-2,-1);
	\path (.75,-1) .. controls (.5,-.5) and (-.5,-.5) .. (-.75,-1); 
	\draw [very thick, red, directed=.65] (.75,-4) to [out=90,in=0] (0,-2.5) to [out=180,in=90] (-.75,-4);
	\draw[very thick] (2,-4) -- (2,-1);
	\draw[very thick] (-2,-4) -- (-2,-1);
	\draw[very thick,directed=.55] (2,-4) -- (.75,-4);
	\draw[very thick,directed=.55] (-.75,-4) -- (-2,-4);
	\draw[very thick,directed=.55] (.75,-4) .. controls (.5,-3.5) and (-.5,-3.5) .. (-.75,-4);
	\draw[very thick,directed=.55] (.75,-4) .. controls (.5,-4.5) and (-.5,-4.5) .. (-.75,-4);
	\node [red, opacity=1]  at (1.25,-1.25) {\tiny{$_{k+l}$}};
	\node[blue, opacity=1] at (-.25,-3.4) {\tiny{$l$}};
	\node[blue, opacity=1] at (.25,-4.1) {\tiny{$k$}};
\end{tikzpicture}
};
\endxy
\end{gathered}
\end{equation}
modulo local relations and isotopy. 
Facets that are $k$-labeled carry decorations by elements in $\C[t_1,\ldots,t_k]^{\mathfrak{S}_k}$.
Moreover, the foam relations imply that the endomorphism algebra of a $k$-labeled edge in $\nFoam$ is isomorphic to 
$\mathrm{H}^\bullet(Gr(k,n))$, so we can view decorations on facets as living in this quotient of $\C[t_1,\ldots,t_k]^{\mathfrak{S}_k}$.
We denote decorations on $1$-labeled facets by labeled dots $\bullet^r$, 
which correspond to the elements $t^r \in \C[t]/t^n \cong \mathrm{H}^\bullet(\C\mathbb{P}^{n-1})$.
This 2-category admits a tensor product (\ie it is a 3-category), 
given by placing webs/foams in front or back of each other.

We will be particularly interested in a limiting version of $\nFoam$, denoted $\vFoam$, 
where we allow objects and web labels to take values in $\N$, and where we remove the relation:
\begin{equation}\label{eq:ndot}
\xy
(0,0)*{
\begin{tikzpicture} [fill opacity=0.2, scale=.5]
	\draw [very thick, fill=red] (1,1) -- (-1,2) -- (-1,-1) -- (1,-2) -- cycle;
	\node [opacity=1] at (0,0) {$\bullet^{n}$};
	\node [red, opacity=1] at (-.5,1.25){\tiny$1$};
\end{tikzpicture}}; 
\endxy
= 0,
\end{equation}
which is the only foam relation explicitly dependent on $n$. 
It follows that $\nFoam$ is equivalent to the 2-category obtained from $\vFoam$ by taking the quotient by the 
ideal generated by equation \eqref{eq:ndot}.
Hence, we can study aspects of $\sln$ link homology uniformly by considering $\vFoam$. 
We will refer to the 1-morphisms in this 2-category as $\slv$ webs.

As in \cite{QR2}, we consider an annular version of $\vFoam$, where webs are embedded in the annulus $\cal{A}=S^1 \times [0,1]$, 
and foams between them lie in $\cal{A} \times [0,1]$.
\begin{definition}
Let $\vAFoam$ be the category whose 
objects are graded formal direct sums of annular webs -- closed, clockwise-oriented $\slv$ webs embedded in $\cal{A}$.
Morphisms are matrices of linear combinations of degree zero annular foams in $\cal{A}\times [0,1]$ mapping between such objects.
\end{definition}
To be precise, annular foams are those that can be locally modeled by the pieces in equation \eqref{eq:foamgens}, 
and whose generic horizontal slices are annular webs. 
This category is monoidal, with tensor product given by glueing the outer boundary of an annulus to the inner boundary of another. 
Note that every annular web can be obtained by taking the ``annular closure'' of a (non-annular) $\slv$ web:
\begin{equation}\label{eq:ExAnWeb}
\begin{tikzpicture}[anchorbase]
\draw [gray] (0,0) rectangle (2.5,1.5);
\draw[very thick, directed=.6] (.5,1) to [out=135,in=0] (0,1.25);
\draw[very thick, directed=.6] (.5,1) to [out=225,in=0] (0,.75);
\draw[very thick, directed=.55] (1,1) to (.5,1);
\draw[very thick, directed=.55] (1.5,.5) to [out=135,in=315] (1,1);
\draw[very thick, directed=.55] (2.5,1.25) to [out=180,in=45] (1,1);
\draw[very thick, directed=.55] (1.5,.5) to [out=225,in=0] (0,.25);
\draw[very thick, directed=.55] (2,.5) to (1.5,.5);
\draw[very thick, directed=.55] (2.5,.75) to [out=180,in=45] (2,.5);
\draw[very thick, directed=.55] (2.5,.25) to [out=180,in=315] (2,.5);
\end{tikzpicture}
\quad \leadsto \quad
\begin{tikzpicture}[anchorbase,scale=.75]
\draw [gray]  (1.25,1.75) ellipse (.4375 and .25);
\draw [gray] (1.25,1.75) ellipse (3.5 and 2);
\draw[very thick] (.5,1) to [out=135,in=0] (0,1.25);
\draw[very thick] (.5,1) to [out=225,in=0] (0,.75);
\draw[very thick, directed=.55] (1,1) to (.5,1);
\draw[very thick, directed=.55] (1.5,.5) to [out=135,in=315] (1,1);
\draw[very thick] (2.5,1.25) to [out=180,in=45] (1,1);
\draw[very thick] (1.5,.5) to [out=225,in=0] (0,.25);
\draw[very thick, directed=.55] (2,.5) to (1.5,.5);
\draw[very thick] (2.5,.75) to [out=180,in=45] (2,.5);
\draw[very thick] (2.5,.25) to [out=180,in=315] (2,.5);
\draw [very thick, directed=.55] (0,1.25) arc (270:90:.5) -- (2.5,2.25) arc (90:-90:.5);
\draw [very thick, directed=.55] (0,.75) arc (270:90:1) -- (2.5,2.75) arc (90:-90:1);
\draw [very thick, directed=.55] (0,.25) arc (270:90:1.5) -- (2.5,3.25) arc (90:-90:1.5);
\node at (.75,1.375) {\tiny$2$};
\node at (1.75,.125) {\tiny$2$};
\end{tikzpicture}
\end{equation}
but this is not true in general for annular foams.

In the next section, we will see how to assign a complex in $\vAFoam$ to any braid closure.
Up to homotopy, this complex will be an isotopy invariant of the corresponding link in $\cal{A} \times [0,1]$. 
We hence turn our attention to the bounded homotopy category of complexes $\Hot^b(\vAFoam)$. 
A key observation, which essentially appears in \cite{QR2}, is that the category $\Hot^b(\vAFoam)$ is equivalent to its subcategory consisting of 
complexes of webs and foams that are $S^1$-equivariant (after forgetting decorations).

\begin{theorem} \label{thm:foamObjDec}
Any complex in $\Hot^b(\vAFoam)$ is homotopy equivalent to a complex in which the only webs that appear are labeled, concentric circles. 
Every morphism between such webs can be expressed as a composition of tensor products of 
foams of the form:
\begin{equation}\label{eq:FfoamEfoam}
\mathsf{F}_r = 
\begin{tikzpicture}[anchorbase, fill opacity=0.2]
\path[fill=red] (.25,.5) to (.25,0) arc (180:360:1.25 and 0.25) to (2.75,.5) arc (360:180:1.25 and 0.25);
\path[fill=red] (.25,.5) to (.25,0) arc (180:0:1.25 and 0.25) to (2.75,.5) arc (0:180:1.25 and 0.25);
\path[fill=blue] (.75,1) to (.25,.5) arc (180:360:1.25 and 0.25) to (2.25,1) arc (360:180:.75 and 0.15);
\path[fill=blue] (.75,1) to (.25,.5) arc (180:0:1.25 and 0.25) to (2.25,1) arc (0:180:.75 and 0.15);
\path[fill=red] (.75,1) to [out=315,in=90] (1,0) arc (180:360:0.5 and 0.1) to [out=90,in=225] (2.25,1) arc (360:180:0.75 and 0.15);
\path[fill=red] (.75,1) to [out=315,in=90] (1,0) arc (180:0:0.5 and 0.1) to [out=90,in=225] (2.25,1) arc (0:180:0.75 and 0.15);
\path[fill=red] (.75,1.5) to (.75,1) arc (180:360:0.75 and 0.15) to (2.25,1.5) arc (360:180:0.75 and 0.15);
\path[fill=red] (.75,1.5) to (.75,1) arc (180:0:0.75 and 0.15) to (2.25,1.5) arc (0:180:0.75 and 0.15);
\path[fill=red] (0,1.5) to [out=270,in=135] (.25,.5) arc (180:360:1.25 and 0.25) to [out=45,in=270] (3,1.5) arc (360:180:1.5 and .3);
\path[fill=red] (0,1.5) to [out=270,in=135] (.25,.5) arc (180:0:1.25 and 0.25) to [out=45,in=270] (3,1.5) arc (0:180:1.5 and .3);
\draw[very thick, red] (2.75,.5) arc (0:180:1.25 and 0.25);
\draw[very thick, red] (2.25,1) arc (0:180:.75 and 0.15);
\draw[very thick] (.25,0) to (.25,.5);
\draw[very thick] (.25,.5) to [out=45,in=225] (.75,1);
\draw[very thick] (.75,1) to (.75,1.5);
\draw[very thick] (.25,.5) to [out=135,in=270] (0,1.5);
\draw[very thick] (1,0) to [out=90,in=315] (.75,1);
\draw[very thick] (2.75,0) to (2.75,.5);
\draw[very thick] (2.75,.5) to [out=135,in=315] (2.25,1);
\draw[very thick] (2.25,1) to (2.25,1.5);
\draw[very thick] (2.75,.5) to [out=45,in=270] (3,1.5);
\draw[very thick] (2,0) to [out=90,in=225] (2.25,1);
\draw [very thick]  (1.5,0) ellipse (.5 and .1);
\draw [very thick]  (1.5,0) ellipse (1.25 and .25);
\draw[very thick, red] (.25,.5) arc (180:360:1.25 and 0.25);
\draw[very thick, red] (.75,1) arc (180:360:.75 and 0.15);
\draw [very thick]  (1.5,1.5) ellipse (.75 and .15);
\draw [very thick]  (1.5,1.5) ellipse (1.5 and .3);
\node[opacity=1] at (.75,.625) {$\bullet^r$};
\end{tikzpicture}
\quad \text{and} \quad
\mathsf{E}_r = 
\begin{tikzpicture}[anchorbase, fill opacity=0.2,rotate=180]
\path[fill=red] (.25,.5) to (.25,0) arc (180:360:1.25 and 0.25) to (2.75,.5) arc (360:180:1.25 and 0.25);
\path[fill=red] (.25,.5) to (.25,0) arc (180:0:1.25 and 0.25) to (2.75,.5) arc (0:180:1.25 and 0.25);
\path[fill=blue] (.75,1) to (.25,.5) arc (180:360:1.25 and 0.25) to (2.25,1) arc (360:180:.75 and 0.15);
\path[fill=blue] (.75,1) to (.25,.5) arc (180:0:1.25 and 0.25) to (2.25,1) arc (0:180:.75 and 0.15);
\path[fill=red] (.75,1) to [out=315,in=90] (1,0) arc (180:360:0.5 and 0.1) to [out=90,in=225] (2.25,1) arc (360:180:0.75 and 0.15);
\path[fill=red] (.75,1) to [out=315,in=90] (1,0) arc (180:0:0.5 and 0.1) to [out=90,in=225] (2.25,1) arc (0:180:0.75 and 0.15);
\path[fill=red] (.75,1.5) to (.75,1) arc (180:360:0.75 and 0.15) to (2.25,1.5) arc (360:180:0.75 and 0.15);
\path[fill=red] (.75,1.5) to (.75,1) arc (180:0:0.75 and 0.15) to (2.25,1.5) arc (0:180:0.75 and 0.15);
\path[fill=red] (0,1.5) to [out=270,in=135] (.25,.5) arc (180:360:1.25 and 0.25) to [out=45,in=270] (3,1.5) arc (360:180:1.5 and .3);
\path[fill=red] (0,1.5) to [out=270,in=135] (.25,.5) arc (180:0:1.25 and 0.25) to [out=45,in=270] (3,1.5) arc (0:180:1.5 and .3);
\draw[very thick, red] (.25,.5) arc (180:360:1.25 and 0.25);
\draw[very thick, red] (.75,1) arc (180:360:.75 and 0.15);
\draw[very thick] (.25,0) to (.25,.5);
\draw[very thick] (.25,.5) to [out=45,in=225] (.75,1);
\draw[very thick] (.75,1) to (.75,1.5);
\draw[very thick] (.25,.5) to [out=135,in=270] (0,1.5);
\draw[very thick] (1,0) to [out=90,in=315] (.75,1);
\draw[very thick] (2.75,0) to (2.75,.5);
\draw[very thick] (2.75,.5) to [out=135,in=315] (2.25,1);
\draw[very thick] (2.25,1) to (2.25,1.5);
\draw[very thick] (2.75,.5) to [out=45,in=270] (3,1.5);
\draw[very thick] (2,0) to [out=90,in=225] (2.25,1);
\draw [very thick]  (1.5,0) ellipse (.5 and .1);
\draw [very thick]  (1.5,0) ellipse (1.25 and .25);
\draw[very thick, red] (2.75,.5) arc (0:180:1.25 and 0.25);
\draw[very thick, red] (2.25,1) arc (0:180:.75 and 0.15);
\draw [very thick]  (1.5,1.5) ellipse (.75 and .15);
\draw [very thick]  (1.5,1.5) ellipse (1.5 and .3);
\node[opacity=1] at (2.25,.5) {$\bullet^r$};
\end{tikzpicture}
\end{equation}
where in these annular foams the middle (blue) facets are $1$-labeled.
\end{theorem}
\begin{proof}
The proof proceeds in two steps: first, we show that every annular web is homotopy equivalent to a complex (which is concentrated only homological degrees zero and one) 
in which the only webs appearing are concentric circles. 
This then immediately implies that every complex in $\Hot^b(\vAFoam)$ is isomorphic to such a complex. 
The second step, showing that all foams between such webs take the desired form, follows from 
equation \eqref{eq:MainDiag} and \cite[Theorem 6.1]{BHLW}. 
Since an entirely foam-based proof of this step is not more-enlightening, we will only show the first step. 
Note that this is essentially the statement of \cite[Proposition 5.1]{QR2}, which is stated in terms of categorical traces; 
however, we give an updated and more-streamlined proof here.

We can assume that all annular webs are in ``ladder form,'' meaning that they are the closure of compositions of tensor products of 
the basic pieces:
\[
\begin{tikzpicture}[anchorbase]
\draw [gray] (0,0) rectangle (2.5,1.5);
\draw[very thick, directed=.2,directed=.8] (2.5,.375) to (0,.375);
\draw[very thick, directed=.2,directed=.8] (2.5,1.125) to (0,1.125);
\draw[very thick, directed=.55] (1.5,1.125) to (1,.375);
\node at (1.5,.75) {\small$r$};
\node at (2.75,.375) {\small$k$};
\node at (2.75,1.125) {\small$l$};
\node at (-.375,.375) {\small$k{+}r$};
\node at (-.375,1.125) {\small$l{-}r$};
\end{tikzpicture}
\quad \text{and} \quad
\begin{tikzpicture}[anchorbase]
\draw [gray] (0,0) rectangle (2.5,1.5);
\draw[very thick, directed=.2,directed=.8] (2.5,.375) to (0,.375);
\draw[very thick, directed=.2,directed=.8] (2.5,1.125) to (0,1.125);
\draw[very thick, directed=.55] (1.5,.375) to (1,1.125);
\node at (1,.75) {\small$r$};
\node at (2.75,.375) {\small$k$};
\node at (2.75,1.125) {\small$l$};
\node at (-.375,.375) {\small$k{-}r$};
\node at (-.375,1.125) {\small$l{+}r$};
\end{tikzpicture}
\]
with identity webs. We refer to the $r$-labeled edges above as the ``rungs,'' and the other edges as the ``uprights'' 
(think of a ladder, which has been laid on its side). Note that it may be the case that one or more of the uprights in such a web 
does not actually appear (\ie appears with label zero).

A radial slice through such an annular web $\cal{W}$ that does not intersect any rungs determines a sequence of numbers $\vec{k}=[k_1,\ldots,k_m]$, 
by considering the labels of the uprights it intersects, ordered from the outermost to the innermost. 
In fact, the web itself is determined (up to isotopy) by a cyclic sequence $\mathbf{k}_\cal{W}$ of such number sequences. 
We can assume moreover that each sequence appearing in $\mathbf{k}_\cal{W}$ has the same length, 
by inserting zeros into sequences shorter than the longest one. 
For a given web $\cal{W}$, the sum $\mathsf{s}_\cal{W} = \sum_{i=1}^m k_i$ is equal for each $\vec{k}$ in $\mathbf{k}_\cal{W}$. 

We can pass from the cyclic sequence $\mathbf{k}_\cal{W}$ to an associated cyclic sequence $\boldsymbol{\lambda}_\cal{W}$, 
whose entries are the sequences obtained from those $\vec{k}=[k_1,\ldots,k_m]$ in $\mathbf{k}_\cal{W}$ by letting 
$\vec{\lambda}=(\lambda_1,\ldots,\lambda_{m-1})$, where $\lambda_i = k_i - k_{i+1}$. 
Up to isotopy, the web $\cal{W}$ is determined by the sequence $\boldsymbol{\lambda}_\cal{W}$, 
together with the integer $\mathsf{s}_\cal{W}$. 
The entries of $\boldsymbol{\lambda}_\cal{W}$ are integral $\slm$ weights, \ie elements in $\Z^{m-1}$, 
and, as such, one (or more) entry is minimal with respect to the standard (partial) order on the weight lattice. 

We now show that $\cal{W}$ is homotopy equivalent to a complex, concentrated in homological degrees zero and one, 
in which all webs $\cal{W}'$ have corresponding cyclic sequences $\boldsymbol{\lambda}_{\cal{W}'}$ 
with either strictly larger minimal weight, or in which the minimal weight appears fewer times.
Moreover, the number of weights appearing in each $\boldsymbol{\lambda}_{\cal{W}'}$ is less than or equal to the number in $\boldsymbol{\lambda}_\cal{W}$ 
and $\mathsf{s}_{\cal{W}'} = \mathsf{s}_\cal{W}$. 

Consider the portion of $\cal{W}$ near the radial slice corresponding to a minimal $\vec{\lambda}$ in $\boldsymbol{\lambda}_\cal{W}$. 
We necessarily have that this portion of $\cal{W}$ takes the form:
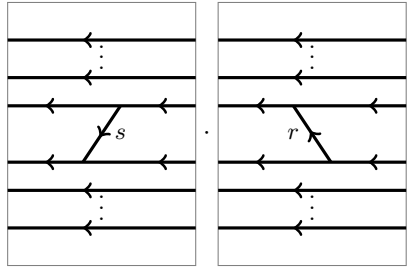
\[
\begin{tikzpicture}[anchorbase]
\draw [gray] (0,-1) rectangle (2.5,2.5);
\draw[very thick, directed=.2,directed=.8] (2.5,.375) to (0,.375);
\draw[very thick, directed=.2,directed=.8] (2.5,1.125) to (0,1.125);
\draw[very thick, directed=.55] (1.5,1.125) to (1,.375);
\draw[very thick, directed =.6] (2.5,1.5) to (0,1.5);
\draw[very thick, directed=.6] (2.5,2) to (0,2);
\draw[very thick, directed=.6] (2.5,0) to (0,0);
\draw[very thick, directed=.6] (2.5,-.5) to (0,-.5);
\node at (1.25,1.875) {\small$\vdots$};
\node at (1.25,-.125) {\small$\vdots$};
\node at (1.5,.75) {\small$s$};
\end{tikzpicture}
\cdot
\begin{tikzpicture}[anchorbase]
\draw [gray] (0,-1) rectangle (2.5,2.5);
\draw[very thick, directed=.2,directed=.8] (2.5,.375) to (0,.375);
\draw[very thick, directed=.2,directed=.8] (2.5,1.125) to (0,1.125);
\draw[very thick, directed=.55] (1.5,.375) to (1,1.125);
\draw[very thick, directed =.6] (2.5,1.5) to (0,1.5);
\draw[very thick, directed=.6] (2.5,2) to (0,2);
\draw[very thick, directed=.6] (2.5,0) to (0,0);
\draw[very thick, directed=.6] (2.5,-.5) to (0,-.5);
\node at (1.25,1.875) {\small$\vdots$};
\node at (1.25,-.125) {\small$\vdots$};
\node at (1,.75) {\small$r$};
\end{tikzpicture}
\]
where the slice corresponding to $\vec{\lambda}$ is in the middle, 
and the rungs to the right and left of this slice may or may not lie between the same pair of uprights. 
If these rungs rungs do not lie between the same pair of uprights, then $\cal{W}$ is isomorphic to the web where 
these rungs appear in opposite order, either via a web isotopy (if the pairs of uprights the rungs pass between are disjoint), 
or using the web isomorphisms:
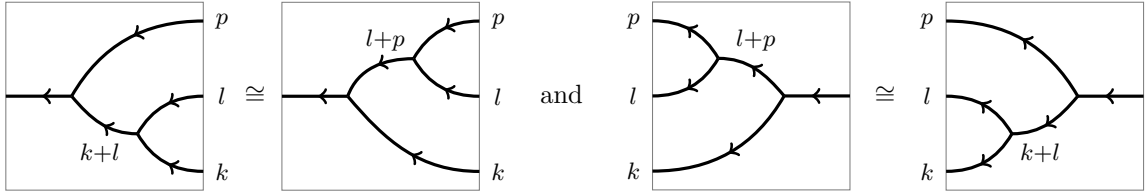
\[
\xy,
(0,0)*{
\begin{tikzpicture}[scale=.5,rotate=90]
	\draw [gray] (-1.5,-1) rectangle (3.5,4.25);
	\draw [very thick, directed=.45] (0,.75) to [out=90,in=220] (1,2.5);
	\draw [very thick, directed=.45] (1,-1) to [out=90,in=330] (0,.75);
	\draw [very thick, directed=.45] (-1,-1) to [out=90,in=210] (0,.75);
	\draw [very thick, directed=.45] (3,-1) to [out=90,in=330] (1,2.5);
	\draw [very thick, directed=.45] (1,2.5) to (1,4.25);
	\node at (-1,-1.5) {\small $k$};
	\node at (1,-1.5) {\small $l$};
	\node at (-.5,1.75) {\small $k{+}l$};
	\node at (3,-1.5) {\small $p$};
\end{tikzpicture}
};
\endxy
\cong \;
\xy
(0,0)*{
\begin{tikzpicture}[scale=.5,rotate=90]
	\draw [gray] (-3.5,-1) rectangle (1.5,4.25);
	\draw [very thick, directed=.45] (0,.75) to [out=90,in=340] (-1,2.5);
	\draw [very thick, directed=.45] (-1,-1) to [out=90,in=210] (0,.75);
	\draw [very thick, directed=.45] (1,-1) to [out=90,in=330] (0,.75);
	\draw [very thick, directed=.45] (-3,-1) to [out=90,in=220] (-1,2.5);
	\draw [very thick, directed=.45] (-1,2.5) to (-1,4.25);
	\node at (1,-1.5) {\small $p$};
	\node at (-1,-1.5) {\small $l$};
	\node at (.5,1.5) {\small $l{+}p$};
	\node at (-3,-1.5) {\small $k$};
\end{tikzpicture}
};
\endxy
\quad\text{and}\quad
\xy
(0,0)*{\rotatebox{180}{
\begin{tikzpicture}[scale=.5,rotate=90]
	\draw [gray] (-1.5,-1) rectangle (3.5,4.25);
	\draw [very thick, rdirected=.55] (0,.75) to [out=90,in=220] (1,2.5);
	\draw [very thick, rdirected=.55] (1,-1) to [out=90,in=330] (0,.75);
	\draw [very thick, rdirected=.55] (-1,-1) to [out=90,in=210] (0,.75);
	\draw [very thick, rdirected=.55] (3,-1) to [out=90,in=330] (1,2.5);
	\draw [very thick, rdirected=.55] (1,2.5) to (1,4.25);
	\node at (-1,-1.5) {\rotatebox{180}{\small $p$}};
	\node at (1,-1.5) {\rotatebox{180}{\small $l$}};
	\node at (-.5,1.75) {\rotatebox{180}{\small $l{+}p$}};
	\node at (3,-1.5) {\rotatebox{180}{\small $k$}};
\end{tikzpicture}
}};
\endxy
\; \cong
\xy
(0,0)*{\reflectbox{\rotatebox{180}{
\begin{tikzpicture}[scale=.5,rotate=270]
	\draw [gray] (-1.5,-1) rectangle (3.5,4.25);
	\draw [very thick, rdirected=.55] (0,.75) to [out=90,in=220] (1,2.5);
	\draw [very thick, rdirected=.55] (1,-1) to [out=90,in=330] (0,.75);
	\draw [very thick, rdirected=.55] (-1,-1) to [out=90,in=210] (0,.75);
	\draw [very thick, rdirected=.55] (3,-1) to [out=90,in=330] (1,2.5);
	\draw [very thick, rdirected=.55] (1,2.5) to (1,4.25);
	\node at (-1,-1.5) {\reflectbox{\rotatebox{180}{\small $k$}}};
	\node at (1,-1.5) {\reflectbox{\rotatebox{180}{\small $l$}}};
	\node at (-.5,1.5) {\reflectbox{\rotatebox{180}{\small $k{+}l$}}};
	\node at (3,-1.5) {\reflectbox{\rotatebox{180}{\small $p$}}};
\end{tikzpicture}
}}};
\endxy
\]
if the rungs share exactly one upright. In either case, the new web $\cal{W}'$ obtained has corresponding $\boldsymbol{\lambda}_{\cal{W}'}$ with either higher minimal weight, 
or has its minimal weight appearing one fewer time than in $\boldsymbol{\lambda}_\cal{W}$. 

If the rungs lie between the same uprights, then locally $\cal{W}$ takes the form
\begin{equation}\label{eq:LocalFormW}
\begin{tikzpicture}[anchorbase]
\draw [gray] (0,0) rectangle (2.5,1.5);
\draw[very thick, directed=.15,directed=.5,directed=.9] (2.5,.375) to (0,.375);
\draw[very thick, directed=.2,directed=.5,directed=.8] (2.5,1.125) to (0,1.125);
\draw[very thick, directed=.55] (1,1.125) to (.5,.375);
\draw[very thick, directed=.55] (2,.375) to (1.5,1.125);
\node at (.5,.75) {\small$s$};
\node at (2,.75) {\small$r$};
\node at (2.75,.375) {\small$k$};
\node at (2.75,1.125) {\small$l$};
\node at (-.625,.375) {\small$k{+}s{-}r$};
\node at (-.625,1.125) {\small$l{-}s{+}r$};
\end{tikzpicture}
\end{equation}
If $k - l \geq r - s$, then we can use the relation
\[
\begin{tikzpicture}[anchorbase]
\draw [gray] (0,0) rectangle (2.5,1.5);
\draw[very thick, directed=.15,directed=.5,directed=.9] (2.5,.375) to (0,.375);
\draw[very thick, directed=.2,directed=.5,directed=.8] (2.5,1.125) to (0,1.125);
\draw[very thick, directed=.55] (1,1.125) to (.5,.375);
\draw[very thick, directed=.55] (2,.375) to (1.5,1.125);
\node at (.5,.75) {\small$s$};
\node at (2,.75) {\small$r$};
\node at (2.75,.375) {\small$k$};
\node at (2.75,1.125) {\small$l$};
\node at (-.625,.375) {\small$k{+}s{-}r$};
\node at (-.625,1.125) {\small$l{-}s{+}r$};
\end{tikzpicture}
\cong 
\bigoplus_{j=0}^{\min(r,s)}
\bigoplus_{{s-r+k-l \brack j}}
\begin{tikzpicture}[anchorbase]
\draw [gray] (0,0) rectangle (2.5,1.5);
\draw[very thick, directed=.2,directed=.5,directed=.8] (2.5,.375) to (0,.375);
\draw[very thick, directed=.15,directed=.5,directed=.9] (2.5,1.125) to (0,1.125);
\draw[very thick, directed=.55] (2,1.125) to (1.5,.375);
\draw[very thick, directed=.55] (1,.375) to (.5,1.125);
\node at (.375,.75) {\scriptsize$r{-}j$};
\node at (2.125,.75) {\scriptsize$s{-}j$};
\node at (2.75,.375) {\small$k$};
\node at (2.75,1.125) {\small$l$};
\node at (-.625,.375) {\small$k{+}s{-}r$};
\node at (-.625,1.125) {\small$l{-}s{+}r$};
\end{tikzpicture}
\]
to see that $\cal{W}$ is isomorphic to a direct sum of webs $\cal{W}'$,  
each of which has corresponding $\boldsymbol{\lambda}_{\cal{W}'}$ with either higher minimal weight, 
or has its minimal weight appearing one fewer time than in $\boldsymbol{\lambda}_\cal{W}$.
If $k - l < r - s$, we instead have the isomorphism:
\[
\begin{tikzpicture}[anchorbase]
\draw [gray] (0,0) rectangle (2.5,1.5);
\draw[very thick, directed=.2,directed=.5,directed=.8] (2.5,.375) to (0,.375);
\draw[very thick, directed=.15,directed=.5,directed=.9] (2.5,1.125) to (0,1.125);
\draw[very thick, directed=.55] (2,1.125) to (1.5,.375);
\draw[very thick, directed=.55] (1,.375) to (.5,1.125);
\node at (.5,.75) {\small$r$};
\node at (2,.75) {\small$s$};
\node at (2.75,.375) {\small$k$};
\node at (2.75,1.125) {\small$l$};
\node at (-.625,.375) {\small$k{+}s{-}r$};
\node at (-.625,1.125) {\small$l{-}s{+}r$};
\end{tikzpicture}
\cong
\bigoplus_{j=0}^{\min(r,s)}
\bigoplus_{{r-s+l-k \brack j}}
\begin{tikzpicture}[anchorbase]
\draw [gray] (0,0) rectangle (2.5,1.5);
\draw[very thick, directed=.15,directed=.5,directed=.9] (2.5,.375) to (0,.375);
\draw[very thick, directed=.2,directed=.5,directed=.8] (2.5,1.125) to (0,1.125);
\draw[very thick, directed=.55] (1,1.125) to (.5,.375);
\draw[very thick, directed=.55] (2,.375) to (1.5,1.125);
\node at (.375,.75) {\scriptsize$s{-}j$};
\node at (2.125,.75) {\scriptsize$r{-}j$};
\node at (2.75,.375) {\small$k$};
\node at (2.75,1.125) {\small$l$};
\node at (-.625,.375) {\small$k{+}s{-}r$};
\node at (-.625,1.125) {\small$l{-}s{+}r$};
\end{tikzpicture}
\]
where the web in equation \eqref{eq:LocalFormW} appears exactly once in the direct sum. 
This gives an isomorphism 
\[
\cal{W}' \xrightarrow{\left(\begin{smallmatrix} \phi_{\cal{S}} \\ \phi_{\cal{W}} \end{smallmatrix}\right)} \cal{S} \oplus \cal{W}
\xrightarrow{\left(\begin{smallmatrix} \iota_{\cal{S}} & \iota_{\cal{W}} \end{smallmatrix}\right)} \cal{W}'
\] 
where $\cal{S}$ consists of a direct sum of webs. Note that for $\cal{W}'$, as well as for all of the webs appearing in $\cal{S}$, 
the minimal weight appearing in their corresponding $\boldsymbol{\lambda}$'s is either larger than that for $\cal{W}$, 
or appears fewer times. We can use this isomorphism to obtain the desired homotopy equivalence:
\begin{equation}\label{eq:RGI1}
\cal{W} \simeq \left(\uwave{\cal{S} \oplus \cal{W}} \xrightarrow{\left(\begin{smallmatrix} \id_{\cal{S}} & 0 \end{smallmatrix}\right)} \cal{S} \right) 
\cong \left(\uwave{\cal{W}'} \xrightarrow{\phi_{\cal{S}}} \cal{S} \right)
\end{equation}
Iterating this procedure produces a homotopy equivalence between $\cal{W}$ and a complex concentrated in homological degrees zero and one, 
and in which all webs are labeled, concentric circles. To do so, we (simultaneously) induct on the minimal weight appearing in the cyclic sequences 
associated with webs in the complex, and the number of times the minimal weight appears in the complex.
Note that there is a maximal possible weight for a given (initial) value of $\mathsf{s}_\cal{W}$, 
and the only webs corresponding to cyclic sequences of length one are concentric circles.

For the induction step, we replace a web containing the minimal weight with webs of higher minimal weight, 
or with webs having fewer occurrences of the minimal weight,  as in equation \eqref{eq:RGI1}. 
To do so, we view our complex $\uwave{\cal{A}} \xrightarrow{d} \cal{B}$ as the cone of the chain map
\begin{equation}\label{eq:conemap}
\left(\uwave{0} \to \cal{A}\right) \xrightarrow{-d} \left(\uwave{0} \to \cal{B}\right).
\end{equation}
If the web we wish to replace is a summand of $\cal{A}$, then we use \eqref{eq:RGI1} to replace the domain of 
\eqref{eq:conemap} with a complex concentrated in homological degrees one and two, and having the desired properties. 
If instead the web we wish to replace is a summand of $\cal{B}$, 
then we use the following variant of \eqref{eq:RGI1}:
 \begin{equation}\label{eq:RGI2}
\cal{W} \simeq \left( \cal{S} \xrightarrow{\left(\begin{smallmatrix} \id_{\cal{S}} \\ 0 \end{smallmatrix}\right)} \uwave{\cal{S} \oplus \cal{W}}  \right) 
\cong \left(\cal{S} \xrightarrow{\iota_{\cal{S}}} \uwave{\cal{W}'}  \right)
\end{equation}
to replace the codomain of \eqref{eq:conemap} with a complex, concentrated in homological degrees zero and one, 
that has the desired properties. Replacing the (co)domain of a chain map with a homotopy equivalent complex yields 
a cone homotopy equivalent to the cone of the original map, as desired.

Hence any web (viewed as a complex concentrated in homological degree zero) is homotopy equivalent to a complex in which 
all webs appearing are concentric, labeled circles. Taking direct sums shows this is true for any complex concentrated in homological degree zero. 
Since any bounded complex is homotopy equivalent to an iterated cone of such complexes, the same is true for any complex in $\Hot^b(\vAFoam)$
\end{proof}

\begin{example} Consider the annular web from equation \eqref{eq:ExAnWeb}, we have:
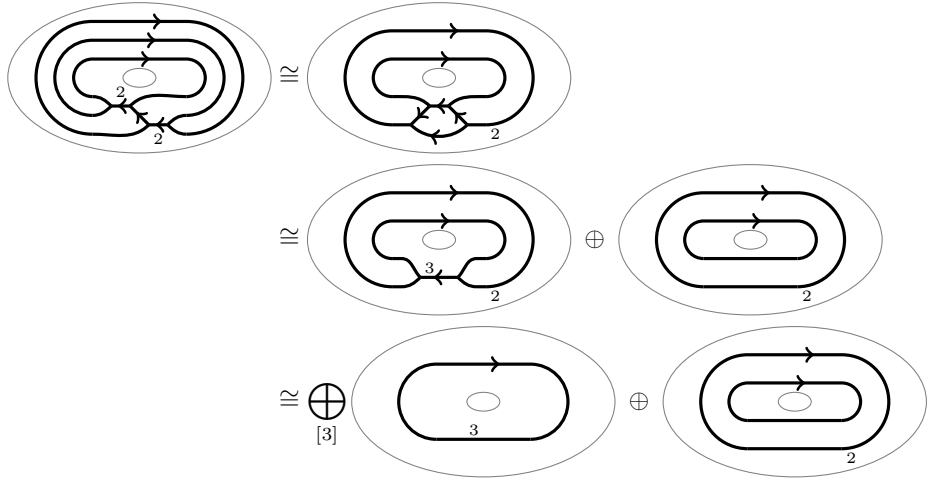
\[
\begin{aligned}
\begin{tikzpicture}[anchorbase,scale=.5]
\draw [gray]  (1.25,1.75) ellipse (.4375 and .25);
\draw [gray] (1.25,1.75) ellipse (3.5 and 2);
\draw[very thick] (.5,1) to [out=135,in=0] (0,1.25);
\draw[very thick] (.5,1) to [out=225,in=0] (0,.75);
\draw[very thick, directed=.65] (1,1) to (.5,1);
\draw[very thick, directed=.65] (1.5,.5) to [out=135,in=315] (1,1);
\draw[very thick] (2.5,1.25) to [out=180,in=45] (1,1);
\draw[very thick] (1.5,.5) to [out=225,in=0] (0,.25);
\draw[very thick, directed=.65] (2,.5) to (1.5,.5);
\draw[very thick] (2.5,.75) to [out=180,in=45] (2,.5);
\draw[very thick] (2.5,.25) to [out=180,in=315] (2,.5);
\draw [very thick, directed=.55] (0,1.25) arc (270:90:.5) -- (2.5,2.25) arc (90:-90:.5);
\draw [very thick, directed=.55] (0,.75) arc (270:90:1) -- (2.5,2.75) arc (90:-90:1);
\draw [very thick, directed=.55] (0,.25) arc (270:90:1.5) -- (2.5,3.25) arc (90:-90:1.5);
\node at (.75,1.375) {\tiny$2$};
\node at (1.75,.125) {\tiny$2$};
\end{tikzpicture}
&\cong
\begin{tikzpicture}[anchorbase,scale=.5]
\draw [gray]  (1.25,1.75) ellipse (.4375 and .25);
\draw [gray] (1.25,1.75) ellipse (3.5 and 2);
\draw[very thick] (2.5,1.25) to [out=180,in=45] (1.5,1);
\draw[very thick] (2.5,.5) to (2,.5);
\draw[very thick, directed=.65] (2,.5) to [out=135,in=315] (1.5,1);
\draw[very thick, directed=.65] (1.5,1) to (1,1);
\draw[very thick] (1,1) to [out=135,in=0] (0,1.25);
\draw[very thick, directed=.65] (1,1) to [out=225,in=45] (.5,.5);
\draw[very thick] (.5,.5) to (0,.5);
\draw[very thick, directed=.65] (2,.5) to [out=225,in=315] (.5,.5);
\draw [very thick, directed=.55] (0,1.25) arc (270:90:.5) -- (2.5,2.25) arc (90:-90:.5);
\draw [very thick, directed=.55] (0,.5) arc (270:90:1.25) -- (2.5,3) arc (90:-90:1.25);
\node at (2.75,.25) {\tiny$2$};
\end{tikzpicture} \\
&\cong 
\begin{tikzpicture}[anchorbase,scale=.5]
\draw [gray]  (1.25,1.75) ellipse (.4375 and .25);
\draw [gray] (1.25,1.75) ellipse (3.5 and 2);
\draw[very thick] (2.5,1.25) to (2.25,1.25) to [out=180,in=45] (1.75,.75);
\draw[very thick, directed=.65] (1.75,.75) to (.75,.75);
\draw[very thick] (2.5,.5) to (2.25,.5) to [out=180,in=315] (1.75,.75);
\draw[very thick] (.75,.75) to [out=135,in=0] (.25,1.25) to (0,1.25);
\draw[very thick] (.75,.75) to [out=225,in=0] (.25,.5) to (0,.5);
\draw [very thick, directed=.55] (0,1.25) arc (270:90:.5) -- (2.5,2.25) arc (90:-90:.5);
\draw [very thick, directed=.55] (0,.5) arc (270:90:1.25) -- (2.5,3) arc (90:-90:1.25);
\node at (2.75,.25) {\tiny$2$};
\node at (1,1) {\tiny$3$};
\end{tikzpicture}
\; \oplus \;
\begin{tikzpicture}[anchorbase,scale=.5]
\draw [gray]  (1.25,1.75) ellipse (.4375 and .25);
\draw [gray] (1.25,1.75) ellipse (3.5 and 2);
\draw[very thick] (2.5,1.25) to (0,1.25);
\draw[very thick] (2.5,.5) to (0,.5);
\draw [very thick, directed=.55] (0,1.25) arc (270:90:.5) -- (2.5,2.25) arc (90:-90:.5);
\draw [very thick, directed=.55] (0,.5) arc (270:90:1.25) -- (2.5,3) arc (90:-90:1.25);
\node at (2.75,.25) {\tiny$2$};
\end{tikzpicture} \\
&\cong 
\bigoplus_{[3]}
\begin{tikzpicture}[anchorbase,scale=.5]
\draw [gray]  (1.25,1.75) ellipse (.4375 and .25);
\draw [gray] (1.25,1.75) ellipse (3.5 and 2);
\draw[very thick] (2.5,.75) to (0,.75);
\draw [very thick, directed=.55] (0,.75) arc (270:90:1) -- (2.5,2.75) arc (90:-90:1);
\node at (1,1) {\tiny$3$};
\end{tikzpicture}
\; \oplus \;
\begin{tikzpicture}[anchorbase,scale=.5]
\draw [gray]  (1.25,1.75) ellipse (.4375 and .25);
\draw [gray] (1.25,1.75) ellipse (3.5 and 2);
\draw[very thick] (2.5,1.25) to (0,1.25);
\draw[very thick] (2.5,.5) to (0,.5);
\draw [very thick, directed=.55] (0,1.25) arc (270:90:.5) -- (2.5,2.25) arc (90:-90:.5);
\draw [very thick, directed=.55] (0,.5) arc (270:90:1.25) -- (2.5,3) arc (90:-90:1.25);
\node at (2.75,.25) {\tiny$2$};
\end{tikzpicture}
\end{aligned}
\]
\end{example}
In this example, we found that passing to the homotopy category of complexes was unnecessary: 
this annular web is \textit{isomorphic} to a direct sum of labeled, concentric circles. 
This is equivalent to the statement that the annular web is equal to a $\N[q,q^{-1}]$-linear combination of 
such concentric circles at the decategorified level.
In fact, this is the case for all examples we've computed, and we suspect this behavior holds generally.

\begin{conjecture}
Any annular $\slv$ web is isomorphic to a direct sum of labeled, concentric circles. 
\end{conjecture}

%
\subsection{Link invariants from annular foams}
%

We will now see how to use Theorem \ref{thm:foamObjDec} to construct invariants of links in $\cal{A} \times [0,1]$ and $S^3$. 
The formulae:
\begin{equation}\label{eq:slnCrossing}
\bigg \llbracket
\xy
(0,0)*{
\begin{tikzpicture}[scale=.5]
	\draw[very thick, <-] (0,1) to [out=0,in=180] (2,0);
	\draw[very thick, <-] (0,0) to [out=0,in=225] (.9,.4);
	\draw[very thick] (1.1,.6) to [out=45,in=180] (2,1);
\end{tikzpicture}
};
\endxy 
\bigg \rrbracket
=
\left(
q
\xy
(0,0)*{
\begin{tikzpicture}[scale=.5]
	\draw[very thick, <-] (0,1) to [out=340,in=200] (2,1);
	\draw[very thick, <-] (0,0) to [out=20,in=160] (2,0);
\end{tikzpicture}
};
\endxy
\xrightarrow{
\xy
(0,0)*{
\begin{tikzpicture} [scale=.2,fill opacity=0.2]
	\path [fill=red] (4.25,2) to (4.25,-.5) to [out=165,in=15] (-.5,-.5) to (-.5,2) to
		[out=0,in=225] (.75,2.5) to [out=270,in=180] (1.625,1.25) to [out=0,in=270] 
			(2.5,2.5) to [out=315,in=180] (4.25,2);
	\path [fill=red] (3.75,3) to (3.75,.5) to [out=195,in=345] (-1,.5) to (-1,3) to [out=0,in=135]
		(.75,2.5) to [out=270,in=180] (1.625,1.25) to [out=0,in=270] 
			(2.5,2.5) to [out=45,in=180] (3.75,3);
	\path[fill=blue] (2.5,2.5) to [out=270,in=0] (1.625,1.25) to [out=180,in=270] (.75,2.5);
	\draw [very thick] (4.25,-.5) to [out=170,in=10] (-.5,-.5);
	\draw [very thick] (3.75,.5) to [out=190,in=350] (-1,.5);
	\draw [very thick, red] (2.5,2.5) to [out=270,in=0] (1.625,1.25);
	\draw [very thick, red] (1.625,1.25) to [out=180,in=270] (.75,2.5);
	\draw [very thick] (3.75,3) to (3.75,.5);
	\draw [very thick] (4.25,2) to (4.25,-.5);
	\draw [very thick] (-1,3) to (-1,.5);
	\draw [very thick] (-.5,2) to (-.5,-.5);
	\draw [very thick] (2.5,2.5) to (.75,2.5);
	\draw [very thick] (.75,2.5) to [out=135,in=0] (-1,3);
	\draw [very thick] (.75,2.5) to [out=225,in=0] (-.5,2);
	\draw [very thick] (3.75,3) to [out=180,in=45] (2.5,2.5);
	\draw [very thick] (4.25,2) to [out=180,in=315] (2.5,2.5);		
\end{tikzpicture}
};
\endxy
}
\uwave{
\xy
(0,0)*{
\begin{tikzpicture}[scale=.55]
	\draw[very thick, <-] (0,1) to [out=0,in=90] (.625,.5);
	\draw[very thick,->] (.625,.5) to [out=270,in=0] (0,0);
	\draw[very thick, directed=.65] (1.375,.5) to (.625,.5);
	\draw[very thick, directed=.45] (2,1) to [out=180,in=90] (1.375,.5); 
	\draw[very thick, rdirected=.65] (1.375,.5) to [out=270,in=180] (2,0);
	\node at (1,1) {\tiny $2$};
\end{tikzpicture}
};
\endxy
}
\right)
\quad , \quad
\bigg \llbracket
\xy
(0,0)*{
\begin{tikzpicture}[scale=.5]
	\draw[very thick, <-] (0,0) to [out=0,in=180] (2,1);
	\draw[very thick, <-] (0,1) to [out=0,in=135] (.9,.6);
	\draw[very thick] (1.1,.4) to [out=315,in=180] (2,0);
\end{tikzpicture}
};
\endxy 
\bigg \rrbracket
=
\left(
\uwave{
\xy
(0,0)*{
\begin{tikzpicture}[scale=.55]
	\draw[very thick, <-] (0,1) to [out=0,in=90] (.625,.5);
	\draw[very thick,->] (.625,.5) to [out=270,in=0] (0,0);
	\draw[very thick, directed=.65] (1.375,.5) to (.625,.5);
	\draw[very thick, directed=.45] (2,1) to [out=180,in=90] (1.375,.5); 
	\draw[very thick, rdirected=.65] (1.375,.5) to [out=270,in=180] (2,0);
	\node at (1,1) {\tiny $2$};
\end{tikzpicture}
};
\endxy
}
\xrightarrow{
\xy
(0,0)*{
\begin{tikzpicture} [scale=.2,fill opacity=0.2]
	\path [fill=red] (4.25,-.5) to (4.25,2) to [out=170,in=10] (-.5,2) to (-.5,-.5) to 
		[out=0,in=225] (.75,0) to [out=90,in=180] (1.625,1.25) to [out=0,in=90] 
			(2.5,0) to [out=315,in=180] (4.25,-.5);
	\path [fill=red] (3.75,.5) to (3.75,3) to [out=190,in=350] (-1,3) to (-1,.5) to 
		[out=0,in=135] (.75,0) to [out=90,in=180] (1.625,1.25) to [out=0,in=90] 
			(2.5,0) to [out=45,in=180] (3.75,.5);
	\path[fill=blue] (.75,0) to [out=90,in=180] (1.625,1.25) to [out=0,in=90] (2.5,0);
	\draw [very thick] (2.5,0) to (.75,0);
	\draw [very thick] (.75,0) to [out=135,in=0] (-1,.5);
	\draw [very thick] (.75,0) to [out=225,in=0] (-.5,-.5);
	\draw [very thick] (3.75,.5) to [out=180,in=45] (2.5,0);
	\draw [very thick] (4.25,-.5) to [out=180,in=315] (2.5,0);
	\draw [very thick, red] (.75,0) to [out=90,in=180] (1.625,1.25) to [out=0,in=90] (2.5,0);
	\draw [very thick] (3.75,3) to (3.75,.5);
	\draw [very thick] (4.25,2) to (4.25,-.5);
	\draw [very thick] (-1,3) to (-1,.5);
	\draw [very thick] (-.5,2) to (-.5,-.5);
	\draw [very thick] (4.25,2) to [out=170,in=10] (-.5,2);
	\draw [very thick] (3.75,3) to [out=190,in=350] (-1,3);	
\end{tikzpicture}
};
\endxy
}
q^{-1}
\xy
(0,0)*{
\begin{tikzpicture}[scale=.5]
	\draw[very thick, <-] (0,1) to [out=340,in=200] (2,1);
	\draw[very thick, <-] (0,0) to [out=20,in=160] (2,0);
\end{tikzpicture}
};
\endxy
\right)
\end{equation}
extend (using horizontal composition and tensor product) to assign a complex $\llbracket \beta \rrbracket$ in $\vFoam$ to any braid $\beta$. 
Moreover, up to homotopy equivalence, this complex is invariant under the braid relations, 
\ie the braid-like Reidemeister II and III moves.
Note that these complexes decategorify to the bracket in equation \eqref{eq:DecatSymCrossing}, 
and also decategorify to the bracket in equation \eqref{eq:DecatCrossing} if we first apply shifts $h^{\pm 1}$ in homological degree, 
and negate the $q$-grading.

More generally, the following formulae assign complexes in $\vFoam$ to braids whose strands are colored by elements in $\N$, 
in the case that $k \geq l$:
\begin{equation}\label{eq:slnColoredCrossing}
\begin{aligned}
\left\llbracket
\xy
(0,0)*{
\begin{tikzpicture}[scale=.5]
	\draw[very thick, <-] (0,1) to [out=0,in=180] (2,0);
	\draw[very thick, <-] (0,0) to [out=0,in=225] (.9,.4);
	\draw[very thick] (1.1,.6) to [out=45,in=180] (2,1);
	\node at (2.25,0) {\tiny$k$};
	\node at (2.25,1) {\tiny$l$};
\end{tikzpicture}
};
\endxy 
\right \rrbracket
&=
\left(
q^l
\xy
(0,0)*{
\begin{tikzpicture}[scale=.5]
	\draw[very thick, directed=.65] (3,0) to (2,0);
	\draw[very thick, directed=.65] (2,0) to [out=135,in=315] (1,1);
	\draw[very thick, ->] (2,0) to [out=225,in=0] (0,0);
	\draw[very thick, directed=.65] (3,1) to [out=180,in=45] (1,1);
	\draw[very thick, ->] (1,1) to (0,1);
	\node at (3.25,0) {\tiny$k$};
	\node at (3.25,1) {\tiny$l$};
\end{tikzpicture}
};
\endxy
\xrightarrow{\del^{-l}}
q^{l-1}
\xy
(0,0)*{
\begin{tikzpicture}[scale=.5]
	\draw[very thick, directed=.65] (3.5,1) to (2.875,1);
	\draw[very thick, directed=.65] (2.875,1) to [out=225,in=45] (2.125,0);
	\draw[very thick, directed=.65] (2.875,1) to [out=135,in=45] (.625,1);
	\draw[very thick, ->] (.625,1) to (0,1);
	\draw[very thick, directed=.65] (3.5,0) to [out=180,in=315] (2.125,0);
	\draw[very thick, directed=.65] (2.125,0) to (1.375,0);
	\draw[very thick, directed=.65] (1.375,0) to [out=135,in=315] (.625,1);
	\draw[very thick, ->] (1.375,0) to [out=225,in=0] (0,0);
	\node at (3.75,0) {\tiny$k$};
	\node at (3.75,1) {\tiny$l$};
	\node at (2.125,.5) {\tiny$1$};
\end{tikzpicture}
};
\endxy
\xrightarrow{\del^{1-l}}
q^{l-2}
\xy
(0,0)*{
\begin{tikzpicture}[scale=.5]
	\draw[very thick, directed=.65] (3.5,1) to (2.875,1);
	\draw[very thick, directed=.65] (2.875,1) to [out=225,in=45] (2.125,0);
	\draw[very thick, directed=.65] (2.875,1) to [out=135,in=45] (.625,1);
	\draw[very thick, ->] (.625,1) to (0,1);
	\draw[very thick, directed=.65] (3.5,0) to [out=180,in=315] (2.125,0);
	\draw[very thick, directed=.65] (2.125,0) to (1.375,0);
	\draw[very thick, directed=.65] (1.375,0) to [out=135,in=315] (.625,1);
	\draw[very thick, ->] (1.375,0) to [out=225,in=0] (0,0);
	\node at (3.75,0) {\tiny$k$};
	\node at (3.75,1) {\tiny$l$};
	\node at (2.125,.5) {\tiny$2$};
\end{tikzpicture}
};
\endxy
\xrightarrow{\del^{2-l}}
\cdots
\xrightarrow{\del^{-1}}
\uwave{
\xy
(0,0)*{
\begin{tikzpicture}[scale=.5]
	\draw[very thick, directed=.65] (3.5,1) to (2.875,1);
	\draw[very thick, directed=.65] (2.875,1) to [out=180,in=45] (2.125,0);
	\draw[very thick, ->] (.625,1) to (0,1);
	\draw[very thick, directed=.65] (3.5,0) to [out=180,in=315] (2.125,0);
	\draw[very thick, directed=.65] (2.125,0) to (1.375,0);
	\draw[very thick, directed=.65] (1.375,0) to [out=135,in=0] (.625,1);
	\draw[very thick, ->] (1.375,0) to [out=225,in=0] (0,0);
	\node at (3.75,0) {\tiny$k$};
	\node at (3.75,1) {\tiny$l$};
\end{tikzpicture}
};
\endxy
}
\right) \\
\left\llbracket
\xy
(0,0)*{
\begin{tikzpicture}[scale=.5]
	\draw[very thick, <-] (0,0) to [out=0,in=180] (2,1);
	\draw[very thick, <-] (0,1) to [out=0,in=135] (.9,.6);
	\draw[very thick] (1.1,.4) to [out=315,in=180] (2,0);
	\node at (2.25,0) {\tiny$l$};
	\node at (2.25,1) {\tiny$k$};
\end{tikzpicture}
};
\endxy 
\right \rrbracket
&=
\left(
\uwave{
\xy
(0,0)*{
\begin{tikzpicture}[scale=.5]
	\draw[very thick, directed=.65] (3.5,1) to (2.875,1);
	\draw[very thick, directed=.65] (2.875,1) to [out=180,in=45] (2.125,0);
	\draw[very thick, ->] (.625,1) to (0,1);
	\draw[very thick, directed=.65] (3.5,0) to [out=180,in=315] (2.125,0);
	\draw[very thick, directed=.65] (2.125,0) to (1.375,0);
	\draw[very thick, directed=.65] (1.375,0) to [out=135,in=0] (.625,1);
	\draw[very thick, ->] (1.375,0) to [out=225,in=0] (0,0);
	\node at (3.75,0) {\tiny$l$};
	\node at (3.75,1) {\tiny$k$};
\end{tikzpicture}
};
\endxy
}
\xrightarrow{\del^{0}}
\cdots
\xrightarrow{\del^{l-3}}
q^{2-l}
\xy
(0,0)*{
\begin{tikzpicture}[scale=.5]
	\draw[very thick, directed=.65] (3.5,1) to (2.875,1);
	\draw[very thick, directed=.65] (2.875,1) to [out=225,in=45] (2.125,0);
	\draw[very thick, directed=.65] (2.875,1) to [out=135,in=45] (.625,1);
	\draw[very thick, ->] (.625,1) to (0,1);
	\draw[very thick, directed=.65] (3.5,0) to [out=180,in=315] (2.125,0);
	\draw[very thick, directed=.65] (2.125,0) to (1.375,0);
	\draw[very thick, directed=.65] (1.375,0) to [out=135,in=315] (.625,1);
	\draw[very thick, ->] (1.375,0) to [out=225,in=0] (0,0);
	\node at (3.75,0) {\tiny$l$};
	\node at (3.75,1) {\tiny$k$};
	\node at (1.375,.5) {\tiny$2$};
\end{tikzpicture}
};
\endxy
\xrightarrow{\del^{l-2}}
q^{1-l}
\xy
(0,0)*{
\begin{tikzpicture}[scale=.5]
	\draw[very thick, directed=.65] (3.5,1) to (2.875,1);
	\draw[very thick, directed=.65] (2.875,1) to [out=225,in=45] (2.125,0);
	\draw[very thick, directed=.65] (2.875,1) to [out=135,in=45] (.625,1);
	\draw[very thick, ->] (.625,1) to (0,1);
	\draw[very thick, directed=.65] (3.5,0) to [out=180,in=315] (2.125,0);
	\draw[very thick, directed=.65] (2.125,0) to (1.375,0);
	\draw[very thick, directed=.65] (1.375,0) to [out=135,in=315] (.625,1);
	\draw[very thick, ->] (1.375,0) to [out=225,in=0] (0,0);
	\node at (3.75,0) {\tiny$l$};
	\node at (3.75,1) {\tiny$k$};
	\node at (1.375,.5) {\tiny$1$};
\end{tikzpicture}
};
\endxy
\xrightarrow{\del^{l-1}}
q^{-l}
\xy
(0,0)*{
\begin{tikzpicture}[scale=.5]
	\draw[very thick, directed=.65] (3,1) to (2,1);
	\draw[very thick, directed=.65] (2,1) to [out=225,in=45] (1,0);
	\draw[very thick, ->] (2,1) to [out=135,in=0] (0,1);
	\draw[very thick, directed=.65] (3,0) to [out=180,in=315] (1,0);
	\draw[very thick, ->] (1,0) to (0,0);
	\node at (3.25,0) {\tiny$l$};
	\node at (3.25,1) {\tiny$k$};
\end{tikzpicture}
};
\endxy
\right)
\end{aligned}
\end{equation}
where the differential is as given in \cite{QR,RW}. Similar formulae hold in the case that $l \leq k$.
Given a balanced, colored braid $\beta$, 
equations \eqref{eq:slnCrossing} and \eqref{eq:slnColoredCrossing} assign a complex 
$\llbracket \widehat{\beta} \rrbracket \in \Hot^b(\vAFoam)$ to its annular closure $\widehat{\beta}$.
By the results in \cite{QR,Cautis}, 
the complex is again invariant under (colored) braid relations.

\begin{theorem}\label{thm:UnivAnnularInv}
The complex $\llbracket \widehat{\beta} \rrbracket \in \Hot^b(\vAFoam)$ is an invariant of the 
colored annular link $\widehat{\beta}$.
\end{theorem}

\begin{proof}
By Theorem \ref{thm:MarkovAnnulus}, it suffices to check that $\llbracket \widehat{\beta} \rrbracket$ is invariant under conjugation. 
Given a colored braid $\gamma$ (which is composable with $\beta$), 
the annular link $\widehat{\gamma \beta \gamma^{-1}}$ is isotopic to $\beta$ via an isotopy 
that first slides $\gamma$ around the annulus 
(\ie slides the crossings through the part of the thickened annulus where the closure took place), 
and then is given by a sequence of Reidemeister II-like braid moves. 

As mentioned above, the latter induce homotopy equivalences, so it suffices to check the former. To do so, note that it 
suffices to check that sliding a single crossing around the annulus induces a homotopy equivalence.
This in turn follows since the foam that slides a portion of a web around the annulus is an isomorphism in $\vAFoam$, 
hence performing it on all of the webs in the complex corresponding to the crossing is a chain isomorphism. 
Thus, in particular, it is a homotopy equivalence.
\end{proof}

The invariant $\llbracket \widehat{\beta} \rrbracket$ contains so much information that it is unwieldy; 
we will later see that it contains all of the information about known type $A$ link homologies (and more!). 
Using Theorem \ref{thm:foamObjDec}, we have a means to obtain more-manageable invariants from it. 
Define $\vAFoam_\circ$ to be the subcategory of $\vAFoam$ where objects are given by essential, labeled, concentric circles, 
and where morphisms are generated by the foams in equation \eqref{eq:FfoamEfoam}. 
Given a balanced, colored braid, we let $\llbracket \widehat{\beta} \rrbracket_\circ$ 
be a complex in $\vAFoam_\circ$ that is homotopy equivalent to $\llbracket \widehat{\beta} \rrbracket$.
Save for the first statement, which we show in Section \ref{sec:Proof}, the following is immediate:

\begin{theorem}\label{thm:Meta}
The category $\vAFoam_\circ$ is a full monoidal subcategory of $\vAFoam$.
Any graded functor $\Upsilon: \vAFoam_\circ \to \mathbf{C}$ to a graded, additive category $\mathbf{C}$ determines an invariant 
$\Upsilon( \llbracket \widehat{\beta} \rrbracket_\circ ) \in \Hot^b(\mathbf{C})$ of the colored annular link $\widehat{\beta}$.
\end{theorem}

In the event that $\mathbf{C}$ is abelian, we can take homology to obtain further invariants.

%
\subsection{Link homology theories}
%

We now explore the invariants defined in the previous section in the case where $\mathbf{C} = \gr\C\Vect^{\Z/2}$, 
the category of graded, complex, super vector spaces.
This construction parallels the Bar-Natan \cite{BN2} approach to Khovanov homology, 
in which one first obtains an abstract invariant of links by considering the Khovanov complex in the homotopy category of complexes over Bar-Natan's cobordism category, 
and then passes to a more-manageable invariant by applying a TQFT-like functor to obtain a complex of vector spaces and takes homology.
In our case, $\llbracket \widehat{\beta} \rrbracket_\circ$ plays the role of the former, and $\Upsilon$ plays the role of the TQFT.
Our next result, 
which we prove in Section \ref{sec:Proof},
shows how to produce many functors $\Upsilon: \vAFoam_\circ \to \gr\C\Vect^{\Z/2}$.

\begin{proposition}\label{prop:Functors}
Given any graded, complex, super vector space $\cal{V}$ and a degree-$2$ endomorphism $T:\cal{V} \to \cal{V}$, 
we obtain a monoidal functor
\[
\Upsilon_{\cal{V},T}^{\cal{S}}: \vAFoam_\circ \to \gr\C\Vect^{\Z/2}
\]
that sends a $k$-labeled essential circle to $\Sym^k\cal{V}$. 
The functor $\Upsilon_{\cal{V},T}^{\cal{S}}$ sends the annular foams $\mathsf{F}_r$ and $\mathsf{E}_r$ 
from equation \eqref{eq:FfoamEfoam} to 
\[
\Sym^k\cal{V} \otimes \Sym^l\cal{V} \xrightarrow{\iota \otimes \id} \Sym^{k-1}\cal{V} \otimes \cal{V} \otimes \Sym^l\cal{V} 
\xrightarrow{\id \otimes T^r \otimes \id} \Sym^{k-1}\cal{V} \otimes \cal{V} \otimes \Sym^l\cal{V} 
\xrightarrow{\id \otimes \pi} \Sym^{k-1}\cal{V} \otimes \Sym^{l+1}\cal{V}
\]
and
\[
\Sym^k\cal{V} \otimes \Sym^l\cal{V} \xrightarrow{\id \otimes \iota} \Sym^{k}\cal{V} \otimes \cal{V} \otimes \Sym^{l-1}\cal{V} 
\xrightarrow{\id \otimes T^r \otimes \id} \Sym^k\cal{V} \otimes \cal{V} \otimes \Sym^{l-1}\cal{V} 
\xrightarrow{\pi \otimes \id} \Sym^{k+1}\cal{V} \otimes \Sym^{l-1}\cal{V}
\]
respectively, where by convention $T^0 = \id_\cal{V}$.
\end{proposition}

To specify the maps $\iota$ and $\pi$, 
we identify $\Sym^k \cal{V}$ with the subspace of $\cal{V}^{\otimes k}$ consisting of symmetric tensors.
The map $\iota$ is inclusion, while $\pi : \Sym^k\cal{V} \otimes \cal{V} \to \Sym^{k+1} \cal{V}$ and 
$ \cal{V} \otimes \Sym^k\cal{V} \to \Sym^{k+1} \cal{V}$ are both $k+1$ times 
the restriction of the projection map $\mathrm{pr}: \cal{V}^{\otimes (k+1)} \to \Sym^{k+1} \cal{V}$.
The latter map is explicitly given by 
\[
\mathrm{pr}(v_1 \otimes \cdots \otimes v_{k+1}) 
= \frac{1}{(k+1)!} \sum_{\sigma \in \mathfrak{S}_{k+1}} \sigma(v_1 \otimes \cdots \otimes v_{k+1})
\]

\begin{remark}\label{rem:FunctorsWedge}
In this proposition, we could have equivalently worked with skew-symmetric powers $\bV^k\cal{V}$ 
in place $\Sym^k \cal{V}$. 
However, given a super vector space $\cal{V}$ we have an isomorphism:
\[
\bV^k\cal{V} \cong \begin{cases} \Pi \Sym^k(\Pi \cal{V}) \quad k \text{ odd} \\ \Sym^k(\Pi \cal{V}) \quad k \text{ even}  \end{cases}
\]
(recall $\Pi$ denotes a shift in super-degree).
Thus, up to a degree shift, working with symmetric powers suffices. 
However, we will make explicit use of the functors corresponding to skew-symmetric tensors without this degree shift 
(\ie corresponding to replacing every occurrence of $\Sym$ in Proposition \ref{prop:Functors} with $\bV$), 
in comparing our construction to $\sln$ Khovanov-Rozansky homology.
We denote these functors by $\Upsilon_{\cal{V},T}^{\wedge}$.
\end{remark}

Combining Proposition \ref{prop:Functors} and Remark \ref{rem:FunctorsWedge} with Theorem \ref{thm:Meta} 
yields homological invariants of (colored) annular links that arise as braid closures. 
Moreover, some of these invariants actually give invariants of the corresponding links $\cal{L}_\beta \subset S^3$.

\begin{theorem}\label{thm:LinkHomologies}
Let $\beta$ be a balanced colored braid.
For each choice of $\cal{V}$ and $T$, the homology $\overline{\cal{H}}_{\cal{V},T}^\square( \widehat{\beta})$
of the complex $\Upsilon_{\cal{V},T}^\square(\llbracket \widehat{\beta}\rrbracket_\circ)$ is an invariant of the colored annular link $\widehat{\beta} \subset \cal{A} \times [0,1]$, 
for $\square = \wedge$ or $\cal{S}$. 
Moreover, we recover each known type A link homology theory by taking the following choices and rescalings:
\begin{itemize}
\item 
Let $\cal{V}_n := q^{1-n} \C[t]/t^n$, 
then\footnote{Here, and throughout, we denote by $t$ the endomorphism that is multiplication by this variable.}
$\cal{H}_{\cal{V}_n,t}^{\wedge}(\cal{L}_\beta) := h^{w_\beta} q^{W_\beta - w_\beta - nw_\beta} 
\overline{\cal{H}}_{\cal{V}_n,t}^{\wedge}( \widehat{\beta})$ 
is an invariant\footnote{The shifts involving $W_\beta - w_\beta$ are necessary for \textit{colored} Markov II invariance,
and vanish in the uncolored case. Compare to the crossing formulae in {\cite[Definition 12.16]{Wu}}, 
where our invariant corresponds to the homology of the mirror image of the given link.} 
of the colored link $\cal{L}_\beta \subset S^3$
and is isomorphic to colored $\sln$ Khovanov-Rozansky link homology $\KhR_n(\cal{L}_\beta)$. 
Here, $t$ is an even variable with $\deg_q(t) = 2$.

\item 
Let $ \cal{V}_\infty := qa^{-1} \C[t,\theta]/\theta^2$, 
then
$\cal{H}_{\cal{V}_\infty,t}^{\cal{S}}(\cal{L}_\beta) := a^{w_\beta} q^{W_\beta - w_\beta} \overline{\cal{H}}_{\cal{V}_\infty,t}^{\cal{S}}( \widehat{\beta})$ 
is an invariant of the colored link $\cal{L}_\beta \subset S^3$
and is isomorphic to the colored variant of Khovanov-Rozansky's triply-graded HOMFLYPT link homology $\HHH(\cal{L}_\beta)$.
Here, $t$ is an even variable with $\deg_{q,a}(t) = (2,0)$ and $\theta$ is an odd variable with $\deg_{q,a}(\theta) = (0,2)$.

\item 
Let $\C^n$ be the standard representation of $\sln$ with the weight-space grading, 
then $\cal{H}_{\C^n,0}^\wedge( \widehat{\beta}) :=  h^{w_\beta} \overline{\cal{H}}_{\C^n,0}^\wedge( \widehat{\beta})$ is 
colored, annular $\sln$ Khovanov-Rozansky homology $\cal{A}\KhR_n(\widehat{\beta})$. 
\end{itemize}
\end{theorem}

We prove the bullet-pointed claims in Section \ref{sec:identifying}, 
where we take a closer look at each of these invariants, as well as others.

Before proceeding, we comment on some of the invariants not explicitly identified in Theorem \ref{thm:LinkHomologies}. 
If we consider $\Upsilon_{\C^n,0}^{\cal{S}}$, 
the homology is a version of annular $\sln$ link homology based on the symmetric tensors, 
\ie in which an essential $k$-labeled circle is assigned the $\sln$ representation $\Sym^k(\C^n)$. 
Up to an overall super-degree shift, this can equivalently be obtained using $\Upsilon_{\cal{V},0}^\wedge$ with $\cal{V} = \C^{0|n}$.
We thus propose this as the definition of \textit{annular $\slnegn$ link homology}, 
since it is defined in exactly the same way as annular $\sln$ link homology, but with $\C^{0|n}$ replacing $\C^{n|0}$ in the final bullet-point in 
Theorem \ref{thm:LinkHomologies}

Much more interesting is the case where we consider $\Upsilon_{\cal{V}_n,t}^{\cal{S}}$. Here, as above, $\cal{V}_n = q^{1-n} \C[t]/t^n$. 
This produces an invariant that is a relative of $\sln$ link homology for the link $\cal{L}_\beta \subset S^3$, 
but, in contrast to Khovanov-Rozansky homology, 
is built from the symmetric powers of $\C^n$, \ie a $k$-colored unknot is assigned $\Sym^k(\cal{V}_n)$.
In this case, the invariant of a $k$-colored link is finite-dimensional, and its decategorification is the quantum $\sln$ Reshetikhin-Turaev 
invariant of the link with components colored by the symmetric power $\Sym^k\C_q^n$ of the standard representation $\C_q^n$ of quantum $\sln$. 

Given that $\Sym^1\C_q^n \cong \C_q^n \cong \bV^1\C_q^n$, the following result is surprising:
\begin{theorem}\label{thm:negativen}
The homology of $q^{n w_\beta + W_\beta - w_\beta}\Upsilon_{\cal{V}_n,t}^{\cal{S}}(\llbracket \widehat{\beta}\rrbracket_\circ)$ is an invariant of colored links in $S^3$.
Moreover, even in the uncolored (\ie $1$-colored) case, it is distinct from $\sln$ Khovanov-Rozansky homology.
\end{theorem}
In particular, in the uncolored, $n=2$ case, 
we obtain a categorification of the Jones polynomial \textit{distinct} from Khovanov homology. 
It is also distinct from the odd Khovanov homology studied in \cite{ORS}. 

Up to a grading shift, this invariant can equivalently be described using $\Upsilon_{\Pi \cal{V}_n,t}^\wedge$, 
and we note that $\Pi \cal{V}_n = \Pi q^{1-n}\C[t]/t^n \cong \C^{0|n}$. 
This description is hence \textit{exactly} the same as that for $\sln$ link homology, 
but with $\C^{n|0} \cong \cal{V}_n$ instead replaced by $\C^{0|n}$.
For this reason, we propose that this homology is \textit{$\slnegn$ link homology}, 
and will denote this invariant as $\cal{H}_{-n}(\cal{L}_\beta)$ for the duration. 
See Section \ref{sec:SymHomology} for further details, including a proof of Theorem \ref{thm:negativen}.

\begin{remark}
Since $\dim_q(\Sym^k(\cal{V}_n)) = {n+k-1 \brack k}$, 
our invariant $\cal{H}_{-n}(\cal{L}_\beta)$ precisely decategorifies to the description of the 
$\sln$ link polynomials given in Proposition \ref{prop:SymPoly}. 
We've hence resolved the question, posed (in the $n=2$ case) in \cite{RTub}, 
of finding a categorification of the symmetric webs presentation of the $\sln$ link polynomials.
\end{remark}

\begin{remark}
We will see in Section \ref{sec:Mirror} that a variant of triply-graded HOMFLYPT homology,
denoted $\HHH^\vee(\cal{L}_\beta)$, can be obtained using $\Upsilon_{\cal{V}_\infty^\vee,t}^\wedge$ in Theorem \ref{thm:LinkHomologies}, 
where $\cal{V}_\infty^\vee$ is a re-graded variant of the super vector space $\cal{V}_\infty = qa^{-1} \C[t,\theta]/\theta^2$.
An isomorphism between $\HHH^\vee(\cal{L}_\beta)$ and $\HHH(\cal{L}_\beta)$ will categorify the mirror symmetry in Remark \ref{rem:Mirror}.
From the perspective of $\HHH^\vee(\cal{L}_\beta)$, 
$\KhR_n(\cal{L}_\beta)$ ``looks like $\sln$ link homology'' and $\cal{H}_{-n}(\cal{L}_\beta)$ ``looks like $\slnegn$ homology,'' 
but from the perspective of $\HHH(\cal{L}_\beta)$, the roles actually appear reversed!
\end{remark}

\section{Categorified quantum groups and traces}\label{sec:CatUqTrace}

In this section, we review background material on categorified quantum groups and categorical traces 
that underlies the constructions in Section \ref{sec:Foams}, 
and that is necessary to prove 
Theorem \ref{thm:Meta},
Proposition \ref{prop:Functors}, 
and Theorem \ref{thm:LinkHomologies}.

\subsection{Categorified quantum groups}\label{sec:2-kac-moody}

Fix a positive integer $m$. 
We denote by $\UU(\slm) := \UU_Q(\slm)$ the cyclic version of the categorified quantum group associated to $\slm$
and corresponding to the fixed choice of scalars $Q$ satisfying: 
\[
t_{i,j} = 
\begin{cases} 
-1 \text{ if } i=j+1 \\
1 \text{ else }
\end{cases}
\quad \text{and} \quad
s_{ij}^{pq} = 0 .
\]
We also choose a compatible choice of bubbles parameters $c_{i,\lambda}$, 
such that\footnote{The specific values of the bubbles parameters will not play a role in our work.}  
$c_{i,\lambda} = \pm1$ for all $\slm$ weights $\lambda$.
See \cite{BHLW2} for the definition of the cyclic version of the categorified quantum group, 
and \cite{KL,KL2,KL3,Rou2,CLau} for earlier versions of this 2-category.

It will often be useful for us to work with $\glm$ instead of $\slm$.
Following \cite{MSV2}, we can succinctly define the corresponding categorified quantum group as
\[
\UU(\glm) := \coprod_\Z \UU(\slm).
\]

Concretely, $\UU(\glm)$ admits a description almost identical to that of $\UU(\slm)$, 
but with objects now given by $\glm$ weights $\mathbf{k} = [k_1,\ldots,k_m]$ instead of 
$\slm$ weights $\lambda = (\lambda_1,\ldots,\lambda_{m-1})$. 
As in $\UU(\slm)$, 2-morphisms in $\UU(\glm)$ are given by string diagrams with regions labeled by weights.
A diagram containing a region labeled by the $\glm$ weight $\mathbf{k}$ 
corresponds to a $2$-morphism in the $\left\lfloor \frac{1}{m} \sum k_i \right\rfloor^{\mathrm{th}}$ copy of $\UU(\slm)$, 
with the corresponding region labeled by the $\slm$ weight $\lambda$ satisfying $\lambda_i = k_i - k_{i+1}$. 
We denote the \textit{Schur quotient} of $\UU(\glm)$ by $\UU(\glm)^{\geq 0}$; this is the quotient of $\UU(\glm)$ obtained by  
killing all objects $[k_1,\dots,k_m]$ such that $k_i < 0$ for some $i$.
Finally, we follow the standard convention that the notation $\Udot$ corresponds to taking the Karoubi completion, 
\ie $\Udot(\slm)$ is the 2-category obtained from $\UU(\slm)$ by taking the Karoubi completion in each $\Hom$-category.

Of particular importance are the \textit{Rickard complexes}, 
certain chain complexes $\tau_i \onel$ in $\Udot(\slm)$ that generate categorical braid group actions on 
integrable 2-representations of $\Udot(\slm)$. They are explicitly given as follows:
\begin{equation} \label{eq:Rickard>}
\tau_i  \onel =
\xymatrix{ 
\uwave{\cal{F}_i^{(\lambda_i)} \onel} \ar[r]^-{d_0} & q^{-1} \cal{F}_i^{(\lambda_i+1)} \cal{E}_i \onel \ar[r]^-{d_1} & \cdots
\ar[r]^-{d_{b-1}} & q^{-b}  \cal{F}_i^{(\lambda_i+b)} \cal{E}_i^{(b)} \onel \ar[r]^-{d_{b}} &\cdots}
\end{equation}
when $\lambda_i \geq 0$ and
\begin{equation} \label{eq:Rickard<}
\tau_i \onel =
\xymatrix{
\uwave{ \cal{E}_i^{(-\lambda_i)} } \onel \ar[r]^-{d_0}
 & q^{-1} \cal{E}_i^{(-\lambda_i+1)} \cal{F}_i \onel  \ar[r]^-{d_1} & \cdots
\ar[r]^-{d_{a-1}} & q^{-a} \cal{E}_i^{(-\lambda_i+a)} \cal{F}_i^{(a)}  \onel \ar[r]^-{d_{a}} &\cdots}
\end{equation}
when $\lambda_i \leq 0$.
The 1-morphisms $\cal{E}_i^{(b)} \onel$ and $\cal{F}_i^{(c)} \onel$ are the divided power 1-morphisms defined in \cite{KLMS}, 
and the differentials on the complex are described using the so-called ``thick calculus'' from that work. 
See \cite{Cautis, QR} for full details.

\subsection{Categorified $\glm$ and foams}

The connection between categorified quantum groups and $\sln$ foams/link homology is studied in \cite{LQR,QR}, 
where the former is used to define the latter, as motivated by categorical skew Howe duality. 
The main result from that paper, extended to the case of $\vFoam$ and given using our conventions for categorified quantum groups above, 
takes the following form.

\begin{theorem}\label{thm:Foamation}
For each $m \geq 2$, there exists a 2-functor $\UU(\glm) \xrightarrow{\Phi_{\infty}} \vFoam$ determined by the following:
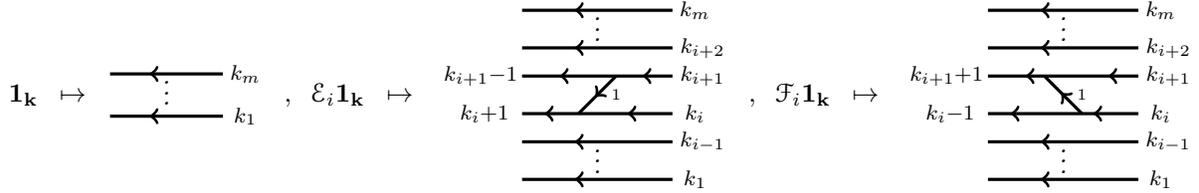
\[
\onek
\;\; \mapsto \;\;
\xy
(0,0)*{
\begin{tikzpicture} [scale=.75]
\draw [very thick, directed=.65] (2,0) -- (0,0);
\draw [very thick, directed=.65] (2,.75) -- (0,.75);
\node at (2.4,0) {\fns$k_1$};
\node at (2.4,.75) {\fns$k_m$};
\node at (1,.55) {$\vdots$};
\end{tikzpicture}};
\endxy
\;\; , \;\;
\cal{E}_i \onek 
\;\; \mapsto \;\;
\xy
(0,0)*{
\begin{tikzpicture} [scale=.5]
\draw [very thick, directed=.65] (4,1.75) -- (0,1.75);
\draw [very thick, directed=.65] (4,2.75) -- (0,2.75);
\draw [very thick, directed=.65] (4,-.75) -- (0,-.75);
\draw [very thick, directed=.65] (4,-1.75) -- (0,-1.75);
\node at (2,2.45){$\vdots$};
\node at (2,-1.05){$\vdots$};
\node at (4.6,2.75) {\fns$k_m$};
\node at (4.8,1.75) {\fns$k_{i+2}$};
\node at (4.8,-.75) {\fns$k_{i-1}$};
\node at (4.6,-1.75) {\fns$k_1$};
\draw [very thick, directed=.85, directed=.3] (4,0) -- (0,0);
\draw [very thick, directed=.2, directed=.75] (4,1) -- (0,1);
\draw [very thick, directed=.55] (2.5,1) -- (1.5,0);
\node at (4.8,1) {\fns$k_{i+1}$};
\node at (4.6,0) {\fns$k_i$};
\node at (-1.1,1) {\fns$k_{i+1}{-}1$};
\node at (-1,0) {\fns$k_i{+}1$};
\node at (2.5,.5) {\tiny$1$};
\end{tikzpicture}};
\endxy
\;\; , \;\;
\cal{F}_i \onek 
\;\; \mapsto \;\;
\xy
(0,0)*{
\begin{tikzpicture} [scale=.5]
\draw [very thick, directed=.65] (4,1.75) -- (0,1.75);
\draw [very thick, directed=.65] (4,2.75) -- (0,2.75);
\draw [very thick, directed=.65] (4,-.75) -- (0,-.75);
\draw [very thick, directed=.65] (4,-1.75) -- (0,-1.75);
\node at (2,2.45){$\vdots$};
\node at (2,-1.05){$\vdots$};
\node at (4.6,2.75) {\fns$k_m$};
\node at (4.8,1.75) {\fns$k_{i+2}$};
\node at (4.8,-.75) {\fns$k_{i-1}$};
\node at (4.6,-1.75) {\fns$k_1$};
\draw [very thick, directed=.85, directed=.3] (4,0) -- (0,0);
\draw [very thick, directed=.2, directed=.75] (4,1) -- (0,1);
\draw [very thick, directed=.55] (2.5,0) -- (1.5,1);
\node at (4.8,1) {\fns$k_{i+1}$};
\node at (4.6,0) {\fns$k_i$};
\node at (-1.1,1) {\fns$k_{i+1}{+}1$};
\node at (-1,0) {\fns$k_i{-}1$};
\node at (2.5,.5) {\tiny$1$};
\end{tikzpicture}};
\endxy
\]
It is determined on the positive part of 
$\UU(\glm)$ via:
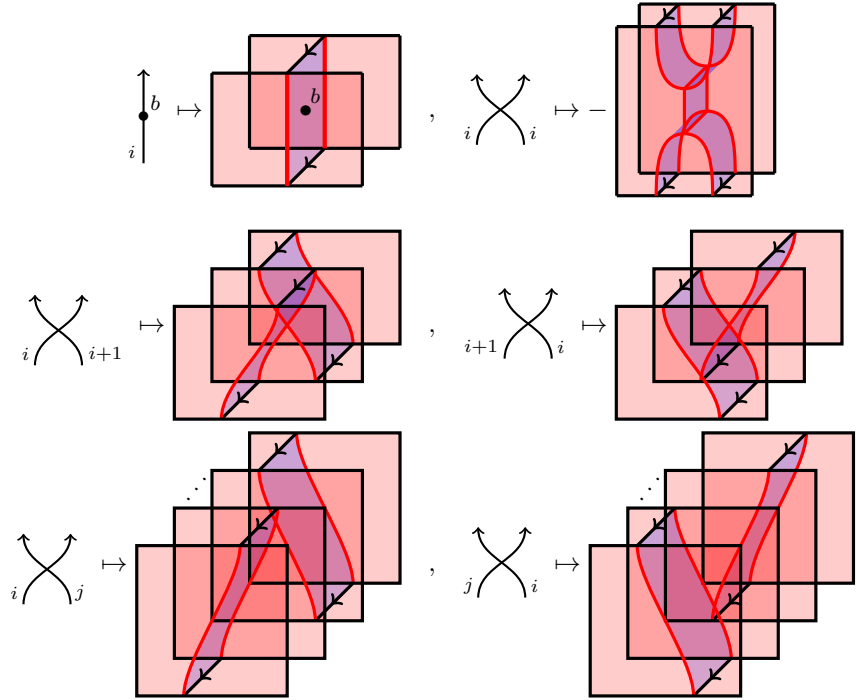
\[
\begin{aligned}
\xy
(0,0)*{
\begin{tikzpicture} [scale=.625]
	\draw[thick,->] (0,0) to (0,2);
	\node at (0,1) {\small$\bullet$};
	\node at (0.25,1.3) {\fns$b$};
	\node at (-.25,.25) {\scs$i$};
\end{tikzpicture}};
\endxy
\mapsto 
\xy
(0,0)*{
\begin{tikzpicture} [fill opacity=0.2, scale=.5, xscale=-1]
     	\path[fill=red] (-2,0) to (2,0) to (2,3) to (-2,3);
     	\path [fill=blue] (1,-1) to (1,2) to (0,3) to (0,0);
     	\filldraw [fill=red] (-1,-1) to (3,-1) to (3,2) to (-1,2);
      	\draw[very thick] (-1,-1) -- (3,-1);
    	\draw[very thick, rdirected=.65] (1,-1) to (0,0);
     	\draw[very thick] (-2,0) -- (2,0);
      	\draw[ultra thick, red] (0,0) to (0,3);
     	\draw[ultra thick, red] (1,-1) to (1,2);
      	\draw[very thick] (-2,3) to (2,3);
    	\draw[very thick, rdirected=.65] (1,2) to (0,3);
     	\draw[very thick] (-1,2) to (3,2);
	\draw[very thick] (-1,-1) to (-1,2);
	\draw[very thick] (3,-1) to (3,2);
	\draw[very thick] (2,0) to (2,3);
	\draw[very thick] (-2,0) to (-2,3);
	\node [opacity=1]  at (0.5,1) {$\bullet$};	
	\node[opacity=1] at (0.25,1.4) {\fns$b$};
\end{tikzpicture}};
\endxy
\quad &, \;\; 
\xy
(0,0)*{
\begin{tikzpicture} [scale=.625]
	\draw[thick,->] (0,0) to [out=90,in=270] (1,1.5);
	\draw[thick,->] (1,0) to [out=90,in=270] (0,1.5);
	\node at (-.2,.25) {\scs$i$};
	\node at (1.2,.25) {\scs$i$};
\end{tikzpicture}};
\endxy
\mapsto 
- \; \xy
(0,0)*{
\begin{tikzpicture} [scale=.3,fill opacity=0.2,xscale=-1]
	\path[fill=red] (3.5,0) to (-2.5,0) to (-2.5,7.5) to (3.5,7.5);
	\path[fill=red] (2.5,1) to (-3.5,1) to (-3.5,8.5) to (2.5,8.5);
	\path[fill=blue] (.5,2.75) to (-.5,3.75) to (-.5,5.75) to (.5,4.75);
	\path[fill=blue] (-.75,7.5) to [out=270,in=130] (-.125,4.91) to (-1.125,5.91) to [out=130,in=270] (-1.75,8.5);
	\path[fill=blue] (-.125,4.91) to [out=320,in=180] (.5,4.75) to (-.5,5.75) to [out=180,in=320] (-1.125,5.91);
	\path[fill=blue] (.5,4.75) to [out=0,in=270] (1.75,7.5) to (.75,8.5) to [out=270,in=0] (-.5,5.75);
	\path[fill=blue] (-.75,0) to [out=90,in=180] (.5,2.75) to (-.5,3.75) to [out=180,in=90] (-1.75,1);
	\path[fill=blue] (.5,2.75) to [out=0,in=140] (1.125,2.59) to (.125,3.59) to [out=140,in=0] (-.5,3.75);
	\path[fill=blue] (1.125,2.59) to [out=320,in=270] (1.75,0) to (.75,1) to [out=270,in=320] (.125,3.59);
	\draw[very thick] (3.5,0) to (-2.5,0);
	\draw[very thick] (2.5,1) to (-3.5,1);
	\draw[very thick, rdirected=.6] (1.75,0) to (.75,1);	
	\draw[very thick, rdirected=.6] (-.75,0) to (-1.75,1);		
	\draw[very thick, red] (-.75,0) to [out=90,in=180] (.5,2.75);
	\draw[very thick, red] (1.75,0) to [out=90,in=0] (.5,2.75);
	\draw[very thick, red] (.5,2.75) to [out=90,in=270] (.5,4.75);
	\draw[very thick, red] (.5,4.75) to [out=0,in=270] (1.75,7.5);
	\draw[very thick, red] (.5,4.75) to [out=180,in=270] (-.75,7.5);
	\draw[very thick, red] (.5,2.75) to (-.5,3.75);
	\draw[very thick, red] (.5,4.75) to (-.5,5.75);
	\draw[very thick, red] (-1.75,1) to [out=90,in=180] (-.5,3.75);
	\draw[very thick, red] (.75,1) to [out=90,in=0] (-.5,3.75);
	\draw[very thick, red] (-.5,3.75) to [out=90,in=270] (-.5,5.75);
	\draw[very thick, red] (-.5,5.75) to [out=0,in=270] (.75,8.5);
	\draw[very thick, red] (-.5,5.75) to [out=180,in=270] (-1.75,8.5);
	\draw[very thick] (3.5,0) to (3.5,7.5);	
	\draw[very thick] (2.5,1) to (2.5,8.5);
	\draw[very thick] (-2.5,0) to (-2.5,7.5);	
	\draw[very thick] (-3.5,1) to (-3.5,8.5);
	\draw[very thick] (3.5,7.5) to (-2.5,7.5);
	\draw[very thick] (2.5,8.5) to (-3.5,8.5);
	\draw[very thick, rdirected=.6] (1.75,7.5) to (.75,8.5);	
	\draw[very thick, rdirected=.6] (-.75,7.5) to (-1.75,8.5);
	\end{tikzpicture}
};
\endxy \\
\xy
(0,0)*{
\begin{tikzpicture} [scale=.625]
	\draw[thick,->] (0,0) to [out=90,in=270] (1,1.5);
	\draw[thick,->] (1,0) to [out=90,in=270] (0,1.5);
	\node at (-.2,.25) {\scs$i$};
	\node at (1.5,.25) {\scs$i{+}1$};
\end{tikzpicture}};
\endxy
\mapsto 
\xy
(0,0)*{
\begin{tikzpicture}  [fill opacity=0.2, scale=.5, xscale=-1]
	\fill [fill=red] (-3,1) rectangle (1,4);
	\path [fill=blue] (-.25,4) .. controls (-.25,3) and (-1.75,2) .. (-1.75,1) --
			(-.75,0) .. controls (-.75,1) and (.75,2) .. (.75,3);	
	\fill [fill=red] (-2,0) rectangle (2,3);
	\path [fill=blue] (-.75,3) .. controls (-.75,2) and (.75,1) .. (.75,0) --
			(1.75,-1) .. controls (1.75,0) and (.25,1) .. (.25,2);
	\fill [fill=red] (-1,-1) rectangle (3,2);
	\draw [very thick, red] (-.25,4) .. controls (-.25,3) and (-1.75,2) .. (-1.75,1);
	\draw [very thick] (-3,1) rectangle (1,4);
	\draw[very thick, rdirected=.6] (-.75,0) -- (-1.75,1);
	\draw[very thick, rdirected=.6] (1.75,-1) -- (.75,0);
	\draw [very thick, red] (-.75,3) .. controls (-.75,2) and (.75,1) .. (.75,0);
	\draw [very thick, red] (.75,3) .. controls (.75,2) and (-.75,1) .. (-.75,0);
	\draw [very thick] (-2,0) rectangle (2,3);
	\draw[very thick, rdirected=.6] (.25,2) -- (-.75,3);
	\draw[very thick, rdirected=.6] (.75,3) -- (-.25,4);
	\draw[very thick, red] (1.75,-1) .. controls (1.75,0) and (.25,1) .. (.25,2);
	\draw [very thick] (-1,-1) rectangle (3,2);
\end{tikzpicture}}
\endxy
\quad &, \;\; 
\xy
(0,0)*{
\begin{tikzpicture} [scale=.625]
	\draw[thick,->] (0,0) to [out=90,in=270] (1,1.5);
	\draw[thick,->] (1,0) to [out=90,in=270] (0,1.5);
	\node at (-.5,.25) {\scs$i{+}1$};
	\node at (1.2,.25) {\scs$i$};
\end{tikzpicture}};
\endxy
\mapsto 
\xy
(0,0)*{
\begin{tikzpicture}  [fill opacity=0.2, scale=.5, xscale=-1]
	\fill [fill=red] (-3,1) rectangle (1,4);
	\path [fill=blue] (-1.75,4) .. controls (-1.75,3) and (-.25,2) .. (-.25,1) --
			(.75,0) .. controls (.75,1) and (-.75,2) .. (-.75,3);
	\fill [fill=red] (-2,0) rectangle (2,3);
	\path [fill=blue] (.75,3) .. controls (.75,2) and (-.75,1) .. (-.75,0) --
			(.25,-1) .. controls (.25,0) and (1.75,1) .. (1.75,2);
	\fill [fill=red] (-1,-1) rectangle (3,2);
	\draw[very thick, rdirected=.6] (.75,0) -- (-.25,1);
	\draw[very thick, rdirected=.6] (.25,-1) -- (-.75,0);
	\draw [very thick, red] (-1.75,4) .. controls (-1.75,3) and (-.25,2) .. (-.25,1);
	\draw [very thick] (-3,1) rectangle (1,4);
	\draw [very thick, red] (-.75,3) .. controls (-.75,2) and (.75,1) .. (.75,0);
	\draw [very thick, red] (.75,3) .. controls (.75,2) and (-.75,1) .. (-.75,0);
	\draw [very thick] (-2,0) rectangle (2,3);
	\draw[very thick, rdirected=.6] (1.75,2) -- (.75,3);
	\draw[very thick, rdirected=.6] (-.75,3) -- (-1.75,4);
	\draw[very thick, red] (.25,-1) .. controls (.25,0) and (1.75,1) .. (1.75,2);
	\draw [very thick] (-1,-1) rectangle (3,2);
\end{tikzpicture}}
\endxy \\
\xy
(0,0)*{
\begin{tikzpicture} [scale=.625]
	\draw[thick,->] (0,0) to [out=90,in=270] (1,1.5);
	\draw[thick,->] (1,0) to [out=90,in=270] (0,1.5);
	\node at (-.2,.25) {\scs$i$};
	\node at (1.2,.25) {\scs$j$};
\end{tikzpicture}};
\endxy
\mapsto 
\xy
(0,0)*{
\begin{tikzpicture} [fill opacity=0.2,scale=.5,xscale=-1]
	\fill [fill=red] (-4,2) rectangle (0,6);
	\path [fill=blue] (-1.25,6) .. controls (-1.25,5) and (-2.75,3) .. (-2.75,2) --
		 (-1.75,1) .. controls (-1.75,2) and (-.25,4) .. (-.25,5);
	\fill [fill=red] (-3,1) rectangle (1,5);
	\fill [fill=red] (-2,0) rectangle (2,4);
	\path [fill=blue] (-.75,4) .. controls (-.75,3) and (.75,1) .. (.75,0) --
		(1.75,-1) .. controls (1.75,0) and (.25,2) .. (.25,3);
	\fill [fill=red] (-1,-1) rectangle (3,3);
	\draw[very thick, red] (-1.25,6) .. controls (-1.25,5) and (-2.75,3) .. (-2.75,2);
	\draw [very thick] (-4,2) rectangle (0,6);
	\draw[very thick, rdirected=.6] (-1.75,1) -- (-2.75,2);
	\draw[very thick, rdirected=.6] (1.75,-1) -- (.75,0);
	\draw[very thick, red] (-1.75,1) .. controls (-1.75,2) and (-.25,4) .. (-.25,5);
	\draw [very thick] (-3,1) rectangle (1,5);
	\draw[very thick, red] (-.75,4) .. controls (-.75,3) and (.75,1) .. (.75,0);
	\draw [very thick] (-2,0) rectangle (2,4);
	\draw[very thick, rdirected=.6] (.25,3) -- (-.75,4);
	\draw[very thick, rdirected=.6] (-.25,5) -- (-1.25,6);
	\draw[very thick, red] (1.75,-1) .. controls (1.75,0) and (.25,2) .. (.25,3);
	\draw [very thick] (-1,-1) rectangle (3,3);
	\node[opacity=1,rotate=90] at (1.625,4.5) {\small$\ddots$};
\end{tikzpicture}}
\endxy
\quad &, \;\;
\xy
(0,0)*{
\begin{tikzpicture} [scale=.625]
	\draw[thick,->] (0,0) to [out=90,in=270] (1,1.5);
	\draw[thick,->] (1,0) to [out=90,in=270] (0,1.5);
	\node at (-.2,.25) {\scs$j$};
	\node at (1.2,.25) {\scs$i$};
\end{tikzpicture}};
\endxy
\mapsto 
\xy
(0,0)*{
\begin{tikzpicture} [fill opacity=0.2,scale=.5,xscale=-1]
	\fill [fill=red] (-4,2) rectangle (0,6);
	\path [fill=blue] (-2.75,6) .. controls (-2.75,5) and (-1.25,3) .. (-1.25,2) --
		 (-.25,1) .. controls (-.25,2) and (-1.75,4) .. (-1.75,5);
	\fill [fill=red] (-3,1) rectangle (1,5);
	\fill [fill=red] (-2,0) rectangle (2,4);
	\path [fill=blue] (.75,4) .. controls (.75,3) and (-.75,1) .. (-.75,0) --
		(.25,-1) .. controls (.25,0) and (1.75,2) .. (1.75,3);
	\fill [fill=red] (-1,-1) rectangle (3,3);
	\draw[very thick, red] (-2.75,6) .. controls (-2.75,5) and (-1.25,3) .. (-1.25,2);
	\draw [very thick] (-4,2) rectangle (0,6);
	\draw[very thick, rdirected=.6] (-.25,1) -- (-1.25,2);
	\draw[very thick, rdirected=.6] (.25,-1) -- (-.75,0);
	\draw[very thick, red] (-.25,1) .. controls (-.25,2) and (-1.75,4) .. (-1.75,5);
	\draw [very thick] (-3,1) rectangle (1,5);
	\draw[very thick, red] (.75,4) .. controls (.75,3) and (-.75,1) .. (-.75,0);
	\draw [very thick] (-2,0) rectangle (2,4);
	\draw[very thick, rdirected=.6] (1.75,3) -- (.75,4);
	\draw[very thick, rdirected=.6] (-1.75,5) -- (-2.75,6);
	\draw[very thick, red] (.25,-1) .. controls (.25,0) and (1.75,2) .. (1.75,3);
	\draw [very thick] (-1,-1) rectangle (3,3);
	\node[opacity=1,rotate=90] at (1.625,4.5) {\small$\ddots$};
\end{tikzpicture}}
\endxy
\end{aligned}
\]
where the $i^{th}$ sheet is always in the front, and $j - i > 1$.
On cap and cup generators, 
$\Phi_{\infty}$ is given\footnote{The factors of $\pm 1$ on the left-oriented cup and cap can be explicitly specified, 
using the $2$-functor from \cite{QR} and a version of the isomorphism in \cite[Theorem 2.1]{BHLW2} 
that rescales the left-oriented cap/cups instead of the right-oriented ones. We won't need specific knowledge of these values in the present work.}
by:
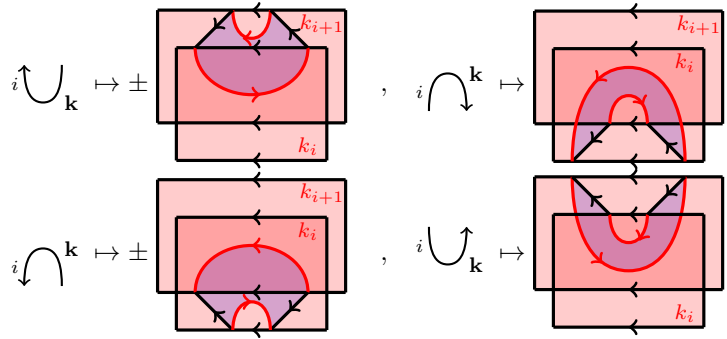
\[
\begin{aligned}
\xy
(0,0)*{
\begin{tikzpicture} [scale=.5,yscale=-1]
	\draw[thick,<-] (0,0) to [out=90,in=180] (.5,1) to [out=0,in=90] (1,0);
	\node at (1.25,1) {\fns$\mathbf{k}$};
	\node at (-.25,.5) {\scs$i$};
\end{tikzpicture}};
\endxy
\mapsto 
\pm \;
\xy
(0,0)*{
\begin{tikzpicture} [scale=.5,fill opacity=0.2]
\path[fill=red] (2,-2) to (-2,-2) to (-2,1) to (2,1);
\path[fill=blue] (1.5,1) to [out=270,in=0] (0,-.25) to [out=180,in=270] (-1.5,1) to
	(-.5,2) to [out=270,in=180] (0,1.25) to [out=0,in=270] (.5,2);
\path[fill=red] (2.5,-1) to (-2.5,-1) to (-2.5,2) to (2.5,2);
\draw[very thick, directed=.5] (2.5,-1) to (-2.5,-1);
\draw[very thick] (2.5,-1) to (2.5,2);
\draw[very thick] (-2.5,-1) to (-2.5,2);
\draw[very thick, directed=.5] (2.5,2) to (-2.5,2);
\draw[very thick, red, rdirected=.5] (-.5,2) to [out=270,in=180] (0,1.25)
	to [out=0,in=270] (.5,2);
\node[red, opacity=1] at (1.875,1.625) {\fns$k_{i+1}$};
\draw[very thick, directed=.5] (2,-2) to (-2,-2);
\draw[very thick] (2,-2) to (2,1);
\draw[very thick] (-2,-2) to (-2,1);
\draw[very thick, directed=.5] (2,1) to (-2,1);
\draw[very thick, red, rdirected=.5] (1.5,1) to [out=270,in=0] (0,-.25)
	to [out=180,in=270] (-1.5,1);
\node[red, opacity=1] at (1.5,-1.625) {\fns$k_i$};
\draw[very thick, directed=.5] (1.5,1) to (.5,2);
\draw[very thick, directed=.5] (-.5,2) to (-1.5,1);
\end{tikzpicture}
}
\endxy
\quad &, \;\; 
\xy
(0,0)*{
\begin{tikzpicture} [scale=.5]
	\draw[thick,->] (0,0) to [out=90,in=180] (.5,1) to [out=0,in=90] (1,0);
	\node at (1.25,1) {\fns$\mathbf{k}$};
	\node at (-.25,.5) {\scs$i$};
\end{tikzpicture}};
\endxy
\mapsto
\xy
(0,0)*{
\begin{tikzpicture} [scale=.5,fill opacity=0.2]
\path[fill=red] (2.5,2) to (-2.5,2) to (-2.5,-1) to (2.5,-1);
\path[fill=blue] (1.5,-2) to [out=90,in=0] (0,.5) to [out=180,in=90] (-1.5,-2) to
	(-.5,-1) to [out=90,in=180] (0,-.25) to [out=0,in=90] (.5,-1);
\path[fill=red] (2,1) to (-2,1) to (-2,-2) to (2,-2);
\draw[very thick, directed=.5] (2.5,2) to (-2.5,2);
\draw[very thick] (2.5,2) to (2.5,-1);
\draw[very thick] (-2.5,2) to (-2.5,-1);
\draw[very thick, directed=.5] (2.5,-1) to (-2.5,-1);
\draw[very thick, red, directed=.75] (-.5,-1) to [out=90,in=180] (0,-.25)
	to [out=0,in=90] (.5,-1);
\node[red, opacity=1] at (1.875,1.625) {\fns$k_{i+1}$};
\draw[very thick, directed=.5] (1.5,-2) to (.5,-1);
\draw[very thick, directed=.5] (-.5,-1) to (-1.5,-2);
\draw[very thick, directed=.5] (2,1) to (-2,1);
\draw[very thick] (2,1) to (2,-2);
\draw[very thick] (-2,1) to (-2,-2);
\draw[very thick, directed=.5] (2,-2) to (-2,-2);
\draw[very thick, red, directed=.65] (1.5,-2) to [out=90,in=0] (0,.5)
	to [out=180,in=90] (-1.5,-2);
\node[red, opacity=1] at (1.5,.625) {\fns$k_i$};
\end{tikzpicture}};
\endxy \\
\xy
(0,0)*{
\begin{tikzpicture} [scale=.5]
	\draw[thick,<-] (0,0) to [out=90,in=180] (.5,1) to [out=0,in=90] (1,0);
	\node at (1.25,1) {\fns$\mathbf{k}$};
	\node at (-.25,.5) {\scs$i$};
\end{tikzpicture}};
\endxy
\mapsto 
\pm \;
\xy
(0,0)*{
\begin{tikzpicture} [scale=.5,fill opacity=0.2]
\path[fill=red] (2.5,2) to (-2.5,2) to (-2.5,-1) to (2.5,-1);
\path[fill=blue] (1.5,-1) to [out=90,in=0] (0,.25) to [out=180,in=90] (-1.5,-1) to
	(-.5,-2) to [out=90,in=180] (0,-1.25) to [out=0,in=90] (.5,-2);
\path[fill=red] (2,1) to (-2,1) to (-2,-2) to (2,-2);
\draw[very thick, directed=.5] (2.5,2) to (-2.5,2);
\draw[very thick] (2.5,2) to (2.5,-1);
\draw[very thick] (-2.5,2) to (-2.5,-1);
\draw[very thick, directed=.5] (2.5,-1) to (-2.5,-1);
\draw[very thick, red, directed=.5] (1.5,-1) to [out=90,in=0] (0,.25)
	to [out=180,in=90] (-1.5,-1);
\node[red, opacity=1] at (1.875,1.625) {\fns$k_{i+1}$};
\draw[very thick, directed=.5] (1.5,-1) to (.5,-2);
\draw[very thick, directed=.5] (-.5,-2) to (-1.5,-1);
\draw[very thick, directed=.5] (2,1) to (-2,1);
\draw[very thick] (2,1) to (2,-2);
\draw[very thick] (-2,1) to (-2,-2);
\draw[very thick, directed=.5] (2,-2) to (-2,-2);
\draw[very thick, red, directed=.5] (-.5,-2) to [out=90,in=180] (0,-1.25)
	to [out=0,in=90] (.5,-2);
\node[red, opacity=1] at (1.5,.625) {\fns$k_i$};
\end{tikzpicture}};
\endxy 
\quad &, \;\; 
\xy
(0,0)*{
\begin{tikzpicture} [scale=.5,yscale=-1]
	\draw[thick,->] (0,0) to [out=90,in=180] (.5,1) to [out=0,in=90] (1,0);
	\node at (1.25,1) {\fns$\mathbf{k}$};
	\node at (-.25,.5) {\scs$i$};
\end{tikzpicture}};
\endxy
\mapsto
\xy
(0,0)*{
\begin{tikzpicture} [scale=.5,fill opacity=0.2]
\path[fill=red] (2,-2) to (-2,-2) to (-2,1) to (2,1);
\path[fill=blue] (1.5,2) to [out=270,in=0] (0,-.5) to [out=180,in=270] (-1.5,2) to
	(-.5,1) to [out=270,in=180] (0,.25) to [out=0,in=270] (.5,1);
\path[fill=red] (2.5,-1) to (-2.5,-1) to (-2.5,2) to (2.5,2);
\draw[very thick, directed=.5] (2.5,-1) to (-2.5,-1);
\draw[very thick] (2.5,-1) to (2.5,2);
\draw[very thick] (-2.5,-1) to (-2.5,2);
\draw[very thick, directed=.5] (2.5,2) to (-2.5,2);
\draw[very thick, red, rdirected=.65] (1.5,2) to [out=270,in=0] (0,-.5)
	to [out=180,in=270] (-1.5,2);
\draw[very thick, directed=.5] (2,-2) to (-2,-2);
\draw[very thick] (2,-2) to (2,1);
\draw[very thick] (-2,-2) to (-2,1);
\draw[very thick, directed=.5] (2,1) to (-2,1);
\draw[very thick, red, rdirected=.75] (-.5,1) to [out=270,in=180] (0,.25)
	to [out=0,in=270] (.5,1);
\node[red, opacity=1] at (1.5,-1.625) {\fns$k_i$};
\draw[very thick, directed=.5] (1.5,2) to (.5,1);
\draw[very thick, directed=.5] (-.5,1) to (-1.5,2);
\end{tikzpicture}};
\endxy
\end{aligned}
\]
This 2-functor factors through $\UU(\glm)^{\geq 0}$ and 
extends to the full 2-subcategory $\Ucheck(\glm)^{\geq 0} \subset \Udot(\glm)^{\geq 0}$ generated by the 
divided power morphisms $\cal{E}_i^{(a)} \onek$ and $\cal{F}_i^{(b)} \onek$.
The crossing formulae given in equations \eqref{eq:slnCrossing} and \eqref{eq:slnColoredCrossing} 
are, up to shifts in quantum and homological degree, the images of the Rickard complexes and their inverses.
\end{theorem}

We will combine this functor with the theory of categorical traces to prove Proposition \ref{prop:Functors}.

\subsection{Traces of 2-categories and current algebras}

In this section, we quickly recall the definitions of categorical traces of a linear 2-category $\mathcal{C}$,
which are defined and extensively studied in~\cite[Section~2]{BHLZ}. 
First, the \textit{vertical trace} is the category $\vTr \mathcal{C}$ in which the objects are the same as those in $\mathcal C$, 
and, given objects $\mathbf{c}, \mathbf{d} \in \mathrm{Ob}(\mathcal{C})$,
\begin{equation}\label{eq:vTrDef}
\Hom_{\vTr \mathcal{C}}(\mathbf{c}, \mathbf{d}) = \left(\bigoplus_{X \in \Hom_\cal{C}(\mathbf{c}, \mathbf{d})} \twoEnd(X)\right) \Bigg/ ( \alpha \gamma - \gamma \alpha )
\end{equation}
where the quotient is taken over all $2$-morphisms $\alpha \colon X \Rightarrow Y$ and $\gamma\colon Y \Rightarrow X$ 
and all $X,Y \colon \mathbf{c} \to \mathbf{d}$. Composition in $\vTr \mathcal{C}$ is given using horizontal composition in $\cal{C}$.

This definition is elucidated by the graphical depiction of 2-categories. 
Recall that in this formalism, 
a $2$-morphism $\alpha: W \Rightarrow Z$ between $1$-morphisms $W,Z \in \Hom_\cal{C}(\mathbf{c},\mathbf{d})$ 
is described by the ``string diagram'':
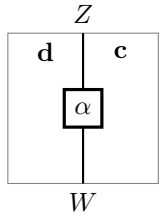
\[
\begin{tikzpicture}[anchorbase, scale=1]
\draw [gray] (0,0) rectangle (2,2);
\draw[very thick] (.75,.75) rectangle (1.25,1.25);
\node at (1,1) {$\alpha$};
\draw[thick] (1,0) to (1,.75);
\draw[thick] (1,1.25) to (1,2);
\node at (1.5,1.75) {$\mathbf{c}$};
\node at (.5,1.75) {$\mathbf{d}$};
\node at (1,-.25) {$W$};
\node at (1,2.25) {$Z$};
\end{tikzpicture}
\]
in which regions are labeled by objects, ``strings'' are labeled by 1-morphisms, and 2-morphisms are denoted by labeled tags. 
In a similar way, morphisms in $\vTr \mathcal{C}$ can be depicted as string diagrams drawn on a cylinder, 
\ie the equivalent class in $\Hom_{\vTr \mathcal{C}}$ of a 2-endomorphism $\alpha$  is depicted via:
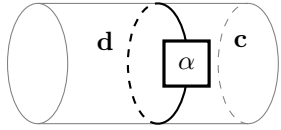
\[
\xy
(0,0)*{
\begin{tikzpicture}[scale=.4]
\draw[gray] (3,-2) arc (-90:90:1 and 2);
\draw[gray, dashed] (3,-2) arc (270:90:1 and 2);
\draw[gray] (-4,0) circle (1 and 2);
\draw[gray] (-4,2) to (3,2);
\draw[gray] (-4,-2) to (3,-2);
\draw[thick] (0,-2) arc (-90:90:1 and 2);
\draw[dashed, thick] (0,-2) arc (270:90:1 and 2);
\fill[white] (.2,-.75) rectangle (1.7,.75);
\draw[very thick] (.2,-.75) rectangle (1.7,.75);
\node at (2.75,.75) {$\mathbf{c}$};
\node at (-1.75,.75) {$\mathbf{d}$};
\node at (.95,0) {$\alpha$};
\end{tikzpicture}
};
\endxy
\]
with composition given by glueing cylinders along circular boundaries. 
The trace relation is encoded diagrammatically via the ability to slide a tag around the back of the cylinder.

This interpretation suggests that interesting structure should also result from the horizontal gluing of a string diagram. 
Along these lines, the \textit{horizontal trace} $\hTr \mathcal{C}$ of a 2-category $\mathcal{C}$
is the category in which objects are 1-endomorphisms $X \colon \mathbf{c} \rightarrow \mathbf{c}$ for $\mathbf{c} \in \mathrm{Ob}(\mathcal C)$. 
Morphisms from $X \colon \mathbf{c} \rightarrow \mathbf{c}$ to $Y \colon \mathbf{d} \rightarrow \mathbf{d}$ are equivalence classes of pairs $(M,\alpha)$ 
consisting of a 1-morphism $M \colon \mathbf{c} \rightarrow \mathbf{d}$ and a 2-morphism $\alpha \colon M \cdot X \Rightarrow Y \cdot M$. 
The equivalence relation\footnote{Here, and for the duration, we denote horizontal composition by $\cdot$, 
and vertical composition by either $\circ$ or concatenation.} 
is generated by $(M, (\id_Y \cdot \tau) \gamma) \sim (N, \gamma (\tau \cdot \id_X))$ 
for $M,N \colon \mathbf{c} \rightarrow \mathbf{d}$, $\gamma \colon M \cdot X \Rightarrow Y \cdot N$, $\tau \colon N \Rightarrow M$, 
and is more-easily visualized using string diagrams:
\begin{equation}\label{eq:78}
\xy
(0,0)*{
\begin{tikzpicture}[scale=.4]
\draw [gray] (-2.5,-2) rectangle (2.5,3);
\draw[very thick] (-2,1) to (2,1) to (2,0) to (-2,0) to (-2,1);
\draw[thick] (-1,1) to (-1,3);
\draw[thick] (-1,0) to (-1,-2);
\draw[thick] (1,0) to (1,-2);
\draw[thick] (1,2.5) to (1,3);
\draw[thick] (1,1) to (1,1.5);
\draw[very thick] (1.5,1.5) to (.5,1.5) to (.5,2.5) to (1.5,2.5) to (1.5,1.5);
\node at (0,.5) {$\gamma$};
\node at (1,2) {$\tau$};
\node at (-1,3.5) {$Y$};
\node at (1,3.5) {$M$};
\node at (1,-2.5) {$X$};
\node at (-1,-2.5) {$M$};
\node at (-1.75,-1) {$\mathbf{d}$};
\node at (0,-1) {$\mathbf{c}$};
\node at (1.75,-1) {$\mathbf{c}$};
\end{tikzpicture}
};
\endxy
\sim
\xy
(0,0)*{
\begin{tikzpicture}[scale=.4]
\draw [gray] (-2.5,-2) rectangle (2.5,3);
\draw[very thick] (-2,1) to (2,1) to (2,0) to (-2,0) to (-2,1);
\draw[thick] (1,1) to (1,3);
\draw[thick] (-1,1) to (-1,3);
\draw[thick] (1,0) to (1,-2);
\draw[thick] (-1,0) to (-1,-.5);
\draw[thick] (-1,-1.5) to (-1,-2);
\draw[very thick] (-1.5,-.5) to (-.5,-.5) to (-.5,-1.5) to (-1.5,-1.5) to (-1.5,-.5);
\node at (0,.5) {$\gamma$};
\node at (-1,-1) {$\tau$};
\node at (-1,3.5) {$Y$};
\node at (1,3.5) {$N$};
\node at (1,-2.5) {$X$};
\node at (-1,-2.5) {$N$};
\node at (-1.75,2) {$\mathbf{d}$};
\node at (0,2) {$\mathbf{d}$};
\node at (1.75,2) {$\mathbf{c}$};
\end{tikzpicture}
};
\endxy
\end{equation}
This is clarified further for the case of 2-categories in which every 1-morphism possesses a biadjoint, in which case we can 
visualize morphisms in $\hTr \cal{C}$ as follows:
\[
(M,\alpha) =
\xy
(0,0)*{
\begin{tikzpicture}[scale=.35]
\draw [gray] (-2.5,-2) rectangle (2.5,3);
\draw[very thick] (-2,1) to (2,1) to (2,0) to (-2,0) to (-2,1);
\draw[thick] (1,1) to (1,3);
\draw[thick] (-1,1) to (-1,3);
\draw[thick] (1,0) to (1,-2);
\draw[thick] (-1,0) to (-1,-2);
\node at (0,.5) {$\alpha$};
\node at (-1,3.5) {$Y$};
\node at (1,3.5) {$M$};
\node at (1,-2.5) {$X$};
\node at (-1,-2.5) {$M$};
\node at (-1.75,2) {$\mathbf{d}$};
\node at (1.75,-1) {$\mathbf{c}$};
\end{tikzpicture}
};
\endxy
\;\; \longleftrightarrow \;\;
\xy
(0,0)*{
\begin{tikzpicture}[scale=.35]
\draw[gray] (2,0) arc (360:180:2 and 1);
\draw[gray, dashed] (2,0) arc (0:180:2 and 1);
\draw[gray] (2,0) to (2,7);
\draw[gray] (-2,0) to (-2,7);
\draw[thick, dashed] (-2,1) to [out=90,in=270] (2,5);
\draw[thick] (-2,1) to [out=270,in=180] (-1.25,.25) to [out=0,in=270] (-.5,2.5);
\draw[thick] (2,5) to [out=90,in=0] (1.25,5.75) to [out=180,in=90] (.5,3.5);
\draw[thick] (0,6) to [out=270,in=90] (-.5,3.5);
\draw[thick] (0,-1) to [out=90,in=270] (.5,2.5);
\fill[white] (-1.5,3.5) to (1.5,3.5) to (1.5,2.5) to (-1.5,2.5) to (-1.5,3.5);
\draw[very thick] (-1.5,3.5) to (1.5,3.5) to (1.5,2.5) to (-1.5,2.5) to (-1.5,3.5);
\draw[gray] (0,7) circle (2 and 1);
\node at (0,3) {$\alpha$};
\node at (0,6.5) {$Y$};
\node at (0,-1.5) {$X$};
\node at (-1.25,5) {$\mathbf{d}$};
\node at (1,0) {$\mathbf{c}$};
\end{tikzpicture}
};
\endxy
\]
and the equivalence relation is given via isotopy around the back of the cylinder.

These two notions of trace are related via a canonical functor $\rmi \colon \vTr \mathcal C \rightarrow \hTr \mathcal C$, 
defined on objects by $\rmi \colon \mathbf{c} \mapsto \one_\mathbf{c}$ 
and on morphisms by sending the class of 
$\alpha \in \twoEnd(X)$ 
to the class of the pair $(X,\alpha)$.
In our description of traces in terms of cylinders, 
this functor is given by rotating\footnote{This gives the functor on the nose in the event that the 2-category $\cal{C}$ is pivotal. 
In other circumstances, scalings may be introduced when performing this rotation.} 
the cylinder representing a morphism in $\vTr \cal{C}$ clockwise by $90^\circ$. 
From this description, it is easy to see that $\rmi$ is in fact the inclusion of the full subcategory 
whose objects are the identity 1-morphisms in $\cal{C}$.

In the event that our 2-category $\cal{C}$ is graded, there is an alternative version of the vertical trace, 
where we allow maps of all degrees in the direct sum appearing in equation \eqref{eq:vTrDef}. 
Concretely, given a graded 2-category $\cal{C}$, we first pass to the 2-category $\cal{C}^*$, 
observe that $\vTr\cal{C}^*$ is then enriched in graded vector spaces, and then define $\vTr^*\cal{C}$ to be 
the corresponding graded linear category.
In this setting, the canonical functor extends to give a functor $\vTr^*\cal{C} \xrightarrow{\rmi} \hTr \cal{C}$, 
which is given identically as before. 
Again, this functor is the inclusion of a full subcategory.

\subsection{Current algebras and traces of foam categories}\label{sec:Proof}

To connect with the foam categories from Section \ref{sec:Foams} giving our link invariants, 
note that all traces discussed in the preceding section are functorial, so Theorem \ref{thm:Foamation} gives the commutative diagram:
\begin{equation}\label{eq:TrDiag}
\xymatrix{
\vTr^*\left(\Ucheck(\glm)^{\geq 0}\right) \ar[rr]^-{\vTr^*(\Phi_{\infty})} \ar[d]^-{\rmi} & &\vTr^*\left(\vFoam\right) \ar[d]^-{\rmi} \\
\hTr\left(\Ucheck(\glm)^{\geq 0}\right) \ar[rr]^-{\hTr(\Phi_{\infty})} & &\hTr\left(\vFoam\right)
}
\end{equation}
Moreover, the vertical trace of a category is equivalent to that of its Karoubi completion, so the top left corner of the diagram 
can be replaced by $\vTr^*\left(\UU(\glm)^{\geq 0}\right)$.

We will now use this diagram to prove the non-trivial claims in Theorem \ref{thm:Meta} and Proposition \ref{prop:Functors}
\ie that $\vAFoam_\circ$ is a full subcategory of $\vAFoam$, 
and that $\Upsilon_{\cal{V},T}^\cal{S}$ (and hence $\Upsilon_{\cal{V},T}^\wedge$) are well-defined functors for each choice of $T:\cal{V} \to \cal{V}$.

To begin, we identify some of the categories appearing in equation \eqref{eq:TrDiag}. The following result is immediate:
\begin{proposition}
The category $\hTr\left(\vFoam\right)$ is monoidally 
equivalent to the annular foam category $\vAFoam$.
\end{proposition}
\begin{proof}
The identification between objects in $\hTr\left(\vFoam\right)$ and $\vAFoam$ is given by identifying a 1-endomorphism in 
$\vFoam$ with its annular closure. On morphisms, the pair $(W,\alpha)$ corresponds to the foam that is given by $\alpha$ in the 
front of the thickened annulus, and slides the web $W$ around the back of the thickened annulus.
\end{proof}

Work of Beliakova, Habiro, Lauda, and Webster identifies the category $\vTr^*\left(\UU(\glm)\right)$ with an idempotented 
version of the polynomial current algebra $\glm \otimes \C[t]$. 
Let $U(\slm[t])$ be the enveloping algebra of the current algebra $\slm \otimes \C[t]$, as defined \eg in \cite{BHLW}. 
Denoting the standard Chevalley generators of $\slm$ by $x_i^+$, $x_i^-$ and $\xi_i$ for $1 \leq i \leq m-1$, 
$U(\slm[t])$ has generators $x_{i,r}^+ = x_i^+ \otimes t^r$, $x_{i,r}^- = x_i^- \otimes t^r$ and $\xi_{i,r} =\xi_i \otimes t^r$ for $r\in \Z_{\geq 0}$. 
Let $\U(\slm[t])$ be the idempotented version of this algebra, 
which can be viewed as a category with objects given by $\slm$ weights, and with 
\[
\Hom_{\U(\slm[t])}(\lambda,\mu) := U(\slm[t])/ I_{\lambda,\mu}
\]
for 
\[
I_{\lambda,\mu} = \sum_{i=1}^{n-1}  U(\slm[t])(\xi_{i,0}-\lambda_i) + \sum_{i=1}^n (\xi_{i,0}-\mu_i) U(\slm[t])
\]
Introducing a grading by setting $\deg_q(t) = 2$, 
the category $\U(\slm[t])$ is then a linear category enriched in graded vector spaces, 
and we can hence equivalently consider it as a graded linear category.

The relation to traces is given as follows:
\begin{theorem}[Beliakova-Habiro-Lauda-Webster \cite{BHLW}]\label{thm:BHLW}
There is an isomorphism
\[
\U(\slm[t]) \xrightarrow{\cong} \vTr^* \UU(\slm)
\]
that is the identity on objects and is determined on morphisms by:
\[
x_{i,r}^+
\mapsto
\xy
(0,0)*{
\begin{tikzpicture}[scale=.3]
\draw[gray] (2.5,-2) arc (-90:90:1 and 2);
\draw[gray, dashed] (2.5,-2) arc (270:90:1 and 2);
\draw[gray] (-3.5,0) circle (1 and 2);
\draw[gray] (-3.5,2) to (2.5,2);
\draw[gray] (-3.5,-2) to (2.5,-2);
\draw[thick,directed=.85] (0,-2) arc (-90:90:1 and 2);
\draw[dashed, thick] (0,-2) arc (270:90:1 and 2);
\node at (1,0) {$\bullet$};
\node at (.6,.5) {\scs$r$};
\node at (1.2,-1.5) {\scs$i$};
\end{tikzpicture}
};
\endxy 
\;\; , \;\;
x_{i,r}^-
\mapsto 
\xy
(0,0)*{
\begin{tikzpicture}[scale=.3]
\draw[gray] (2.5,-2) arc (-90:90:1 and 2);
\draw[gray, dashed] (2.5,-2) arc (270:90:1 and 2);
\draw[gray] (-3.5,0) circle (1 and 2);
\draw[gray] (-3.5,2) to (2.5,2);
\draw[gray] (-3.5,-2) to (2.5,-2);
\draw[thick,rdirected=.2] (0,-2) arc (-90:90:1 and 2);
\draw[dashed, thick] (0,-2) arc (270:90:1 and 2);
\node at (1,0) {$\bullet$};
\node at (.6,.5) {\scs$r$};
\node at (1.2,1.5) {\scs$i$};
\end{tikzpicture}
};
\endxy 
\]
Under this functor
\[
\xi_{i,0} \1_\lambda \mapsto \lambda_i \;
\xy
(0,0)*{
\begin{tikzpicture}[scale=.3]
\draw[gray] (2.5,-2) arc (-90:90:1 and 2);
\draw[gray, dashed] (2.5,-2) arc (270:90:1 and 2);
\draw[gray] (-3.5,0) circle (1 and 2);
\draw[gray] (-3.5,2) to (2.5,2);
\draw[gray] (-3.5,-2) to (2.5,-2);
\node at (2.75,.75) {$\lambda$};
\end{tikzpicture}
};
\endxy 
\;\; , \;\;
\xi_{i,r} \mapsto
\sum_{a+b=r}
(a+1) \;
\xy
(0,0)*{
\begin{tikzpicture}[scale=.3]
\draw[gray] (3,-2) arc (-90:90:1 and 2);
\draw[gray, dashed] (3,-2) arc (270:90:1 and 2);
\draw[gray] (-4,0) circle (1 and 2);
\draw[gray] (-4,2) to (3,2);
\draw[gray] (-4,-2) to (3,-2);
	\draw[thick, ->] (-.5,0) arc (0:-360:1);
	\node at (-1.5,-1) {$\bullet$};
	\node at (-1.5,-1.5) {\tiny$\spadesuit{+} a$};
	\draw[thick, <-] (2.5,0) arc (0:-360:1);
	\node at (1.5,-1) {$\bullet$};
	\node at (1.5,-1.5) {\tiny$\spadesuit {+} b$};
\end{tikzpicture}
};
\endxy
\]
where we use the $\spadesuit$-notation from \cite{KLMS} for dotted bubbles.
\end{theorem}

As usual, it will be more convenient to work in the $\glm$ setting, so we consider $\U(\glm[t])$, which is equal to a 
direct sum of copies of $\U(\slm[t])$. More precisely, it is the category whose objects are $\glm$ weights, 
and with 
\[
\Hom_{\U(\glm[t])}(\mathbf{k},\mathbf{p}) := U(\slm[t])/ I_{\lambda,\mu}
\]
where $\lambda$ and $\mu$ are the $\slm$ weights associated with the $\glm$ weights $\mathbf{k}$ and $\mathbf{p}$, respectively. 
We also let $\U(\glm[t])^{\geq 0}$ be the quotient of $\U(\glm[t])$ by the ideal generated by all morphisms that factor through a
$\glm$ weight $\mathbf{k} \in \Z^m \smallsetminus \Z_{\geq 0}^m$.
It immediately follows from Theorem \ref{thm:BHLW} that we have a surjective map
\[
\U(\glm[t])^{\geq 0} \longrightarrow \vTr^* \left(\UU(\glm)^{\geq 0}\right)
\]
and \cite[Theorem 7.5]{BHLW} shows that this map is an isomorphism.

We can also use Theorem \ref{thm:BHLW} to identify $\vTr^*\left(\vFoam\right)$. 
Using the canonical inclusion to the horizontal trace, 
it immediately follows that $\vTr^*\left(\vFoam\right)$ is isomorphic to the full subcategory of 
$\vAFoam$ in which objects are given by concentric, labeled circles.
Recall now from \cite{QR} that the family of 2-functors $\Phi_{\infty}$ from Theorem \ref{thm:Foamation} is ``eventually full,''
\ie given any foam $\varkappa$ in $\vFoam$, 
we can find some $m$ and a 2-morphism $\alpha$ in $\UU(\glm)$ so that $\Phi_{\infty}(\alpha) = \varkappa$. 
It follows that the same is true after taking traces, hence all morphisms in $\vTr^*\left(\vFoam\right)$ 
can be expressed as those in the image of 
$\U(\glm[t])^{\geq 0} \xrightarrow{\vTr(\Phi_{\infty})} \vTr^*\left(\vFoam\right)$.
These are precisely those generated by the annular foams in equation \eqref{eq:FfoamEfoam}, 
so we see that $\vAFoam_\circ$ is equivalent to $\vTr^*\left(\vFoam\right)$.

To summarize, we have shown that the diagram in equation \eqref{eq:TrDiag} takes the form:
\begin{equation}\label{eq:MainDiag}
\xymatrix{
\U(\glm[t])^{\geq 0} \ar[rr]^-{\vTr(\Phi_{\infty})} \ar[d]^-{\rmi} & &\vAFoam_\circ \ar[d]^-{\rmi} \\
\hTr\left(\Ucheck(\glm)^{\geq 0}\right) \ar[rr]^-{\hTr(\Phi_{\infty})} & &\vAFoam
}
\end{equation}
Since the vertical arrows are inclusions of full subcategories, this proves the claim given in Theorem \ref{thm:Meta}, 
that $\vAFoam_\circ$ is a full subcategory of $\vAFoam$.

We conclude this section by proving Proposition \ref{prop:Functors}, 
\ie we show that the functors $\Upsilon_{\cal{V},T}^\cal{S}$ are well-defined.
To do so, we use a variant of Howe duality.
Let $\cal{V}$ be a graded super vector space and let $T:\cal{V} \to \cal{V}$ be a linear transformation. 
The action of $T$ gives $\cal{V}$ the structure of a $\C[t]$-module, so the following is immediate:
\begin{lemma}\label{lem:basic}
The super vector space $\C^m \otimes \cal{V}$ is a representation of $\U(\glm[t])^{\geq 0}$.
Moreover, if $\deg_q(T) = 2$, then this is a graded representation.
\end{lemma}
\begin{proof}
To be explicit, the $\mathbf{k}$-weight space is $\cal{V}$ if $\mathbf{k} = \mathbf{e}_i$, 
the $\glm$ weight with $1$ in the $i^{th}$ entry and all other entries zero, and is zero otherwise. 
Choosing a basis $\{u_1,\dots,u_m\}$ for $\C^m$, the generating morphisms act via:
\[
\begin{aligned}
\1_{\mathbf{e}_i} x_{i,r}^+ \1_{\mathbf{e}_{i+1}} &: u_{i+1}\otimes v \mapsto u_i\otimes T^r(v) \\
\1_{\mathbf{e}_{i+1}} x_{i,r}^- \1_{\mathbf{e}_{i}} &: u_i\otimes v \mapsto u_{i+1}\otimes T^r(v)
\end{aligned}
\]
and are zero otherwise.
\end{proof}

It follows that the tensor algebra $\bigotimes^\bullet \left( \C^m \otimes \cal{V} \right)$ and 
hence its symmetric and skew-symmetric subspaces $\Sym^\bullet \left( \C^m \otimes \cal{V} \right)$ and 
$\bV^\bullet \left( \C^m \otimes \cal{V} \right)$ have induced representations of $\U(\glm[t])^{\geq 0}$, 
which again are graded.
We focus on the symmetric tensors, since, as noted in Remark \ref{rem:FunctorsWedge}, 
we can obtain the skew-symmetric tensors from these by shifting $\cal{V}$ in super-degree.

The representation $\Sym^\bullet \left( \C^m \otimes \cal{V} \right)$ determines a graded functor 
\[
\widehat{\Upsilon}_{\cal{V},T}^{\cal{S}}: \U(\glm[t])^{\geq 0} \to \gr\C\Vect^{\Z/2}
\]
that we now describe. 
The $\glm$-weight space decomposition is given by: 
\begin{equation}\label{eq:skewHowe}
\begin{aligned}
\Sym^\bullet \left( \C^m \otimes \cal{V} \right) 
&\cong \Sym^\bullet \left( \underbrace{\cal{V} \oplus \cdots \oplus \cal{V}}_m \right) \\
&\cong \bigoplus_{\mathbf{k} \in \Z_{\geq 0}^m} \Sym^{k_1} \cal{V} \otimes \cdots \otimes \Sym^{k_m} \cal{V}
\end{aligned}
\end{equation}
so $\widehat{\Upsilon}_{\cal{V},T}^{\cal{S}}(\mathbf{k}) = \Sym^{k_1} \cal{V} \otimes \cdots \otimes \Sym^{k_m} \cal{V}$.
The images of the generating morphisms are given as follows. 
Each of $\widehat{\Upsilon}_{\cal{V},T}^{\cal{S}}(x_{i,r}^\pm)$ is the identity on all tensor factors except the $i^{th}$ and $(i+1)^{st}$. 
On these factors, which take the form $\Sym^{k_i} \cal{V} \otimes \Sym^{k_{i+1}} \cal{V}$, 
we have that
\begin{equation}\label{eq:FunctorCAGens}
\begin{aligned}
\widehat{\Upsilon}_{\cal{V},T}^{\cal{S}}(\1_{\mathbf{k}-\epsilon_i} x_{i,r}^- \1_{\mathbf{k}}) 
&= (\id_{k_i-1} \otimes {_{1}\pi_{k_{i+1}}}) \circ (\id_{k_i-1} \otimes T^r \otimes \id_{k_{i+1}}) 
\circ ({_{k_i-1}}\iota_{1} \otimes \id_{k_{i+1}}) \\
\widehat{\Upsilon}_{\cal{V},T}^{\cal{S}}(\1_{\mathbf{k}+\epsilon_i} x_{i,r}^+ \1_{\mathbf{k}}) 
&= ({_{k_i}\pi_{1}} \otimes \id_{k_{i+1}-1}) \circ (\id_{k_i} \otimes T^r \otimes \id_{k_{i+1}-1}) 
\circ (\id_{k_i} \otimes {_{1}\iota_{k_{i+1}-1}})
\end{aligned}
\end{equation}
where $_{k}\pi_l: \Sym^k\cal{V} \otimes \Sym^l \cal{V} \to \Sym^{k+l} \cal{V}$ is ${{k+l} \choose k}$ times the canonical projection,
$_{k}\iota_l: \Sym^{k+l} \cal{V} \to \Sym^k\cal{V} \otimes \Sym^l \cal{V}$ is 
the inclusion map, 
and $\id_k$ is the identity on $\Sym^k\cal{V}$.

We now claim that $\widehat{\Upsilon}_{\cal{V},T}^{\cal{S}}$ factors through $\vAFoam_\circ$, 
\ie that there exists a functor $\Upsilon_{\cal{V},T}^{\cal{S}}: \vAFoam_\circ \to \gr\C\Vect^{\Z/2}$ so that 
the diagram:
\[
\xymatrix{
\U(\glm[t])^{\geq 0} \ar[rr]^-{\vTr(\Phi_{\infty})} \ar[drrr]_-{\widehat{\Upsilon}_{\cal{V},T}^{\cal{S}}} & &\vAFoam_\circ \ar[dr]^-{\Upsilon_{\cal{V},T}^{\cal{S}}} \\
& & &\gr\C\Vect^{\Z/2}
}
\]
commutes for every $m \geq 1$. Since every object and morphism in $\vAFoam_\circ$ is the image of one in $\U(\glm[t])^{\geq 0}$ 
for $m$ sufficiently large, it is clear how $\Upsilon_{\cal{V},T}^{\cal{S}}$ should be defined: it sends concentric circles labeled by $k_1,\ldots,k_m$ to 
the super vector space: 
\[
\Upsilon_{\cal{V},T}^{\cal{S}}(\mathbf{k}) := \Sym^{k_1} \cal{V} \otimes \cdots \otimes \Sym^{k_m} \cal{V}
\]
and sends the foams $\mathsf{F}_r$ and $\mathsf{E}_r$ from 
equation \eqref{eq:FfoamEfoam} to the images of $x_{i,r}^-$ and $x_{i,r}^+$ from equation \eqref{eq:FunctorCAGens}, respectively.

It thus suffices to check that this definition is well-defined, that is, 
if two morphisms in $\vAFoam_\circ$ are equal (when viewed as morphisms in $\vAFoam$) then their images under 
$\Upsilon_{\cal{V},T}$ agree. To see that this is the case, recall from \cite{QR} that the family of 2-functors $\Phi_{\infty}$ 
from Theorem \ref{thm:Foamation} is ``eventually faithful,'' \ie if two 2-morphisms in $\vFoam$ agree, then we can find $m$ and 
preimages of the 2-morphisms in $\Ucheck(\glm)^{\geq 0}$ that agree. 
It follows that the same is true for the functors 
$\U(\glm[t])^{\geq 0} \xrightarrow{\vTr(\Phi_{\infty})} \vAFoam_\circ$, so the well-definedness of $\widehat{\Upsilon}_{\cal{V},T}^{\cal{S}}$, 
together with the compatibility of these functors with the various inclusions $\U(\glm[t])^{\geq 0} \hookrightarrow \U(\glnn{m+m'}[t])^{\geq 0}$,
implies that $\Upsilon_{\cal{V},T}^{\cal{S}}$ is well-defined.

\begin{remark}
Repeating this argument with skew-symmetric powers gives functors 
$\widehat{\Upsilon}_{\cal{V},T}^{\wedge}: \U(\glm[t])^{\geq 0} \to \gr\C\Vect^{\Z/2}$ 
and proves the well-definedness of $\Upsilon_{\cal{V},T}^{\wedge}$.
\end{remark}

\subsection{Dotted web categories}\label{sec:dotted}

The functor $\Upsilon_{\cal{V},T}^{\cal{S}}: \vAFoam_\circ \to \gr\C\Vect^{\Z/2}$ is most easily visualized using a description of $\vAFoam_\circ$ 
in terms of a non-negatively graded extension of Cautis-Kamnitzer-Morrison's $\slv$ webs. Indeed, recall that the objects in $\vAFoam_\circ$ 
are rotationally symmetric with respect to the core of the annulus, and the same is true for morphisms after forgetting the decorations on facets.
We can hence describe this category by taking the intersection of a morphism with a radial plane, and remembering the decorations.

It follows that $\vAFoam_\circ$ admits a description as a monoidal category in which objects are sequences 
$\mathbf{k} \in {\displaystyle \coprod_{m>0}} \Z^m_{\geq 0}$ and morphisms are generated by:
\[
\xy
(0,0)*{
\begin{tikzpicture}[scale=.5,rotate=270]
	\draw [very thick] (1.25,0) to (-1.25,0);
	\node at (1.5,0) {\tiny $k$};
	\node at (-1.5,0) {\tiny $k$};
\end{tikzpicture}
};
\endxy 
\quad , \quad
\xy
(0,0)*{
\begin{tikzpicture}[scale=.5,rotate=270]
	\draw [very thick] (2,0) to (.75,0);
	\draw [very thick] (.75,0) to [out=120,in=0] (-.5,.75);
	\draw [very thick] (.75,0) to [out=240,in=0] (-.5,-.75);
	\node at (2.25,0) {\tiny $k{+}l$};
	\node at (-.75,.75) {\tiny $l$};
	\node at (-.75,-.75) {\tiny $k$};
\end{tikzpicture}
};
\endxy
\quad , \quad
\xy
(0,0)*{
\begin{tikzpicture}[scale=.5,rotate=270]
	\draw [very thick] (-2,0) to (-.75,0);
	\draw [very thick] (-.75,0) to [out=60,in=180] (.5,.75);
	\draw [very thick] (-.75,0) to [out=300,in=180] (.5,-.75);
	\node at (-2.25,0) {\tiny $k{+}l$};
	\node at (.75,.75) {\tiny $l$};
	\node at (.75,-.75) {\tiny $k$};
\end{tikzpicture}
};
\endxy 
\quad , \quad
\xy
(0,0)*{
\begin{tikzpicture}[scale=.5,rotate=270]
	\draw [very thick] (1.25,0) to (-1.25,0);
	\node at (1.5,0) {\tiny $k$};
	\node at (-1.5,0) {\tiny $k$};
	\node at (0,.1875) {$\bullet^{f}$};
\end{tikzpicture}
};
\endxy 
\]
where $f \in \C[t_1,\ldots,t_k]^{\mathfrak{S}_k}$. 
Note that we draw these dotted webs vertically, and omit orientations, as they represent morphisms that are built from 2-morphisms in $\vFoam$, 
in contrast to the previously appearing oriented horizontal webs, which are 1-morphisms in $\vFoam$.

The relations in $\vFoam$ imply that the undotted webs satisfy the $q=1$ specialization of the upward relations, 
\ie those in equation \eqref{eq:UpwardRel},
as well as the decoration sliding relation:
\[
\xy
(0,0)*{
\begin{tikzpicture}[scale=.5,rotate=270]
	\draw [very thick] (-2,0) to (-.75,0);
	\draw [very thick] (-.75,0) to [out=60,in=180] (.5,.75);
	\draw [very thick] (-.75,0) to [out=300,in=180] (.5,-.75);
	\node at (-2.25,0) {\tiny $k{+}l$};
	\node at (.75,.75) {\tiny $l$};
	\node at (.75,-.75) {\tiny $k$};
	\node at (-1.25,.325) {$\bullet^{e_r}$};
\end{tikzpicture}
};
\endxy =
\sum_{i=0}^r
\xy
(0,0)*{
\begin{tikzpicture}[scale=.5,rotate=270]
	\draw [very thick] (-2,0) to (-.75,0);
	\draw [very thick] (-.75,0) to [out=60,in=180] (.5,.75);
	\draw [very thick] (-.75,0) to [out=300,in=180] (.5,-.75);
	\node at (-2.25,0) {\tiny $k{+}l$};
	\node at (.75,.75) {\tiny $l$};
	\node at (.75,-.75) {\tiny $k$};
	\node at (.125,-.425) {$\bullet_{e_i}$};
	\node at (0,1.3) {$\bullet^{e_{r-i}}$};
\end{tikzpicture}
};
\endxy
\]
Together, these imply that any decoration can be expressed in terms of decorations on $1$-labeled strands, 
\ie simply as dots labeled by positive integers. 
Using this, the functor $\Upsilon_{\cal{V},T}^\cal{S}$ is determined by sending:
\begin{equation} \label{eq:SplitMerge}
\xy
(0,0)*{
\begin{tikzpicture}[scale=.4,rotate=270]
	\draw [very thick] (2,0) to (.75,0);
	\draw [very thick] (.75,0) to [out=120,in=0] (-.5,.75);
	\draw [very thick] (.75,0) to [out=240,in=0] (-.5,-.75);
	\node at (2.25,0) {\tiny $k{+}l$};
	\node at (-.875,.75) {\tiny $l$};
	\node at (-.875,-.75) {\tiny $k$};
\end{tikzpicture}
};
\endxy \mapsto {_k \iota_l}
\quad , \quad
\xy
(0,0)*{
\begin{tikzpicture}[scale=.4,rotate=270]
	\draw [very thick] (-2,0) to (-.75,0);
	\draw [very thick] (-.75,0) to [out=60,in=180] (.5,.75);
	\draw [very thick] (-.75,0) to [out=300,in=180] (.5,-.75);
	\node at (-2.25,0) {\tiny $k{+}l$};
	\node at (.75,.75) {\tiny $l$};
	\node at (.75,-.75) {\tiny $k$};
\end{tikzpicture}
};
\endxy \mapsto {_k \pi_l}
\quad , \quad
\xy
(0,0)*{
\begin{tikzpicture}[scale=.4,rotate=270]
	\draw [very thick] (1.25,0) to (-1.25,0);
	\node at (1.5,0) {\tiny $1$};
	\node at (-1.5,0) {\tiny $1$};
	\node at (0,.1875) {$\bullet^{r}$};
\end{tikzpicture}
};
\endxy \mapsto T^r 
\end{equation}

\section{Link homologies}\label{sec:identifying}

For each choice of $\cal{V}$ and $T$, and $\square = \cal{S}$ or $\wedge$, 
Theorem \ref{thm:LinkHomologies} produces a homology theory $\overline{\cal{H}}_{\cal{V},T}^\square( \widehat{\beta})$ for colored, annular braid closures $\widehat{\beta}$. 
In this section, we investigate these invariants for various choices of these parameters. 
In particular, we prove the bullet-pointed claims in Theorem \ref{thm:LinkHomologies}, 
showing that certain choices give known invariants (after suitable rescalings), 
some of which descend to invariants of the corresponding links $\cal{L}_\beta \subset S^3$.
We also prove Theorem \ref{thm:negativen}.

\subsection{Khovanov-Rozansky $\sln$ link homology}\label{sec:KhRsln}

In this section we prove the following result, which is the first bullet-pointed claim in Theorem \ref{thm:LinkHomologies}.
\begin{theorem}\label{thm:sln}
Let $\beta$ be a balanced, colored braid, and $\cal{V}_n = q^{1-n} \C[t]/t^n$ with $\deg_q(t) = 2$, 
then $\cal{H}_{\cal{V}_n,t}^{\wedge}( \cal{L}_\beta) := h^{w_\beta}q^{W_\beta - w_\beta - nw_\beta} \overline{\cal{H}}_{\cal{V}_n,t}^{\wedge}( \widehat{\beta})$ 
is isomorphic to colored Khovanov-Rozansky homology, and hence is an invariant of the colored link $\cal{L}_\beta \subset S^3$.
\end{theorem}

\begin{proof}
In \cite[Theorem 6.2]{QR2}, we showed that colored $\sln$ Khovanov-Rozansky homology $\mathsf{KhR}_n( \cal{L}_\beta)$ can be recovered from the
complex $\llbracket \widehat{\beta} \rrbracket_\circ$ using an action of the current algebra $\glm[t]$ on the center of objects in the foam 2-category $\nFoam$.
It thus suffices to show that this current algebra representation 
is isomorphic to that given by $\widehat{\Upsilon}_{\cal{V}_n,t}^\wedge$.

We first quickly recall this description of $\mathsf{KhR}_n( \cal{L}_\beta)$, referring the reader to \cite{QR2} for more details. 
In \cite{QR}, we show that the Khovanov-Rozansky homology of $\cal{L}_\beta$ is the homology of the complex of graded vector spaces
\[
\mathrm{C}_n(\beta) := q^{-s_n(\mathbf{k})} \twoHOM_{\nFoam}\left( I_\mathbf{k}, \llbracket \beta \rrbracket \right)
\]
(up to the writhe-dependent shifts $h^{w_\beta}q^{W_\beta - w_\beta - nw_\beta}$ in homological and quantum degree needed for colored Markov II invariance).
Here $I_\mathbf{k}$ denotes the identity web of the object $\mathbf{k}$ 
corresponding to the coloring of the endpoints of $\beta$, 
and $s_n(\mathbf{k}):=\sum_i k_i (n-k_i)$. 
We can simplify this complex in a manner similar to Theorem \ref{thm:foamObjDec}: 
the simplifications corresponding to web isomorphisms or the homotopy equivalences in equations 
\eqref{eq:RGI1} and \eqref{eq:RGI2} can be applied to $\llbracket \beta \rrbracket$, hence to $\mathrm{C}_n(\beta)$. 
The remaining simplification, web isomorphism via isotopy around the annulus, also has an analogue in this context. 
It suffices to slide webs around the annulus one ladder-rung at a time, and this corresponds to the 
graded vector space isomorphism
\begin{equation}\label{eq:slideshift}
q^{-s_n(\mathbf{k})} \twoHOM_{\nFoam}\left( I_\mathbf{k}, \cal{R} \cdot 
\cal{W}\right) \cong q^{-s_n(\mathbf{k}\pm\epsilon_i)} \twoHOM_{\nFoam}\left( I_{\mathbf{k} \pm \epsilon_i}, \cal{W} \cdot  
\cal{R}\right)
\end{equation}
where $\cal{R} \in \Hom_{\nFoam}(\mathbf{k} \pm \epsilon_i, \mathbf{k})$ is a web consisting of a single ladder rung between the $i^{th}$ and $(i+1)^{st}$ uprights. 
This isomorphism is given via composition with cap and cup foam generators.
Simplifying $\mathrm{C}_n(\beta)$ by applying the analogues of the moves used to simplify $\llbracket \widehat{\beta} \rrbracket$ to $\llbracket \widehat{\beta} \rrbracket_\circ$, 
we end with a complex that is the image of the complex $\llbracket \widehat{\beta} \rrbracket_\circ$ under the functor $\vAFoam \to \gr\C\Vect$ induced from 
the action of $\U(\glm[t])^{\geq 0}$ on the center of objects in $\nFoam$.

Recall from \cite[Section 6]{BGHL} that the center of objects of a linear 2-category is the direct sum of all $\twoEnd$-spaces of identity 1-morphisms.
In the case that our 2-category $\cal{C}$ is graded linear, we instead consider the vector space
\[
\cal{Z}(\cal{C}) := \bigoplus_{\mathbf{c} \in \mathrm{Ob}(\cal{C})} \twoEND_\cal{C}(\one_\mathbf{c})
\]
which gives a graded analogue. 
Here $\twoEND_\cal{C}(\one_\mathbf{c}) : = \twoHOM_\cal{C}(\one_\mathbf{c},\one_\mathbf{c})$. 
As is shown in \cite{BHLW}, any 2-representation $\UU(\glm) \xrightarrow{\Phi} \cal{C}$ induces an action of $\U(\glm[t])$ on $\cal{Z}(\cal{C})$, 
\ie a functor $\U(\glm[t]) \to \gr\C\Vect$ that sends a $\glm$ weight $\mathbf{k}$ to $\twoEND_\cal{C}(\one_{\Phi(\mathbf{k})})$
and sends the generator $x_{i,r}^- \1_\mathbf{k}$ to the composition:
\begin{equation}\label{eq:CoO}
\twoEND_\cal{C}(\one_{\Phi(\mathbf{k})}) \xrightarrow{\alpha \mapsto 
\Phi\Big(
\hspace{-.1cm}
\xy
(0,0)*{
\begin{tikzpicture} [scale=.25]
	\draw[thick,<-] (0,0) to (0,2);
	\node at (-.375,1.5) {\scs$i$};
\end{tikzpicture}};
\endxy \hspace{.1cm}
\Big) \cdot \alpha \cdot
\Phi\Big(
\hspace{-.1cm}
\xy
(0,0)*{
\begin{tikzpicture} [scale=.25]
	\draw[thick,->] (0,0) to (0,2);
	\node at (-.375,.5) {\scs$i$};
\end{tikzpicture}};
\endxy \hspace{.1cm}
\Big)} 
\twoEND_\cal{C}(\Phi(\cal{F}_i\cal{E}_i\one_{\mathbf{k}-\epsilon_i}))
\xrightarrow{\beta \mapsto
\Phi\Big(
\hspace{-.1cm}
\xy
(0,0)*{
\begin{tikzpicture} [scale=.375]
	\draw[thick,<-] (0,0) to [out=90,in=180] (.5,1) to [out=0,in=90] (1,0);
	\node at (1.125,.5) {\scs$\bullet^r$};
	\node at (-.25,.75) {\scs$i$};
\end{tikzpicture}};
\endxy
\hspace{.1cm}
\Big)\circ \beta \circ
\Phi\Big(
\hspace{-.1cm}
\xy
(0,0)*{
\begin{tikzpicture} [scale=.375,yscale=-1]
	\draw[thick,->] (0,0) to [out=90,in=180] (.5,1) to [out=0,in=90] (1,0);
	\node at (-.25,.5) {\scs$i$};
\end{tikzpicture}};
\endxy
\hspace{.1cm}
\Big)}  \twoEND_\cal{C}(\one_{\Phi(\mathbf{k}-\epsilon_i)})
\end{equation}
and \mm for $x_{i,r}^+ \1_\mathbf{k}$.

In our setting, this representation sends a $\glm$-weight $\mathbf{k}$ to 
\[
q^{-s_n(\mathbf{k})}
\twoEND_{\nFoam}\left( 
\xy
(0,0)*{
\begin{tikzpicture} [scale=.6]
\draw [very thick, directed=.65] (2,0) -- (0,0);
\draw [very thick, directed=.65] (2,.75) -- (0,.75);
\node at (2.4,0) {\fns$k_1$};
\node at (2.4,.75) {\fns$k_m$};
\node at (1.125,.55) {$\vdots$};
\end{tikzpicture}};
\endxy
\right)
\cong
q^{-s_n(\mathbf{k})}
\mathrm{H}^*(Gr(k_1,n)) \otimes \cdots \otimes \mathrm{H}^*(Gr(k_m,n)) 
\cong 
\widehat{\Upsilon}_{\cal{V}_n,t}(\mathbf{k})
\]

As mentioned in Section \ref{sec:Foams}, the first isomorphism follows from \cite[Remark 4.1]{QR}.
The second isomorphism is given as follows. Recall that $\mathrm{H}^\bullet(Gr(k,n))$ has a basis consisting of classes of 
Schur polynomials $s_\mu$, where $\mu$ is a partition that fits in a $k \times (n-k)$ box 
(\ie $\mu$ has at most $k$ parts, and $\mu_1 \leq n-k$). The isomorphism is then given on each factor by the isomorphism
\begin{equation}\label{eq:GrassToWedge}
\mathrm{H}^\bullet(Gr(k,n)) \cong \bV^k \C[t]/t^n
\end{equation}
given by
\[
s_\mu \longmapsto \sum_{\sigma \in \mathfrak{S}_k} (-1)^{\mathrm{sgn}(\sigma)} \sigma(t^{\mu_1+k-1} \otimes t^{\mu_2+k-2} \otimes \cdots \otimes t^{\mu_k}) 
\in \bV^k \C[t]/t^n.
\]

It remains to show that the actions of $x_{i,r}^\pm \1_\mathbf{k}$ in this representation agree with $\widehat{\Upsilon}_{\cal{V}_n,t}(x_{i,r}^\pm \1_\mathbf{k})$. 
In fact, it suffices to show that they agree\footnote{Indeed, if they disagree, we could fix this by rescaling our isomorphism on highest weight vectors in the 
relevant (quotients of) Weyl modules into which our representations decompose. 
See \cite{BHLW} for details on the relation between Weyl modules and traces of categorified quantum groups.} 
up to a sign, which is what we will show. 
Moreover, we only consider the case of $x_{i,r}^+ \1_\mathbf{k}$, since the argument for 
$x_{i,r}^- \1_\mathbf{k}$ is completely analogous.
Consider the action of $x_{i,r}^+ \1_\mathbf{k}$, which is given locally on the $i^{th}$ and $(i+1)^{st}$ tensor factors via
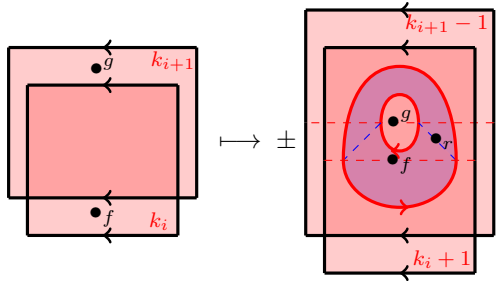
\[
\xy
(0,0)*{
\begin{tikzpicture} [scale=.5,fill opacity=0.2]
\path[fill=red] (2,-2) to (-2,-2) to (-2,2) to (2,2);
\path[fill=red] (2.5,-1) to (-2.5,-1) to (-2.5,3) to (2.5,3);
\draw[very thick, directed=.5] (2.5,-1) to (-2.5,-1);
\draw[very thick] (2.5,-1) to (2.5,3);
\draw[very thick] (-2.5,-1) to (-2.5,3);
\draw[very thick, directed=.5] (2.5,3) to (-2.5,3);
\node[opacity=1] at (0,2.5) {$\bullet^g$};
\draw[very thick, directed=.5] (2,-2) to (-2,-2);
\draw[very thick] (2,-2) to (2,2);
\draw[very thick] (-2,-2) to (-2,2);
\draw[very thick, directed=.5] (2,2) to (-2,2);
\node[red, opacity=1] at (1.5,-1.625) {\fns$k_i$};
\node[red, opacity=1] at (1.875,2.625) {\fns$k_{i+1}$};
\node[opacity=1] at (0,-1.5) {$\bullet_f$};
\end{tikzpicture}
}
\endxy
\; \longmapsto \;
\pm \;
\xy
(0,0)*{
\begin{tikzpicture} [scale=.5,fill opacity=0.2]
\begin{scope}[shift={(0,-3)}]
\path[fill=red] (2,-2) to (-2,-2) to (-2,1) to (2,1);
\path[fill=blue] (1.5,1) to [out=270,in=0] (0,-.25) to [out=180,in=270] (-1.5,1) to
	(-.5,2) to [out=270,in=180] (0,1.25) to [out=0,in=270] (.5,2);
\path[fill=red] (2.5,-1) to (-2.5,-1) to (-2.5,2) to (2.5,2);
\draw[very thick, directed=.5] (2.5,-1) to (-2.5,-1);
\draw[very thick] (2.5,-1) to (2.5,2);
\draw[very thick] (-2.5,-1) to (-2.5,2);
\draw[very thick, red, rdirected=.5] (-.5,2) to [out=270,in=180] (0,1.25)
	to [out=0,in=270] (.5,2);
\draw[very thick, directed=.5] (2,-2) to (-2,-2);
\draw[very thick] (2,-2) to (2,1);
\draw[very thick] (-2,-2) to (-2,1);
\draw[very thick, red, rdirected=.5] (1.5,1) to [out=270,in=0] (0,-.25)
	to [out=180,in=270] (-1.5,1);
\node[red, opacity=1] at (1.125,-1.625) {\fns$k_i+1$};
\end{scope}
\path[fill=red] (2.5,2) to (-2.5,2) to (-2.5,-1) to (2.5,-1);
\path[fill=blue] (1.5,-2) to [out=90,in=0] (0,.5) to [out=180,in=90] (-1.5,-2) to
	(-.5,-1) to [out=90,in=180] (0,-.25) to [out=0,in=90] (.5,-1);
\path[fill=red] (2,1) to (-2,1) to (-2,-2) to (2,-2);
\draw[very thick, directed=.5] (2.5,2) to (-2.5,2);
\draw[very thick] (2.5,2) to (2.5,-1);
\draw[very thick] (-2.5,2) to (-2.5,-1);
\draw[dashed, red] (2.5,-1) to (-2.5,-1);
\draw[very thick, red] (-.5,-1) to [out=90,in=180] (0,-.25)
	to [out=0,in=90] (.5,-1);
\node[opacity=1] at (0,-.9) {$\bullet^g$};
\draw[dashed, blue] (1.5,-2) to (.5,-1);
\draw[dashed, blue] (-.5,-1) to (-1.5,-2);
\node[opacity=1] at (1.125,-1.5) {$\bullet_r$};
\draw[very thick, directed=.5] (2,1) to (-2,1);
\draw[very thick] (2,1) to (2,-2);
\draw[very thick] (-2,1) to (-2,-2);
\draw[dashed, red] (2,-2) to (-2,-2);
\draw[very thick, red] (1.5,-2) to [out=90,in=0] (0,.5)
	to [out=180,in=90] (-1.5,-2);
\node[red, opacity=1] at (1.25,1.625) {\fns$k_{i+1}-1$};
\node[opacity=1] at (0,-2.1) {$\bullet_f$};
\end{tikzpicture}
}
\endxy
\]

More topologically, the action of $x_{i,r}^+ \1_\mathbf{k}$ is given (up to a factor of $\pm1$)
by sewing an $r$-dotted, $1$-labeled ribbon along singular seams parallel to the boundaries of the two sheets, 
and adjusting the boundary labelings. 
From this latter description,
it is clear that this can be expressed as the composition of three maps: the first sews an undotted $1$-labeled ribbon to the back sheet, producing an endomorphism of 
parallel 1-labeled and $(k_{i+1}{-}1)$-labeled strands. The second adds $r$ dots to the $1$-labeled portion, and the third sews a $1$-labeled ribbon between the front sheet and the $1$-labeled boundary. 

The foam relations imply that
any endomorphism of parallel strands can be written as a linear combination of decorations on vertical sheets, 
so we can assume that we have performed this simplification after applying the first of the three maps. It follows that the second map is simply 
the tensor product of identity maps with
the endomorphism of a 1-labeled strand given by an $r$-dotted sheet, 
and the third map is given by sewing in a $1$-labeled ribbon between 
decorated $k_i$-labeled and $1$-labeled sheets.
We have thus written the action of $x_{i,r}^+ \1_\mathbf{k}$ as a composition:
\begin{equation}\label{eq:slnComp}
\bV^{k_i} \cal{V}_n \otimes \bV^{k_{i+1}} \cal{V}_n \rightarrow \bV^{k_i} \cal{V}_n \otimes \cal{V}_n \otimes \bV^{k_{i+1}-1} \cal{V}_n
\rightarrow \bV^{k_i} \cal{V}_n \otimes \cal{V}_n \otimes \bV^{k_{i+1}-1} \cal{V}_n \rightarrow
\bV^{k_i + 1} \cal{V}_n \otimes \bV^{k_{i+1}-1} \cal{V}_n
\end{equation}
where the middle map is $\id \otimes t^r \otimes \id$. 
Recall from equation \eqref{eq:FunctorCAGens} that $\widehat{\Upsilon}_{\cal{V}_n,t}(x_{i,r}^+ \1_\mathbf{k})$ factors in exactly the same way, 
with the middle map given by $t^r$ (for our current choice of $T$). 
It hence suffices to show that the first and third maps above agree with those in equation \eqref{eq:FunctorCAGens},
up to $\pm1$.

We begin with the third map. To help with identifying it, we first note that the isomorphism from equation \eqref{eq:GrassToWedge} can be 
explicitly realized in $\nFoam$ by applying the ``blister'' relations from \cite{QR} to a sheet decorated by the Schur polynomial $s_\mu$ with $\mu=(\mu_1,\ldots,\mu_k)$. 
The first such relation: 
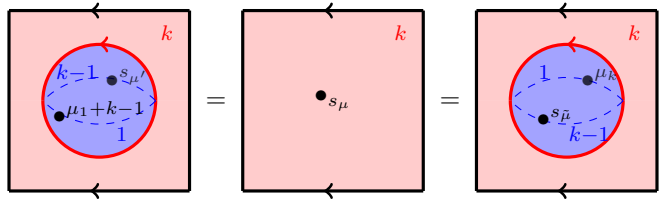
\[
\xy
(0,0)*{
\begin{tikzpicture} [scale=.6,fill opacity=0.2]
	\path[fill=blue] (1.25,0) to [out=90,in=0] (0,1.25) to [out=180,in=90] (-1.25,0) to [out=270,in=180] (0,-1.25) to [out=0,in=270] (1.25,0);
	\path[fill=blue] (1.25,0) to [out=90,in=0] (0,1.25) to [out=180,in=90] (-1.25,0) to [out=270,in=180] (0,-1.25) to [out=0,in=270] (1.25,0);
	\path[fill=red] (2,0) to (1.25,0) to [out=90,in=0] (0,1.25) to [out=180,in=90] (-1.25,0) to (-2,0) to (-2,2) to (2,2);
	\path[fill=red] (2,0) to (1.25,0) to [out=270,in=0] (0,-1.25) to [out=180,in=270] (-1.25,0) to (-2,0) to (-2,-2) to (2,-2);	
	\draw[very thick, directed=.55] (2,-2) to (-2,-2);
	\draw[very thick] (2,-2) to (2,2);
	\draw[very thick] (-2,-2) to (-2,2);
	\draw[very thick, red, directed=.5] (1.25,0) to [out=90,in=0] (0,1.25) to [out=180,in=90] (-1.25,0);
	\draw[very thick, red] (1.25,0) to [out=270,in=0] (0,-1.25) to [out=180,in=270] (-1.25,0);	
	\draw[blue,dashed] (1.25,0) to [out=225,in=315] (-1.25,0);
	\draw[dashed,blue] (1.25,0) to [out=135,in=45] (-1.25,0);
	\draw[very thick, directed=.55] (2,2) to (-2,2);
	\node[red,opacity=1] at (1.5,1.5) {\fns$k$};
	\node[blue,opacity=1] at (.5,-.75) {\fns$1$};
	\node[blue,opacity=1] at (-.5,.625) {\fns$k{-}1$};
	\node[opacity=1] at (0,-.25) {$\bullet^{\mu_1{+}k{-}1}$};
	\node[opacity=.75] at (.625,.5) {$\bullet^{s_{\mu'}}$};
\end{tikzpicture}
};
\endxy
\; = \;
\xy
(0,0)*{
\begin{tikzpicture} [scale=.6,fill opacity=0.2]
	\path[fill=red] (2,2) to (2,-2) to (-2,-2) to (-2,2); 
	\draw[very thick, directed=.55] (2,-2) to (-2,-2);
	\draw[very thick] (2,-2) to (2,2);
	\draw[very thick] (-2,-2) to (-2,2);
	\draw[very thick, directed=.55] (2,2) to (-2,2);
	\node[red,opacity=1] at (1.5,1.5) {\fns$k$};
	\node[opacity=1] at (0,0) {$\bullet_{s_\mu}$};
\end{tikzpicture}
};
\endxy
\; = \;
\xy
(0,0)*{
\begin{tikzpicture} [scale=.6,fill opacity=0.2]
	\path[fill=blue] (1.25,0) to [out=90,in=0] (0,1.25) to [out=180,in=90] (-1.25,0) to [out=270,in=180] (0,-1.25) to [out=0,in=270] (1.25,0);
	\path[fill=blue] (1.25,0) to [out=90,in=0] (0,1.25) to [out=180,in=90] (-1.25,0) to [out=270,in=180] (0,-1.25) to [out=0,in=270] (1.25,0);
	\path[fill=red] (2,0) to (1.25,0) to [out=90,in=0] (0,1.25) to [out=180,in=90] (-1.25,0) to (-2,0) to (-2,2) to (2,2);
	\path[fill=red] (2,0) to (1.25,0) to [out=270,in=0] (0,-1.25) to [out=180,in=270] (-1.25,0) to (-2,0) to (-2,-2) to (2,-2);	
	\draw[very thick, directed=.55] (2,-2) to (-2,-2);
	\draw[very thick] (2,-2) to (2,2);
	\draw[very thick] (-2,-2) to (-2,2);
	\draw[very thick, red, directed=.5] (1.25,0) to [out=90,in=0] (0,1.25) to [out=180,in=90] (-1.25,0);
	\draw[very thick, red] (1.25,0) to [out=270,in=0] (0,-1.25) to [out=180,in=270] (-1.25,0);	
	\draw[blue,dashed] (1.25,0) to [out=225,in=315] (-1.25,0);
	\draw[dashed,blue] (1.25,0) to [out=135,in=45] (-1.25,0);
	\draw[very thick, directed=.55] (2,2) to (-2,2);
	\node[red,opacity=1] at (1.5,1.5) {\fns$k$};
	\node[blue,opacity=1] at (.5,-.75) {\fns$k{-}1$};
	\node[blue,opacity=1] at (-.5,.625) {\fns$1$};
	\node[opacity=1] at (-.25,-.375) {$\bullet^{s_{\tilde{\mu}}}$}; 
	\node[opacity=.75] at (.75,.5) {$\bullet^{\mu_k}$};
\end{tikzpicture}
};
\endxy
\]
expresses such a decorated sheet in terms of a blister with a dotted 1-labeled sheet and 
a $(k-1)$-labeled sheet decorated by a Schur polynomial, either corresponding to the partition
$\mu'=(\mu_2,\ldots,\mu_k)$ or $\tilde{\mu} = (\mu_1+1,\ldots,\mu_{k-1}+1)$ 
depending on whether the $1$-labeled sheet is in front or back, respectively.
Repeated application expresses a sheet decorated by $s_\mu$
as a sheet with an ``iterated blister'' containing dotted 1-labeled sheets,
where the $i^{th}$ blister sheet from the front carries $\mu_i+k-i$ dots. 
The second blister relation:
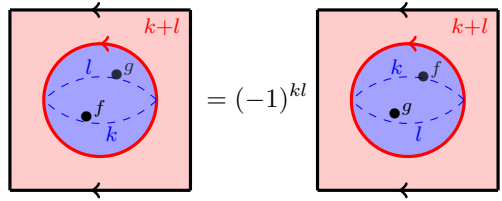
\[
\xy
(0,0)*{
\begin{tikzpicture} [scale=.6,fill opacity=0.2]
	\path[fill=blue] (1.25,0) to [out=90,in=0] (0,1.25) to [out=180,in=90] (-1.25,0) to [out=270,in=180] (0,-1.25) to [out=0,in=270] (1.25,0);
	\path[fill=blue] (1.25,0) to [out=90,in=0] (0,1.25) to [out=180,in=90] (-1.25,0) to [out=270,in=180] (0,-1.25) to [out=0,in=270] (1.25,0);
	\path[fill=red] (2,0) to (1.25,0) to [out=90,in=0] (0,1.25) to [out=180,in=90] (-1.25,0) to (-2,0) to (-2,2) to (2,2);
	\path[fill=red] (2,0) to (1.25,0) to [out=270,in=0] (0,-1.25) to [out=180,in=270] (-1.25,0) to (-2,0) to (-2,-2) to (2,-2);	
	\draw[very thick, directed=.55] (2,-2) to (-2,-2);
	\draw[very thick] (2,-2) to (2,2);
	\draw[very thick] (-2,-2) to (-2,2);
	\draw[very thick, red, directed=.5] (1.25,0) to [out=90,in=0] (0,1.25) to [out=180,in=90] (-1.25,0);
	\draw[very thick, red] (1.25,0) to [out=270,in=0] (0,-1.25) to [out=180,in=270] (-1.25,0);	
	\draw[blue,dashed] (1.25,0) to [out=225,in=315] (-1.25,0);
	\draw[dashed,blue] (1.25,0) to [out=135,in=45] (-1.25,0);
	\draw[very thick, directed=.55] (2,2) to (-2,2);
	\node[red,opacity=1] at (1.375,1.625) {\fns$k{+}l$};
	\node[blue,opacity=1] at (.25,-.75) {\fns$k$};
	\node[blue,opacity=1] at (-.25,.75) {\fns$l$};
	\node[opacity=1] at (-.15,-.25) {$\bullet^f$};
	\node[opacity=.75] at (.5,.625) {$\bullet^g$};
\end{tikzpicture}
};
\endxy
\; = (-1)^{kl} \;
\xy
(0,0)*{
\begin{tikzpicture} [scale=.6,fill opacity=0.2]
	\path[fill=blue] (1.25,0) to [out=90,in=0] (0,1.25) to [out=180,in=90] (-1.25,0) to [out=270,in=180] (0,-1.25) to [out=0,in=270] (1.25,0);
	\path[fill=blue] (1.25,0) to [out=90,in=0] (0,1.25) to [out=180,in=90] (-1.25,0) to [out=270,in=180] (0,-1.25) to [out=0,in=270] (1.25,0);
	\path[fill=red] (2,0) to (1.25,0) to [out=90,in=0] (0,1.25) to [out=180,in=90] (-1.25,0) to (-2,0) to (-2,2) to (2,2);
	\path[fill=red] (2,0) to (1.25,0) to [out=270,in=0] (0,-1.25) to [out=180,in=270] (-1.25,0) to (-2,0) to (-2,-2) to (2,-2);	
	\draw[very thick, directed=.55] (2,-2) to (-2,-2);
	\draw[very thick] (2,-2) to (2,2);
	\draw[very thick] (-2,-2) to (-2,2);
	\draw[very thick, red, directed=.5] (1.25,0) to [out=90,in=0] (0,1.25) to [out=180,in=90] (-1.25,0);
	\draw[very thick, red] (1.25,0) to [out=270,in=0] (0,-1.25) to [out=180,in=270] (-1.25,0);	
	\draw[blue,dashed] (1.25,0) to [out=225,in=315] (-1.25,0);
	\draw[dashed,blue] (1.25,0) to [out=135,in=45] (-1.25,0);
	\draw[very thick, directed=.55] (2,2) to (-2,2);
	\node[red,opacity=1] at (1.375,1.625) {\fns$k{+}l$};
	\node[blue,opacity=1] at (.25,-.75) {\fns$l$};
	\node[blue,opacity=1] at (-.25,.75) {\fns$k$};
	\node[opacity=1] at (-.15,-.25) {$\bullet^g$};
	\node[opacity=.75] at (.5,.625) {$\bullet^f$};
\end{tikzpicture}
};
\endxy
\]
allows us to identify such a foam with the basis vector 
$\sum_{\sigma \in \mathfrak{S}_k} (-1)^{\mathrm{sgn}(\sigma)} \sigma(t^{\mu_1+k-1} \otimes t^{\mu_2+k-2} \otimes \cdots \otimes t^{\mu_k})$ in $\bV^k \C[t]/t^n$.
Using this identification, it immediately follows that the map 
$\bV^{k_i} \C[t]/t^n \otimes \C[t]/t^n \rightarrow \bV^{k_i + 1} \C[t]/t^n$ given by 
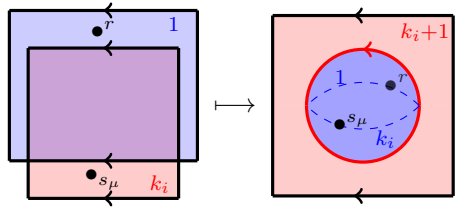
\[
\xy
(0,0)*{
\begin{tikzpicture} [scale=.5,fill opacity=0.2]
\path[fill=red] (2,-2) to (-2,-2) to (-2,2) to (2,2);
\path[fill=blue] (2.5,-1) to (-2.5,-1) to (-2.5,3) to (2.5,3);
\draw[very thick, directed=.5] (2.5,-1) to (-2.5,-1);
\draw[very thick] (2.5,-1) to (2.5,3);
\draw[very thick] (-2.5,-1) to (-2.5,3);
\draw[very thick, directed=.5] (2.5,3) to (-2.5,3);
\node[opacity=1] at (0,2.5) {$\bullet^r$};
\draw[very thick, directed=.5] (2,-2) to (-2,-2);
\draw[very thick] (2,-2) to (2,2);
\draw[very thick] (-2,-2) to (-2,2);
\draw[very thick, directed=.5] (2,2) to (-2,2);
\node[red, opacity=1] at (1.5,-1.625) {\fns$k_i$};
\node[blue, opacity=1] at (1.875,2.625) {\fns$1$};
\node[opacity=1] at (0,-1.5) {$\bullet_{s_\mu}$};
\end{tikzpicture}
}
\endxy
\; \longmapsto \;
\xy
(0,0)*{
\begin{tikzpicture} [scale=.6,fill opacity=0.2]
	\path[fill=blue] (1.25,0) to [out=90,in=0] (0,1.25) to [out=180,in=90] (-1.25,0) to [out=270,in=180] (0,-1.25) to [out=0,in=270] (1.25,0);
	\path[fill=blue] (1.25,0) to [out=90,in=0] (0,1.25) to [out=180,in=90] (-1.25,0) to [out=270,in=180] (0,-1.25) to [out=0,in=270] (1.25,0);
	\path[fill=red] (2,0) to (1.25,0) to [out=90,in=0] (0,1.25) to [out=180,in=90] (-1.25,0) to (-2,0) to (-2,2) to (2,2);
	\path[fill=red] (2,0) to (1.25,0) to [out=270,in=0] (0,-1.25) to [out=180,in=270] (-1.25,0) to (-2,0) to (-2,-2) to (2,-2);	
	\draw[very thick, directed=.55] (2,-2) to (-2,-2);
	\draw[very thick] (2,-2) to (2,2);
	\draw[very thick] (-2,-2) to (-2,2);
	\draw[very thick, red, directed=.5] (1.25,0) to [out=90,in=0] (0,1.25) to [out=180,in=90] (-1.25,0);
	\draw[very thick, red] (1.25,0) to [out=270,in=0] (0,-1.25) to [out=180,in=270] (-1.25,0);	
	\draw[blue,dashed] (1.25,0) to [out=225,in=315] (-1.25,0);
	\draw[dashed,blue] (1.25,0) to [out=135,in=45] (-1.25,0);
	\draw[very thick, directed=.55] (2,2) to (-2,2);
	\node[red,opacity=1] at (1.375,1.625) {\fns$k_i{+}1$};
	\node[blue,opacity=1] at (.5,-.75) {\fns$k_i$};
	\node[blue,opacity=1] at (-.5,.625) {\fns$1$};
	\node[opacity=1] at (-.25,-.375) {$\bullet^{s_\mu}$}; 
	\node[opacity=.75] at (.75,.5) {$\bullet^r$};
\end{tikzpicture}
};
\endxy
\]
is equal to $k_i+1$ times the canonical projection.
This determines the third map in equation \eqref{eq:slnComp} 
and verifies that it equals ${_{k_i}\pi_{1}} \otimes \id_{k_{i+1}-1}$, as desired.

We now turn to the first map in equation \eqref{eq:slnComp}, which is determined by the map 
$\bV^{k_{i+1}} \cal{V}_n \rightarrow \cal{V}_n \otimes \bV^{k_{i+1}-1} \cal{V}_n$
given by:
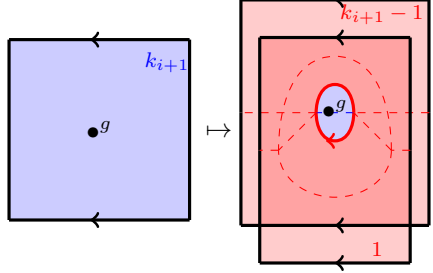
\[
\xy
(0,0)*{
\begin{tikzpicture} [scale=.6,fill opacity=0.2]
	\path[fill=blue] (2,2) to (2,-2) to (-2,-2) to (-2,2); 
	\draw[very thick, directed=.55] (2,-2) to (-2,-2);
	\draw[very thick] (2,-2) to (2,2);
	\draw[very thick] (-2,-2) to (-2,2);
	\draw[very thick, directed=.55] (2,2) to (-2,2);
	\node[blue,opacity=1] at (1.5,1.5) {\fns$k_{i+1}$};
	\node[opacity=1] at (0,0) {$\bullet^g$};
\end{tikzpicture}
};
\endxy
\mapsto
\xy
(0,0)*{
\begin{tikzpicture} [scale=.5,fill opacity=0.2]
\begin{scope}[shift={(0,-3)}]
\path[fill=red] (2,1) to (1.5,1) to [out=270,in=0] (0,-.25) to [out=180,in=270] (-1.5,1) 
	to (-2,1) to (-2,-2) to (2,-2) to (2,1);
\path[fill=red] (1.5,1) to [out=270,in=0] (0,-.25) to [out=180,in=270] (-1.5,1) to
	(-.5,2) to [out=270,in=180] (0,1.25) to [out=0,in=270] (.5,2);
\path[fill=red] (-2.5,2) to (-.5,2) to [out=270,in=180] (0,1.25)
	to [out=0,in=270] (.5,2) to (2.5,2) to (2.5,-1) to (-2.5,-1);
\path[fill=blue] (-.5,2) to [out=270,in=180] (0,1.25)
	to [out=0,in=270] (.5,2);
\draw[very thick, directed=.5] (2.5,-1) to (-2.5,-1);
\draw[very thick] (2.5,-1) to (2.5,2);
\draw[very thick] (-2.5,-1) to (-2.5,2);
\draw[very thick, red, rdirected=.5] (-.5,2) to [out=270,in=180] (0,1.25)
	to [out=0,in=270] (.5,2);
\draw[very thick, directed=.5] (2,-2) to (-2,-2);
\draw[very thick] (2,-2) to (2,1);
\draw[very thick] (-2,-2) to (-2,1);
\draw[dashed, red] (1.5,1) to [out=270,in=0] (0,-.25)
	to [out=180,in=270] (-1.5,1);
\node[red, opacity=1] at (1.125,-1.625) {\fns$1$};
\end{scope}
\path[fill=red] (2.5,2) to (-2.5,2) to (-2.5,-1) to (-.5,-1) to [out=90,in=180] (0,-.25)
	to [out=0,in=90] (.5,-1) to (2.5,-1);
\path[fill=blue] (-.5,-1) to [out=90,in=180] (0,-.25)
	to [out=0,in=90] (.5,-1);
\path[fill=red] (1.5,-2) to [out=90,in=0] (0,.5) to [out=180,in=90] (-1.5,-2) to
	(-.5,-1) to [out=90,in=180] (0,-.25) to [out=0,in=90] (.5,-1);
\path[fill=red] (2,-2) to (1.5,-2) to [out=90,in=0] (0,.5)
	to [out=180,in=90] (-1.5,-2) to (-2,-2) to (-2,1) to (2,1);
\draw[very thick, directed=.5] (2.5,2) to (-2.5,2);
\draw[very thick] (2.5,2) to (2.5,-1);
\draw[very thick] (-2.5,2) to (-2.5,-1);
\draw[dashed, red] (2.5,-1) to (.5,-1);
\draw[dashed, blue] (.5,-1) to (-.5,-1);
\draw[dashed, red] (-.5,-1) to (-2.5,-1);
\draw[very thick, red] (-.5,-1) to [out=90,in=180] (0,-.25)
	to [out=0,in=90] (.5,-1);
\node[opacity=1] at (0,-.9) {$\bullet^g$};
\draw[dashed, red] (1.5,-2) to (.5,-1);
\draw[dashed, red] (-.5,-1) to (-1.5,-2);
\draw[very thick, directed=.5] (2,1) to (-2,1);
\draw[very thick] (2,1) to (2,-2);
\draw[very thick] (-2,1) to (-2,-2);
\draw[dashed, red] (2,-2) to (1.5,-2);
\draw[dashed, red] (-1.5,-2) to (-2,-2);
\draw[dashed, red] (1.5,-2) to [out=90,in=0] (0,.5)
	to [out=180,in=90] (-1.5,-2);
\node[red, opacity=1] at (1.25,1.625) {\fns$k_{i+1}-1$};
\end{tikzpicture}
}
\endxy
\]
To compute this map, note that its domain is the lowest weight space for the $\slnn{2}$ representation given by 
considering the $t$-degree zero part of the action of $\U(\glnn{2}[t])^{\geq 0}$ 
on the center of objects of the 2-subcategory of
$\nFoam$ whose objects have only two entries that sum to $k_{i+1}$. 
The map itself is the action of the standard Chevalley generator $E \in \slnn{2}$. 
The corresponding action of the Chevalley generator $F \in \slnn{2}$ is a map in the opposite direction
$\cal{V}_n \otimes \bV^{k_{i+1}-1} \cal{V}_n \rightarrow \bV^{k_{i+1}} \cal{V}_n$ and is exactly as given above, 
hence is $k_{i+1}$ times the canonical projection.

The isotypic decomposition shows that
\[
F = \begin{pmatrix} F' & 0 \end{pmatrix}
\]
when restricted to this weight space, 
where $F'$ is an isomorphism from the subspace of non-lowest weight vectors in the $(2-k_{i+1})^{st}$ weight space to the $(-k_{i+1})^{st}$ 
weight space. We must have $-k_{i+1} \cdot \id = EF - FE = 0 - FE$, so it follows that 
\[
E = k_{i+1} \begin{pmatrix} (F')^{-1} \\ 0 \end{pmatrix}
\]
and, in particular, is completely determined by the action of $F$ on this representation. 
We have already seen that $F$ is $k_{i+1}$ times the canonical projection, hence it
agrees with $\widehat{\Upsilon}_{\cal{V}_n,t}^\wedge(x_{1,0}^- \1_{[1,k_{i+1}-1]})$.
It follows that $E$ agrees with $\widehat{\Upsilon}_{\cal{V}_n,t}^\wedge(x_{1,0}^+\1_{[0,k_{i+1}]})$ and therefore is the canonical inclusion, as desired.
\end{proof}

\begin{remark}\label{rem:KhRonthenose}
Given our conventions for the $\sln$ link polynomials in Section \ref{sec:Decat}, 
we must first negate the $q$-degree of $\cal{H}_{\cal{V}_n,t}^{\wedge}( \cal{L}_\beta) \cong \mathsf{KhR}_n(\cal{L})$ before decategorifying to obtain $P_n(\cal{L})$ on the nose. 
In other words, we precisely have $\chi(\mathsf{KhR}_n(\cal{L}) \big|_{q \leftrightarrow q^{-1}}) = P_n(\cal{L})$, 
where $\chi$ denotes taking the alternating sum of graded dimensions. 
We will use the notation $\mathsf{KhR}^{\vee}_n(\cal{L})$ to denote $\mathsf{KhR}_n(\cal{L})$ with its $q$-degree negated; 
this then gives $\chi(\mathsf{KhR}^{\vee}_n(\cal{L})) = P_n(\cal{L})$.
\end{remark}

\subsection{HOMFLYPT link homology}\label{sec:TriplyGraded}

Let $\HHH(\cal{L})$ denote the triply-graded categorification of the HOMFLYPT polynomial introduced in \cite{KhR2}.
In this section, we prove the following result, which is the second bullet-pointed claim in Theorem \ref{thm:LinkHomologies}.

\begin{theorem}\label{thm:HOMFLYPT}
Let $\beta$ be a balanced, colored braid 
and $\cal{V}_\infty = qa^{-1} \C[t,\theta]/\theta^2$,
where $\C[t,\theta]/\theta^2$ is a bi-graded superalgebra such that $t$ is an even variable with $\deg_{q,a}(t) = (2,0)$ and
$\theta$ is an odd variable with $\deg_{q,a}(\theta) = (0,2)$,
then $\cal{H}_{\cal{V}_\infty,t}^{\cal{S}}( \cal{L}_\beta) := a^{w_\beta} q^{W_\beta - w_\beta} \overline{\cal{H}}_{\cal{V}_\infty,t}^{\cal{S}}( \widehat{\beta})$
is isomorphic to the colored variant of Khovanov-Rozansky's triply-graded HOMFLYPT link homology $\HHH(\cal{L}_\beta)$, 
and hence is an invariant of the colored link $\cal{L}_\beta \subset S^3$.
\end{theorem}

Recall that work of Khovanov \cite{KhHHH} describes $\HHH(\cal{L}_\beta)$ in terms of the Hochschild homology of Soergel bimodules. 
The latter are bimodules over a polynomial ring introduced by Soergel \cite{Soergel1,Soergel2} in his study of category $\cal{O}$.
We quickly summarize this construction: in \cite{RouBraid}, Rouquier assigns a complex of Soergel bimodules 
to each braid group generator, and shows that the tensor product of such complexes corresponding to a product of generators gives 
an invariant of the resulting braid $\beta$. 
Khovanov shows that the HOMFLYPT link homology $\HHH(\cal{L}_\beta)$ can be obtained from the complex of Soergel bimodules assigned to $\beta$, 
by first taking the Hochschild homology of the terms in the complex to obtain a complex of bi-graded vector spaces, and then taking homology. 
Moreover, this description extends to colored braids and links \cite{WebWil,Cautis2} if we work 
in the more-general context of singular Soergel bimodules \cite{GeoWil}.

For our considerations, the details of the above construction are most easily understood in terms of 2-representations of 
categorified quantum groups and foam 2-categories. 
Indeed, as explained in \cite{Web5}, 
the Khovanov-Lauda 2-representation of $\UU(\slm)$ built from the equivariant cohomology 
of partial flag varieties \cite{KL3,Lau2}
specifies a 2-representation of $\UU(\glm)$ on the 2-category of singular Soergel bimodules. 
Recall that the latter is the idempotent completion of the 2-category $\mathbf{BSBim}$ of Bott-Samelson bimodules, defined as follows.

Objects in $\mathbf{BSBim}$ are as in $\vFoam$, that is, sequences $\mathbf{k} = [k_1,\ldots,k_m]$ with $k_i > 0$. 
The $\Hom$-categories in $\mathbf{BSBim}$ are defined as follows. 
Given such a sequence $\mathbf{k}$, set $N = \sum k_i$ and consider the subgroup 
$\mathfrak{S}_{\mathbf{k}} : = \mathfrak{S}_{k_1} \times \cdots \times \mathfrak{S}_{k_m}$ of the symmetric group $\mathfrak{S}_N$.
Let $R^\mathbf{k} := \C[t_1,\ldots,t_N]^{\mathfrak{S}_{\mathbf{k}}} \subseteq \C[t_1,\ldots,t_N] =: R$, and note that if 
$\mathfrak{S}_{\mathbf{k}} \leq \mathfrak{S}_{\mathbf{p}}$ then $R^\mathbf{p} \subseteq R^\mathbf{k}$, so the latter is a module over the former.
We define $\Hom_{\mathbf{BSBim}}(\mathbf{k},\mathbf{k}')$ to be the additive category generated by 
$(R^{\mathbf{k}'},R^\mathbf{k})$-bimodules 
of the form
\[
q^d R^{\mathbf{k}'}\otimes_{R^{\mathbf{p}'}} \cdots \otimes_{R^\mathbf{p}} R^\mathbf{k}
\]
when the sum of the entries in $\mathbf{k}$ and $\mathbf{k}'$ agree (and to be trivial otherwise).
Here, as before, the power of $q$ denotes a shift in degree, and $\deg_q(t_i)=2$.
As in the case of $\vFoam$, this 2-category admits a tensor product\footnote{The ``other'' tensor product, 
over the rings $R^\mathbf{k}$, is the horizontal composition in $\mathbf{BSBim}$.},
given on objects by concatenating sequences, and on $1$- and $2$-morphisms by taking the tensor product over $\C$.

The results in \cite{Web5} further imply that
the Khovanov-Lauda 2-functor $\UU(\glm) \xrightarrow{\Phi_{BS}} \mathbf{BSBim}$ factors through $\vFoam$, 
giving a monoidal 2-functor $\vFoam \xrightarrow{\Psi_{BS}} \mathbf{BSBim}$ that we now describe. 
It is the identity on objects, and is given on generating webs as follows:
\[
\xy
(0,0)*{
\begin{tikzpicture}[scale=.5,rotate=-45]
	\draw [very thick, ->] (1.25,0) to (-1.25,0);
	\node at (1.5,0) {\tiny $k$};
	\node at (-1.5,0) {\tiny $k$};
\end{tikzpicture}
};
\endxy 
\mapsto R^{[k]} \;\; , \;\;
\xy
(0,0)*{
\begin{tikzpicture}[scale=.5,rotate=-45]
	\draw [very thick, directed=.55] (2,0) to (.75,0);
	\draw [very thick,directed=.55] (.75,0) to [out=120,in=0] (-.5,.75);
	\draw [very thick,directed=.55] (.75,0) to [out=240,in=0] (-.5,-.75);
	\node at (2.25,0) {\tiny $k{+}l$};
	\node at (-.75,.75) {\tiny $l$};
	\node at (-.75,-.75) {\tiny $k$};
\end{tikzpicture}
};
\endxy \mapsto q^{-kl} R^{[k,l]}  \otimes_{R^{[k+l]}} R^{[k+l]}
\;\; , \;\;
\xy
(0,0)*{
\begin{tikzpicture}[scale=.5,rotate=-45]
	\draw [very thick, rdirected=.55] (-2,0) to (-.75,0);
	\draw [very thick,rdirected=.55] (-.75,0) to [out=60,in=180] (.5,.75);
	\draw [very thick,rdirected=.55] (-.75,0) to [out=300,in=180] (.5,-.75);
	\node at (-2.25,0) {\tiny $k{+}l$};
	\node at (.75,.75) {\tiny $l$};
	\node at (.75,-.75) {\tiny $k$};
\end{tikzpicture}
};
\endxy \mapsto R^{[k+l]} \otimes_{R^{[k+l]}} R^{[k,l]}
\]
Note that in the exponent $[k]$ denotes a sequence with only one entry (and not a quantum integer), 
so $R^{[k]}= \C[t_1,\ldots,t_k]^{\mathfrak{S}_k}$ is symmetric polynomials in $k$-variables, 
viewed as the diagonal $(R^{[k]},R^{[k]})$-bimodule. 
(Here we've rotated our webs by $45^\circ$ so they point diagonally, instead of leftward, for the sake of space and aesthetics.) 

The images of foams under $\Psi_{BS}$ can be deduced by translating the formulae in \cite{KL3} into our language. 
We will only explicitly need the formulae for the foams that are the images under $\Phi_{\infty}$ of cap and cup morphisms in $\UU(\glm)$.
Up to factors of $\pm1$ (which we won't specifically need), they are given as the following morphisms of $(R^{[k,l]},R^{[k,l]})$-bimodules:
\[
\begin{aligned}
\xy
(0,0)*{
\begin{tikzpicture} [scale=.5,fill opacity=0.2]
\path[fill=red] (2,-2) to (-2,-2) to (-2,1) to (2,1);
\path[fill=blue] (1.5,2) to [out=270,in=0] (0,-.5) to [out=180,in=270] (-1.5,2) to
	(-.5,1) to [out=270,in=180] (0,.25) to [out=0,in=270] (.5,1);
\path[fill=red] (2.5,-1) to (-2.5,-1) to (-2.5,2) to (2.5,2);
\draw[very thick, directed=.5] (2.5,-1) to (-2.5,-1);
\draw[very thick] (2.5,-1) to (2.5,2);
\draw[very thick] (-2.5,-1) to (-2.5,2);
\draw[very thick, directed=.5] (2.5,2) to (-2.5,2);
\draw[very thick, red, rdirected=.65] (1.5,2) to [out=270,in=0] (0,-.5)
	to [out=180,in=270] (-1.5,2);
\node[red, opacity=1] at (2.25,1.625) {\fns$l$};
\draw[very thick, directed=.5] (2,-2) to (-2,-2);
\draw[very thick] (2,-2) to (2,1);
\draw[very thick] (-2,-2) to (-2,1);
\draw[very thick, directed=.5] (2,1) to (-2,1);
\draw[very thick, red, rdirected=.75] (-.5,1) to [out=270,in=180] (0,.25)
	to [out=0,in=270] (.5,1);
\node[red, opacity=1] at (1.75,-1.625) {\fns$k$};
\draw[very thick, directed=.6] (1.5,2) to (.5,1);
\draw[very thick, directed=.6] (-.5,1) to (-1.5,2);
\end{tikzpicture}};
\endxy
& \mapsto
\left\{
\begin{aligned}
R^{[k,l]} \longrightarrow & q^{1-k-l} R^{[k,1,l-1]} \otimes_{R^{[k+1,l-1]}} R^{[k,1,l-1]} \\
1 \mapsto & \sum_{j=0}^k (-1)^{k-j}  t_{k+1}^{j} \otimes e_{k-j}(t_1,\ldots,t_k)
\end{aligned}
\right. \\
\xy
(0,0)*{
\begin{tikzpicture} [scale=.5,fill opacity=0.2]
\path[fill=red] (2,-2) to (-2,-2) to (-2,1) to (2,1);
\path[fill=blue] (1.5,1) to [out=270,in=0] (0,-.25) to [out=180,in=270] (-1.5,1) to
	(-.5,2) to [out=270,in=180] (0,1.25) to [out=0,in=270] (.5,2);
\path[fill=red] (2.5,-1) to (-2.5,-1) to (-2.5,2) to (2.5,2);
\draw[very thick, directed=.5] (2.5,-1) to (-2.5,-1);
\draw[very thick] (2.5,-1) to (2.5,2);
\draw[very thick] (-2.5,-1) to (-2.5,2);
\draw[very thick, directed=.5] (2.5,2) to (-2.5,2);
\draw[very thick, red, rdirected=.5] (-.5,2) to [out=270,in=180] (0,1.25)
	to [out=0,in=270] (.5,2);
\node[red, opacity=1] at (1.875,1.625) {\fns$l$};
\draw[very thick, directed=.5] (2,-2) to (-2,-2);
\draw[very thick] (2,-2) to (2,1);
\draw[very thick] (-2,-2) to (-2,1);
\draw[very thick, directed=.5] (2,1) to (-2,1);
\draw[very thick, red, rdirected=.5] (1.5,1) to [out=270,in=0] (0,-.25)
	to [out=180,in=270] (-1.5,1);
\node[red, opacity=1] at (1.5,-1.625) {\fns$k$};
\draw[very thick, directed=.6] (1.5,1) to (.5,2);
\draw[very thick, directed=.6] (-.5,2) to (-1.5,1);
\end{tikzpicture}
}
\endxy
& \mapsto
\left\{
\begin{aligned}
R^{[k,l]} \longrightarrow & q^{1-k-l} R^{[k-1,1,l]} \otimes_{R^{[k-1,l+1]}} R^{[k-1,1,l]} \\
1 \mapsto & \sum_{j=0}^l (-1)^{l-j}  t_{k}^{j} \otimes e_{l-j}(t_{k+1},\ldots,t_{k+l})
\end{aligned}
\right.
\end{aligned}
\]
\[
\begin{aligned}
\xy
(0,0)*{
\begin{tikzpicture} [scale=.5,fill opacity=0.2]
\path[fill=red] (2.5,2) to (-2.5,2) to (-2.5,-1) to (2.5,-1);
\path[fill=blue] (1.5,-1) to [out=90,in=0] (0,.25) to [out=180,in=90] (-1.5,-1) to
	(-.5,-2) to [out=90,in=180] (0,-1.25) to [out=0,in=90] (.5,-2);
\path[fill=red] (2,1) to (-2,1) to (-2,-2) to (2,-2);
\draw[very thick, directed=.5] (2.5,2) to (-2.5,2);
\draw[very thick] (2.5,2) to (2.5,-1);
\draw[very thick] (-2.5,2) to (-2.5,-1);
\draw[very thick, directed=.5] (2.5,-1) to (-2.5,-1);
\draw[very thick, red, directed=.5] (1.5,-1) to [out=90,in=0] (0,.25)
	to [out=180,in=90] (-1.5,-1);
\node[red, opacity=1] at (1.875,1.625) {\fns$l$};
\draw[very thick, directed=.5] (1.5,-1) to (.5,-2);
\draw[very thick, directed=.5] (-.5,-2) to (-1.5,-1);
\draw[very thick, directed=.5] (2,1) to (-2,1);
\draw[very thick] (2,1) to (2,-2);
\draw[very thick] (-2,1) to (-2,-2);
\draw[very thick, directed=.5] (2,-2) to (-2,-2);
\draw[very thick, red, directed=.5] (-.5,-2) to [out=90,in=180] (0,-1.25)
	to [out=0,in=90] (.5,-2);
\node[red, opacity=1] at (1.5,.625) {\fns$k$};
\end{tikzpicture}};
\endxy 
& \mapsto
\left\{
\begin{aligned}
q^{1-k-l} R^{[k,1,l-1]} \otimes_{R^{[k+1,l-1]}} R^{[k,1,l-1]} & \longrightarrow R^{[k,l]} \\
t_{k+1}^r \otimes t_{k+1}^s & \mapsto  h_{r+s+1-l}(t_{k+1},\ldots,t_{k+l})
\end{aligned}
\right.
\\
\xy
(0,0)*{
\begin{tikzpicture} [scale=.5,fill opacity=0.2]
\path[fill=red] (2.5,2) to (-2.5,2) to (-2.5,-1) to (2.5,-1);
\path[fill=blue] (1.5,-2) to [out=90,in=0] (0,.5) to [out=180,in=90] (-1.5,-2) to
	(-.5,-1) to [out=90,in=180] (0,-.25) to [out=0,in=90] (.5,-1);
\path[fill=red] (2,1) to (-2,1) to (-2,-2) to (2,-2);
\draw[very thick, directed=.5] (2.5,2) to (-2.5,2);
\draw[very thick] (2.5,2) to (2.5,-1);
\draw[very thick] (-2.5,2) to (-2.5,-1);
\draw[very thick, directed=.5] (2.5,-1) to (-2.5,-1);
\draw[very thick, red, directed=.75] (-.5,-1) to [out=90,in=180] (0,-.25)
	to [out=0,in=90] (.5,-1);
\node[red, opacity=1] at (2.25,1.625) {\fns$l$};
\draw[very thick, directed=.5] (1.5,-2) to (.5,-1);
\draw[very thick, directed=.5] (-.5,-1) to (-1.5,-2);
\draw[very thick, directed=.5] (2,1) to (-2,1);
\draw[very thick] (2,1) to (2,-2);
\draw[very thick] (-2,1) to (-2,-2);
\draw[very thick, directed=.5] (2,-2) to (-2,-2);
\draw[very thick, red, directed=.65] (1.5,-2) to [out=90,in=0] (0,.5)
	to [out=180,in=90] (-1.5,-2);
\node[red, opacity=1] at (1.75,-1.625) {\fns$k$};
\end{tikzpicture}};
\endxy 
& \mapsto
\left\{
\begin{aligned}
q^{1-k-l} R^{[k-1,1,l]} \otimes_{R^{[k-1,l+1]}} R^{[k-1,1,l]} & \longrightarrow  R^{[k,l]} \\
t_{k}^r \otimes t_{k}^s & \mapsto  h_{r+s+1-k}(t_1,\ldots,t_k)
\end{aligned}
\right.
\end{aligned}
\]

We can now succinctly describe Khovanov's construction of HOMFLYPT link homology. 
Given a (balanced, colored) braid $\beta$, we consider the complex $\llbracket \beta \rrbracket \in \Hot^b(\vFoam)$. 
Applying $\Psi_{BS}$ produces a complex $\Psi_{BS}(\llbracket \beta \rrbracket)$ of (singular) Soergel bimodules. 
Taking Hochschild homology\footnote{Since all of our rings $R^{\mathbf{k}}$ are isomorphic to polynomial rings, 
we could equivalently work with Hochschild cohomology.} 
produces a complex of bi-graded vector spaces. 
After shifting\footnote{This shift is required for invariance under the second Markov move. 
Here, $i_\beta$ is the sum of the labels of the strands in $\beta$, 
and the shifts involving it precisely correspond to the shift by $qa^{-1}$ in defining $\cal{V}_\infty$ in Theorem \ref{thm:HOMFLYPT}.}
by $q^{i_\beta + W_\beta - w_\beta} a^{w_\beta-i_\beta}$,
the homology of this complex is an invariant of the corresponding link $\cal{L}_\beta \subset S^3$, 
and this is the HOMFLYPT link homology $\HHH(\cal{L}_\beta)$.

We now show that this invariant is the same as the one constructed in Theorem \ref{thm:LinkHomologies} 
using $\Upsilon_{\cal{V}_\infty,t}^{\cal{S}}$ for $\cal{V}_\infty =qa^{-1} \C[t,\theta]/\theta^2$. 
To begin, we recall the following, which \eg appears as a special case of the results in \cite{BPW}; see also \cite{KhHHH,RouBraid}.
\begin{proposition}
The Hochschild homology functors $\HH_\bullet: \End_{\mathbf{BSBim}}(\mathbf{k}) \to \C\Vect$ factor through the horizontal trace, 
\ie they induce a functor $\HH_\bullet: \hTr(\mathbf{BSBim}) \to \gr\C\Vect$.
\end{proposition}
It follows that, up to the aforementioned shift, 
$\HHH(\cal{L}_\beta)$ is the homology of the complex obtained by applying the composition of functors:
\begin{equation}\label{eq:HHfunctor}
\Hot^b(\hTr(\vFoam)) \xrightarrow{\hTr(\Psi_{BS})} \Hot^b(\hTr(\mathbf{BSBim})) \xrightarrow{\HH_\bullet} \Hot^b(\gr\C\Vect)
\end{equation}
to the complex $\llbracket \widehat{\beta} \rrbracket$. We can now give the following.

\begin{proof}[Proof (of Theorem \ref{thm:HOMFLYPT}).]
In light of Theorem \ref{thm:foamObjDec}, it suffices to show that upon restricting 
equation \eqref{eq:HHfunctor} 
to $\vAFoam_\circ \subseteq \vAFoam \cong \hTr(\vFoam)$, 
the induced functor $\vAFoam_\circ \to \gr\C\Vect$ agrees with $\Upsilon_{\cal{V}_\infty,t}^\cal{S}$, 
up to isomorphism and grading shift.

To see this, we first compute $\HH_\bullet \circ \hTr(\Psi_{BS})$ on a $k$-labeled circle. 
By definition, this is the Hochschild homology of 
\[
R^{[k]} = \C[t_1,\ldots,t_k]^{\mathfrak{S}_k} = \C[e_1,\ldots,e_k]
\]
where $e_i$ denotes the $i^{th}$ elementary symmetric polynomial in the variables $t_1,\ldots,t_k$. 
The Hochschild homology of polynomial rings is well-known, see \eg \cite{Weibel}, and hence gives that 
\begin{equation}\label{eq:HOMFLYacircle}
\HH_\bullet(\C[e_1,\ldots,e_k]) \cong \C[e_1,\ldots,e_k] \otimes_\C \bV^\bullet(\C^k).
\end{equation}
We have $\deg_q(t_i)=2$ for all $i$, so $\deg_q(e_i) = 2i$.
In this presentation, the basis of $\C^k$ consists of vectors commonly denoted $\{de_i\}_{i=1}^k$
having Hochschild degree equal to one, and $q$-degree equal to $q^{2i-2}$. 
In our conventions, the $a$-degree equals twice the Hochschild degree, 
so $\deg_{q,a}(de_i) = (2i-2,2)$.
The computation
\begin{equation}\label{eq:FundThmSymDiff}
\Sym^k(\C[t,\theta]/\theta^2) \cong \left(\C[t_1,\ldots,t_k]\langle dt_1,\ldots, dt_k\rangle\right)^{\mathfrak{S}_k} 
\cong \C[e_1,\ldots,e_k]\langle de_1,\ldots,de_k\rangle
\end{equation}
which follows from the ``fundamental theorem of symmetric differential forms,''
shows that this agrees with the value of $\Upsilon_{\cal{V}_\infty,t}^\cal{S}$ on a $k$-labeled circle,
up to the required grading shift of $q^ka^{-k}$.

It remains to compute $\HH_\bullet \circ \hTr(\Psi_{BS})$ on the annular foams in equation \eqref{eq:FfoamEfoam}, 
and we only consider the first, as the second is completely analogous. 
As in Section \ref{sec:KhRsln}, 
it suffices to only check that this agrees with the image of these foams under $\Upsilon_{\cal{V}_\infty,t}^\cal{S}$ up to a sign. 
Moreover, we can factor each of the foams appearing in equation \eqref{eq:FfoamEfoam} into a composition of three annular foams, 
where the first foam splits off a $1$-labeled circle from one the circles, 
the second is the dotted cylinder endomorphism of the $1$-labeled circle, 
and the third merges the $1$-labeled circle with another (in each step, tensored with the identify maps on the un-involved circles).

To aid in this computation, we recall the computation of the Hochschild homology of polynomial rings via Koszul resolutions, 
and relate the resolution used to compute the value of $k$ concentric $1$-labeled circles
with one that computes 
equation \eqref{eq:HOMFLYacircle} above. 
For the former, we must compute the Hochschild homology of the polynomial ring $\C[t_1,\ldots,t_k]$.
In this case, the Koszul complex takes the form:
\begin{equation}\label{eq:Koszul}
\cal{P}_k := \C[x_1,\ldots,x_k] \otimes_\C \C[y_1,\ldots,y_k] \otimes_\C \bV^\bullet (\C^k)
\end{equation}
and, if we choose a basis $\{\zeta_1,\ldots,\zeta_k\}$ for $\C^k$, 
the differential is given by contraction with $\sum_{i=1}^k (y_i - x_i)\zeta_i^*$. 
Tensoring with $\C[t_1,\ldots,t_k]$, viewed as the diagonal $(\C[x_1,\ldots,x_k],\C[y_1,\ldots,y_k])$-bimodule, 
we obtain the complex $\C[t_1,\ldots,t_k] \otimes_\C \bV^\bullet (\C^k)$ with zero differential, 
which is hence the Hochschild homology of $\C[t_1,\ldots,t_k]$.
The symmetric group $\mathfrak{S}_k$ acts on the complex in \eqref{eq:Koszul}, 
and this action commutes with the differential. 
The terms in this complex are free 
$(\C[y_1,\ldots,y_k]^{\mathfrak{S}_k},\C[z_1,\ldots,z_k]^{\mathfrak{S}_k})$-bimodules, 
and taking $\mathfrak{S}_k$-invariants gives a free resolution $\cal{P}_k^{[k]}$ of $\C[t_1,\ldots,t_k]^{\mathfrak{S}_k}$ 
as a direct summand 
of \eqref{eq:Koszul}. 
Moreover, the map $\cal{P}_k \to \C[t_1,\ldots,t_k] \otimes_\C \bV^\bullet (\C^k)$ given by 
$x_i \mapsto t_i, y_i \mapsto t_i$ induces a surjective 
map\footnote{Here, and for the duration, the notation $A^e:= A \otimes_\C A$ denotes the enveloping algebra of a commutative $\C$-algebra $A$.
Recall that the Hochschild homology $\HH_\bullet(A,M)$ of an $(A,A)$-bimodule $M$ is computed by finding any projective resolution $P^\bullet \to M$, 
and taking the homology of $A \otimes_{A^e} P^\bullet$. We denote $\HH_\bullet(A,A) =: \HH_\bullet(A)$.} 
$R^{[k]} \otimes_{(R^{[k]})^e}\cal{P}_k^{[k]} \twoheadrightarrow \left( \C[t_1,\ldots,t_k] \otimes_\C \bV^\bullet (\C^k) \right)^{\mathfrak{S}_k}$,
and equation \eqref{eq:HOMFLYacircle} implies that this is a quasi-isomorphism 
(where we view the codomain as a complex with zero differential). 

We now compute the constituent pieces of the requisite foams. 
For the dotted cylinder endomorphism of a $1$-labeled circle, 
we must compute the map induced in Hochschild homology by the map
\[
R^{[1]} := \C[t] \xrightarrow{\multt} \C[t] =: R^{[1]}.
\]
The Koszul resolution of $\C[t]$ is given by $\C[x,y] \xrightarrow{x-y} \C[x,y]$, 
and the above map lifts to the map that is multiplication by $x$ on each term in the complex. 
After tensoring with $\C[t]$, this is precisely the map 
$\C[t,\theta]/\theta^2 \xrightarrow{\multt} \C[t,\theta]/\theta^2$, 
in agreement with $\Upsilon^\cal{S}_{\cal{V}_\infty,t}$.

Next, we consider the image under $\HH_\bullet \circ \hTr(\Psi_{BS})$ of the foam that splits a $k$-labeled circle into a
$(k-1)$- and $1$-labeled circle, \ie the foam:
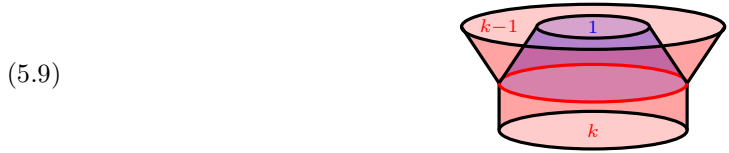
\begin{equation}\label{eq:circlesplit}
\begin{tikzpicture}[anchorbase, fill opacity=0.2]
\path[fill=red] (.25,.5) to (.25,-.125) arc (180:360:1.25 and 0.25) to (2.75,.5) arc (360:180:1.25 and 0.25);
\path[fill=red] (.25,.5) to (.25,-.125) arc (180:0:1.25 and 0.25) to (2.75,.5) arc (0:180:1.25 and 0.25);
\path[fill=blue] (.75,1.25) to (.25,.5) arc (180:360:1.25 and 0.25) to (2.25,1.25) arc (360:180:.75 and 0.15);
\path[fill=blue] (.75,1.25) to (.25,.5) arc (180:0:1.25 and 0.25) to (2.25,1.25) arc (0:180:.75 and 0.15);
\path[fill=red] (-.25,1.25) to (.25,.5) arc (180:360:1.25 and 0.25) to (3.25,1.25) arc (360:180:1.75 and .3);
\path[fill=red] (-.25,1.25) to (.25,.5) arc (180:0:1.25 and 0.25) to (3.25,1.25) arc (0:180:1.75 and .3);
\draw[very thick, red] (2.75,.5) arc (0:180:1.25 and 0.25);
\draw[very thick] (.25,-.125) to (.25,.5);
\draw[very thick] (.25,.5) to (.75,1.25);
\draw[very thick] (.25,.5) to (-.25,1.25);
\draw[very thick] (2.75,-.125) to (2.75,.5);
\draw[very thick] (2.75,.5) to (2.25,1.25);
\draw[very thick] (2.75,.5) to (3.25,1.25);
\draw [very thick]  (1.5,-.125) ellipse (1.25 and .25);
\draw[very thick, red] (.25,.5) arc (180:360:1.25 and 0.25);
\draw[very thick] (1.5,1.25) ellipse (.75 and 0.15);
\draw [very thick]  (1.5,1.25) ellipse (1.75 and .3);
\node[red, opacity=1] at (.25, 1.25) {\scriptsize$k{-}1$};
\node[red, opacity=1] at (1.5, -.125) {\scriptsize$k$};
\node[blue, opacity=1] at (1.5, 1.25) {\scriptsize$1$};
\end{tikzpicture}
\end{equation}
This can be presented as the composition of three annular foams: the first is the annular closure of
the $l=1$ case of the second to last foam in equation \eqref{eq:foamgens}, 
the second is an isotopy sliding a trivalent vertex around the annulus, 
and the third is the annular closure of the $l=1$ case of the first foam on the second line 
in equation \eqref{eq:foamgens}.
After mapping to Soergel bimodules, 
the first is the map induced in Hochschild homology by the map of $(R^{[k]},R^{[k]})$-bimodules
\begin{equation}\label{eq:split1}
\begin{aligned}
R^{[k]} &\rightarrow R^{[k-1,1]} \\
1 &\mapsto 1.
\end{aligned}
\end{equation}
Note that $\cal{P}_k^{[k]}$ is a complex of projective, hence free, $(R^{[k]},R^{[k]})$-bimodules, 
and that $R^{[k-1,1]}$ is a free (left or right) $R^{[k]}$-module. 
It follows that $R^{[k-1,1]} \otimes_{R^{[k]}} \cal{P}_k^{[k]}$ is a projective resolution of 
$R^{[k-1,1]} \otimes_{R^{[k]}} R^{[k]} \cong R^{[k-1,1]}$, so we can lift the map in equation \eqref{eq:split1} 
to the map $\cal{P}_k^{[k]} \to R^{[k-1,1]} \otimes_{R^{[k]}} \cal{P}_k^{[k]}$ that sends $v \mapsto 1 \otimes v$.
The map on Hochschild homology is 
hence
induced by the map
\[
R^{[k]} \otimes_{(R^{[k]})^e} \cal{P}_k^{[k]}  \to R^{[k]} \otimes_{(R^{[k]})^e} \left( R^{[k-1,1]} \otimes_{R^{[k]}} \cal{P}_k^{[k]} \right) 
\]
that sends $1 \otimes v \mapsto 1 \otimes 1 \otimes v$. 
Next, the slide isotopy is the isomorphism 
\[
\HH_\bullet(R^{[k]}, R^{[k-1,1]}) \xrightarrow{\cong} \HH_\bullet(R^{[k-1,1]}, R^{[k-1,1]} \otimes_{R^{[k]}} R^{[k-1,1]}).
\]
To explicitly compute this, first note that $R^{[k-1,1]}\otimes_{R^{[k]}} \cal{P}_k^{[k]} \otimes_{R^{[k]}} R^{[k-1,1]}$ is a projective 
resolution of $R^{[k-1,1]} \otimes_{R^{[k]}} R^{[k-1,1]} \cong R^{[k-1,1]}\otimes_{R^{[k]}} R^{[k]} \otimes_{R^{[k]}} R^{[k-1,1]}$, 
as a $(R^{[k-1,1]},R^{[k-1,1]})$-bimodule. 
This follows since the terms of $\cal{P}_k^{[k]}$ are free $(R^{[k]},R^{[k]})$-bimodules, and
\[
R^{[k-1,1]} \otimes_{R^{[k]}} (R^{[k]} \otimes_\C R^{[k]}) \otimes_{R^{[k]}} R^{[k-1,1]} \cong R^{[k-1,1]} \otimes_\C R^{[k-1,1]}.
\]
Using this resolution, the slide map is induced by the map
\[
R^{[k]} \otimes_{(R^{[k]})^e} \left( R^{[k-1,1]} \otimes_{R^{[k]}} \cal{P}_k^{[k]} \right) \xrightarrow{\cong}
R^{[k-1,1]} \otimes_{(R^{[k-1,1]})^e} \left( R^{[k-1,1]}\otimes_{R^{[k]}} \cal{P}_k^{[k]} \otimes_{R^{[k]}} R^{[k-1,1]} \right)
\]
that sends $1 \otimes f \otimes v \mapsto 1 \otimes f \otimes v \otimes 1$.
Finally, the third map is the one induced on Hochschild homology by the multiplication map
\[
\begin{aligned}
R^{[k-1,1]} \otimes_{R^{[k]}} R^{[k-1,1]} &\rightarrow R^{[k-1,1]} \\
f \otimes g &\mapsto fg
\end{aligned}
\]
of $(R^{[k-1,1]},R^{[k-1,1]})$-bimodules. 
In order to lift this map to our projective resolution, it helps to note that, after using the isomorphism 
$R^{[k-1,1]} \otimes_{R^{[k]}} R^{[k-1,1]} \cong R^{[k-1,1]} \otimes_{R^{[k]}} R^{[k]} \otimes_{R^{[k]}} R^{[k-1,1]}$,
the above map is the same as the composition:
\[
R^{[k-1,1]} \otimes_{R^{[k]}} R^{[k]} \otimes_{R^{[k]}} R^{[k-1,1]}
\hookrightarrow R^{[k-1,1]} \otimes_{R^{[k]}} R^{[k-1,1]} \otimes_{R^{[k]}} R^{[k-1,1]} \rightarrow 
R^{[k-1,1]}
\]
where the second map is simply multiplication $f \otimes g \otimes h \mapsto fgh$.
We can lift this to the map
\[
R^{[k-1,1]}\otimes_{R^{[k]}} \cal{P}_k^{[k]} \otimes_{R^{[k]}} R^{[k-1,1]} \hookrightarrow 
R^{[k-1,1]}\otimes_{R^{[k]}} \cal{P}_k^{[k-1,1]} \otimes_{R^{[k]}} R^{[k-1,1]} \rightarrow
\cal{P}_k^{[k-1,1]}
\]
given by $f \otimes v \otimes g \mapsto fvg$. Our final map is the one induced by this on Hochschild homology.

Composing these three maps, 
we find that the foam in equation \eqref{eq:circlesplit} is sent to the map on homology induced by the map
\[
R^{[k]} \otimes_{(R^{[k]})^e} \cal{P}_k^{[k]} \to R^{[k-1,1]} \otimes_{(R^{[k-1,1]})^e} \cal{P}_k^{[k-1,1]}
\]
that sends $1 \otimes v \mapsto 1 \otimes v$, 
\ie on homology it is simply the inclusion 
\[
\left(\C[t_1,\ldots,t_k] \otimes_\C \bV^\bullet (\C^k)\right)^{\mathfrak{S}_k} 
\xhookrightarrow{\iota} \left(\C[t_1,\ldots,t_k] \otimes_\C \bV^\bullet (\C^k)\right)^{\mathfrak{S}_{a-1} \times \mathfrak{S}_1}
\]
as desired.

Finally, it suffices to compute the image under $\HH_\bullet \circ \hTr(\Psi_{BS})$ of the foam: 
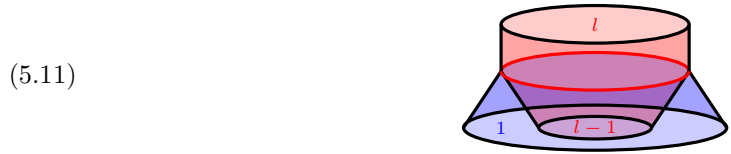
\begin{equation}\label{eq:circlemerge}
\begin{tikzpicture}[anchorbase, fill opacity=0.2,yscale=-1]
\path[fill=red] (.25,.5) to (.25,-.125) arc (180:360:1.25 and 0.25) to (2.75,.5) arc (360:180:1.25 and 0.25);
\path[fill=red] (.25,.5) to (.25,-.125) arc (180:0:1.25 and 0.25) to (2.75,.5) arc (0:180:1.25 and 0.25);
\path[fill=red] (.75,1.25) to (.25,.5) arc (180:360:1.25 and 0.25) to (2.25,1.25) arc (360:180:.75 and 0.15);
\path[fill=red] (.75,1.25) to (.25,.5) arc (180:0:1.25 and 0.25) to (2.25,1.25) arc (0:180:.75 and 0.15);
\path[fill=blue] (-.25,1.25) to (.25,.5) arc (180:360:1.25 and 0.25) to (3.25,1.25) arc (360:180:1.75 and .3);
\path[fill=blue] (-.25,1.25) to (.25,.5) arc (180:0:1.25 and 0.25) to (3.25,1.25) arc (0:180:1.75 and .3);
\draw[very thick, red] (.25,.5) arc (180:360:1.25 and 0.25);
\draw[very thick] (.25,-.125) to (.25,.5);
\draw[very thick] (.25,.5) to (.75,1.25);
\draw[very thick] (.25,.5) to (-.25,1.25);
\draw[very thick] (2.75,-.125) to (2.75,.5);
\draw[very thick] (2.75,.5) to (2.25,1.25);
\draw[very thick] (2.75,.5) to (3.25,1.25);
\draw[very thick] (1.5,1.25) ellipse (.75 and 0.15);
\draw [very thick]  (1.5,1.25) ellipse (1.75 and .3);
\draw[very thick, red] (2.75,.5) arc (0:180:1.25 and 0.25);
\draw [very thick]  (1.5,-.125) ellipse (1.25 and .25);
\node[blue, opacity=1] at (.25, 1.25) {\scriptsize$1$};
\node[red, opacity=1] at (1.5, -.125) {\scriptsize$l$};
\node[red, opacity=1] at (1.5, 1.25) {\scriptsize$l-1$};
\end{tikzpicture}
\end{equation}
To do so, we use a similar argument as in Section \ref{sec:KhRsln}. 
Indeed, precomposing the map $\HH_\bullet \circ \hTr(\Psi_{BS})$ with the map $\U(\slnn{2}[t]) \xrightarrow{\vTr{\Phi_\infty}} \vAFoam$, 
we see that the image of the foam in equation \eqref{eq:circlemerge} is the action of the standard Chevalley generator $F \in \slnn{2}$ mapping from 
the second-lowest to the lowest weight 
space of an $\slnn{2}$ representation\footnote{Although this representation is infinite-dimensional, 
it is finite-dimensional in each $q$-degree and the $\slnn{2}$ action preserves this degree.}. 
It follows that on this weight space
we can write $F = \begin{pmatrix} F' & 0 \end{pmatrix}$ with $F'$ invertible, using the isotypic decomposition. 
On the lowest weight space, which in this case has weight $-l$, 
the Chevalley generator $E$ can similarly be written as $E = \begin{pmatrix} E' \\ 0 \end{pmatrix}$, and we have
\[
-l \cdot \id = EF - FE = 0 - \begin{pmatrix} F' & 0 \end{pmatrix} \begin{pmatrix} E' \\ 0 \end{pmatrix} = -F'E'
\]
thus $F' = l \cdot (E')^{-1}$.
Now, $E$ acts via the foam
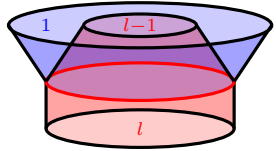
\[
\begin{tikzpicture}[anchorbase, fill opacity=0.2]
\path[fill=red] (.25,.5) to (.25,-.125) arc (180:360:1.25 and 0.25) to (2.75,.5) arc (360:180:1.25 and 0.25);
\path[fill=red] (.25,.5) to (.25,-.125) arc (180:0:1.25 and 0.25) to (2.75,.5) arc (0:180:1.25 and 0.25);
\path[fill=red] (.75,1.25) to (.25,.5) arc (180:360:1.25 and 0.25) to (2.25,1.25) arc (360:180:.75 and 0.15);
\path[fill=red] (.75,1.25) to (.25,.5) arc (180:0:1.25 and 0.25) to (2.25,1.25) arc (0:180:.75 and 0.15);
\path[fill=blue] (-.25,1.25) to (.25,.5) arc (180:360:1.25 and 0.25) to (3.25,1.25) arc (360:180:1.75 and .3);
\path[fill=blue] (-.25,1.25) to (.25,.5) arc (180:0:1.25 and 0.25) to (3.25,1.25) arc (0:180:1.75 and .3);
\draw[very thick, red] (2.75,.5) arc (0:180:1.25 and 0.25);
\draw[very thick] (.25,-.125) to (.25,.5);
\draw[very thick] (.25,.5) to (.75,1.25);
\draw[very thick] (.25,.5) to (-.25,1.25);
\draw[very thick] (2.75,-.125) to (2.75,.5);
\draw[very thick] (2.75,.5) to (2.25,1.25);
\draw[very thick] (2.75,.5) to (3.25,1.25);
\draw [very thick]  (1.5,-.125) ellipse (1.25 and .25);
\draw[very thick, red] (.25,.5) arc (180:360:1.25 and 0.25);
\draw[very thick] (1.5,1.25) ellipse (.75 and 0.15);
\draw [very thick]  (1.5,1.25) ellipse (1.75 and .3);
\node[blue, opacity=1] at (.25, 1.25) {\scriptsize$1$};
\node[red, opacity=1] at (1.5, -.125) {\scriptsize$l$};
\node[red, opacity=1] at (1.5, 1.25) {\scriptsize$l{-}1$};
\end{tikzpicture}
\]
and a similar computation to the one used to compute the foam in equation \eqref{eq:circlesplit} 
shows that this is the inclusion 
\[
\left(\C[t_1,\ldots,t_l] \otimes_\C \bV^\bullet (\C^l)\right)^{\mathfrak{S}_l} 
\xhookrightarrow{\iota} \left(\C[t_1,\ldots,t_l] \otimes_\C \bV^\bullet (\C^l)\right)^{\mathfrak{S}_{1} \times \mathfrak{S}_{l-1}}
\]
hence the foam in equation \eqref{eq:circlemerge} is $l$ times the projection, as desired. 
\end{proof}

\begin{remark}\label{rem:HHHdecat}
Our construction of $\HHH(\cal{L})$ in this section precisely categorifies the description of the HOMFLYPT polynomial 
given in Proposition \ref{prop:SymPoly}, which is given as the $n \to \infty$ limit of the \textit{symmetric webs} presentation of the 
$\sln$ link polynomials. Indeed, the complex $\llbracket \widehat{\beta} \rrbracket_\circ$ decategorifies to exactly give $\{ \widehat{\beta} \}_\circ$, 
and the isomorphism in equation \eqref{eq:FundThmSymDiff} shows that we exactly have
\[
\dim_{q,a}\Sym^k(\cal{V}_\infty) = q^{k} a^{-k} \prod_{i=1}^k \frac{1-a^2q^{2i-2}}{1-q^{2i}} = \mathsf{c}_{\Sym^k}
\]
For this reason, we will see in Section \ref{sec:diffnegn} that the spectral sequence induced by the ``most obvious'' 
differential on this complex converges not to $\KhR_n(\cal{L})$, but rather to the $\slnegn$ link homology 
that is constructed in the next section (and is likewise built from symmetric powers).
\end{remark}

\subsection{$\slnegn$ link homology} \label{sec:SymHomology}

As detailed after Theorem \ref{thm:LinkHomologies}, 
$\overline{\cal{H}}_{\cal{V}_n,t}^\cal{S}( \widehat{\beta})$
is particularly interesting. 
In this section, we study this invariant, and prove Theorem \ref{thm:negativen}.

At the decategorified level, 
the $\sln$ Reshetikhin-Turaev polynomials~\cite{RT_ribbon} can be defined for links colored by arbitrary $U_q(\sln)$ representations; 
however, in the typical web-based approach to link polynomials 
and link homology, 
$k$-colored link components are understood as colored by the 
fundamental representations $\bV^k(\C^n_q)$.
On the other hand, the symmetric powers of the fundamental representation are the ones involved 
in the construction of 3-manifold 
invariants\footnote{This requires specializing $q$ to a root of unity \cite{RT_3mfd,Witten3}, 
although the results in \cite{CoopKrush2} suggest other means to recover the 3-manifold invariants.} and $(2+1)$-dimensional TQFTs.
Thus far, the typical method for categorifying such colored invariants has mostly followed the route of constructing categorical analogues of the 
highest weight projectors from the $k$-fold tensor product of the standard representation to the $k^{th}$ symmetric representation, 
see \eg~\cite{CoopKrush,SS,Roz,Rose,Cautis}.
This technique results in homology theories for colored knots and links that are manifestly infinite-dimensional, 
and thus cannot be directly used to build $(3+1)$-dimensional TQFTs. 
The following result is hence promising, 
and follows immediately from the definition of $\overline{\cal{H}}_{\cal{V},T}^\cal{S}( \widehat{\beta})$ in Theorem \ref{thm:LinkHomologies}, 
and the discussion in Section \ref{sec:Decat}, particularly Remark \ref{rem:Coloring}.

\begin{proposition} 
Let $\beta$ be a balanced, colored braid and $\cal{V}_n = q^{1-n}\C[t]/t^n$ with $\deg_q(t) =2$,
then $q^{n w_\beta + W_\beta - w_\beta} \overline{\cal{H}}_{\cal{V}_n,t}^\cal{S}( \widehat{\beta})$
is finite dimensional, and taking the alternating sum of graded dimensions 
recovers the $\sln$ Reshetikhin-Turaev invariant of the associated colored link $\cal{L}_\beta$, with components colored by the corresponding 
symmetric powers of the standard representation $\C_q^n$ of $U_q(\mathfrak{sl}_n)$.
\end{proposition}
\begin{proof}
The invariant is finite-dimensional since it is the homology of a bounded complex of finite-dimensional graded vector spaces.
Comparing the construction of $\overline{\cal{H}}_{\cal{V}_n,t}^\cal{S}( \widehat{\beta})$ with the discussion in Section \ref{sec:Decat},
and noting that $\dim_q(\Sym^k(\cal{V}_n)) = {n+k-1 \brack k}$,
proves the second claim.
\end{proof}

From the definition, it is only clear \ap that $q^{n w_\beta + W_\beta - w_\beta} \overline{\cal{H}}_{\cal{V}_n,t}^\cal{S}( \widehat{\beta})$ is an invariant of the annular link $\widehat{\beta}$. 
We now show that this homology is indeed invariant under the second Markov move, 
and hence determines an invariant of links $\cal{L}_\beta \subset S^3$.
Recall that we denote the resulting homology theory $\cal{H}_{-n}(\cal{L}_\beta)$ and call it \textit{$\slnegn$ link homology}, 
since up to a shift in super-degree, it can equivalently be obtained as 
$q^{n w_\beta + W_\beta - w_\beta} \overline{\cal{H}}_{\Pi \cal{V}_n,t}^\wedge( \widehat{\beta})$,
and thus is the analogue of $\sln$ Khovanov-Rozansky homology for the Lie superalgebra $\slnegn = \slzn$
(compare to Theorem \ref{thm:sln}, and note that $\Pi \cal{V}_n \cong \C^{0|n}$). 
Nevertheless, we work with its formulation 
in terms of symmetric powers of $\cal{V}_n$
to avoid the use of super vector spaces.

To prove Markov II invariance, we observe that, surprisingly, this invariant has appeared before, although in a different guise, 
\ie without the use of our annular evaluation technology, and without the relation to $\slzn$.
Indeed, while this work was in progress,
the paper \cite{Cautis2} appeared. 
Therein, Cautis constructs a differential $\delta_n$ on the complex used to compute HOMFLYPT link homology,
\ie in our notation, on the complex $q^{i_\beta + W_\beta - w_\beta} a^{w_\beta-i_\beta} \HH_\bullet(\Psi_{BS}(\llbracket \beta \rrbracket))$. 
Denoting the ``usual'' differential on this complex by $\del$, 
\ie the one coming from equations \eqref{eq:slnCrossing} and \eqref{eq:slnColoredCrossing}, 
it is shown in \cite[Theorem 6.1]{Cautis2} that $\delta_n$ commutes with $\del$, and that the homology of the total complex 
$\Tot(q^{i_\beta + W_\beta - w_\beta} a^{w_\beta-i_\beta} \HH_\bullet(\Psi_{BS}(\llbracket \beta \rrbracket)), \del, \delta_n)$ 
is an invariant of $\cal{L}_\beta$.

In Section \ref{sec:diffnegn} below, we introduce a differential $d_{-n}$ on the complex 
\begin{equation}\label{eq:HOMFLYPTcomplex}
C_\circ(\widehat{\beta}) := a^{w_\beta-i_\beta} q^{i_\beta + W_\beta - w_\beta} \HH_\bullet(\hTr(\Psi_{BS}(\llbracket \widehat{\beta} \rrbracket_\circ))) \simeq 
a^{w_\beta-i_\beta} q^{i_\beta + W_\beta - w_\beta} \HH_\bullet(\Psi_{BS}(\llbracket \beta \rrbracket))
\end{equation}
that commutes with $\del$, and so that 
\[
\mathrm{H}^\bullet(\mathrm{H}^\bullet(C_\circ(\widehat{\beta}),d_{-n}),\del) \cong \cal{H}_{-n}(\cal{L}_\beta).
\]
Here $\mathrm{H}^\bullet$ denotes taking homology with respect to the indicated differentials.
For degree reasons, the spectral sequence 
$\mathrm{H}^\bullet(\mathrm{H}^\bullet(C_\circ(\widehat{\beta}),d_{-n}),\del) \Rightarrow \mathrm{H}^\bullet(\Tot(C_\circ(\widehat{\beta}),\del, d_{-n}))$
collapses at the second page, which gives that $\cal{H}_{-n}(\cal{L}_\beta) \cong \mathrm{H}^\bullet(\Tot(C_\circ(\widehat{\beta}),\del, d_{-n}))$.

Hence, to show that $\cal{H}_{-n}(\cal{L}_{\beta})$ is an invariant of links in $S^3$, 
it suffices\footnote{In general, given two double complexes $(C,\del_C,\delta_C)$ and $(D,\del_D,\delta_D)$ and a homotopy equivalence 
$f:(C,\del_C) \xrightarrow{\simeq}(D,\del_D)$ with homotopy inverse $g$, to deduce that $f$ and $g$ induce a homotopy equivalence between 
the total complexes, it suffices to check that $f$ and $g$ are chain maps with respect to the differentials $\delta_C$ and $\delta_D$, 
and that these differentials commute with the maps realizing the homotopies $\id_C \simeq gf$ and $\id_D \simeq fg$. 
In our case, all homotopy equivalences are Gaussian eliminations, hence are built from the entries of the ``horizontal'' differentials,
so the check that the ``vertical'' differentials commute with the relevant maps follows since these latter differentials are given summand-wise, 
and commute with the former.} 
to show that the differential $\delta_n$ agrees with $d_{-n}$ after applying the homotopy equivalence 
in equation \eqref{eq:HOMFLYPTcomplex}. 
To see this, we observe that, starting with the complex 
$q^{i_\beta + W_\beta - w_\beta} a^{w_\beta-i_\beta} \HH_\bullet(\Psi_{BS}(\llbracket \beta \rrbracket))$, 
we can proceed as in Theorem \ref{thm:foamObjDec}, 
simplifying the complex by replacing the terms in the complex, which are the Hochschild homology of bimodules assigned to 
the various webs in $\llbracket \beta \rrbracket$, with the Hochschild homology of simpler webs. 
It suffices to show that the differential $\delta_n$ is compatible with these simplifications, 
and then to identify it with $d_{-n}$ on the simplified complex 
$C_\circ(\widehat{\beta})$.
The differential $\delta_n$ acts summand-wise on $q^{i_\beta + W_\beta - w_\beta} a^{w_\beta-i_\beta} \HH_\bullet(\Psi_{BS}(\llbracket \beta \rrbracket))$:
each web in $\llbracket \beta \rrbracket$ contributes a term that is the Hochschild homology of an 
$(R^\mathbf{k},R^\mathbf{k})$-bimodule (where $\mathbf{k}$ corresponds to the labelings of the boundary of the web), 
and the differential is given on such a term by the action of an element $\gamma_\mathbf{k} \in \HH^1(R^\mathbf{k})$ 
on the Hochschild homology.

All of the homotopy equivalences used in Theorem \ref{thm:foamObjDec} to simplify the complex $\llbracket \widehat{\beta} \rrbracket$ 
take one of the following three forms:
\begin{enumerate}
\item a web $\cal{W}$ is replaced by a direct sum $\bigoplus_i \cal{W}_i$ of isomorphic webs,
\item a ``reverse'' Gaussian elimination homotopy equivalence is used to replace a web that is isomorphic to a summand of another web, or
\item an annular web is replaced by an isotopic one, using an isotopy that slides part of the web around the annulus.
\end{enumerate}
In our setting, compatibility with the analogue of the first type of equivalence follows since
the action of $\gamma_\mathbf{k}$ on $\HH_\bullet(\Psi_{BS}(\cal{W}))$ agrees with that on 
$\bigoplus_i \HH_\bullet(\Psi_{BS}(\cal{W}_i))$. 
This also confirms the second type, since the maps realizing the homotopy equivalence are built from the terms in the isomorphism 
and the differentials in the complex, and these all commute with the action of $\gamma_\mathbf{k}$.
Compatibility with the third type is equivalent to showing that if a web $\cal{W}$ with boundary $\mathbf{k}$ can be written as the 
(horizontal) composition of two webs $\cal{W} = \cal{W}_1 \cdot \cal{W}_2$ glued along boundary $\mathbf{p}$, 
then the action of $\gamma_\mathbf{k}$ on $\HH_\bullet(\Psi_{BS}(\cal{W}_1 \cdot \cal{W}_2))$ agrees with that of 
$\gamma_\mathbf{p}$ on $\HH_\bullet(\Psi_{BS}(\cal{W}_2 \cdot \cal{W}_1))$. 
This follows directly from \cite[Lemma 6.2]{Cautis2}. 
Finally, identifying $\delta_n$ on $\HH_\bullet(\hTr(\Psi_{BS}(\llbracket \widehat{\beta} \rrbracket_\circ)))$ amounts to computing 
the action of $\gamma_\mathbf{k}$ on $\HH_\bullet(R^{\mathbf{k}})$, which we do in Section \ref{sec:diffn} below, 
after we define our differential $d_{-n}$.

We now note a surprising result, 
which will complete the proof of Theorem \ref{thm:negativen}.
For an uncolored (\ie $1$-colored) braid $\beta$, 
both $\cal{H}_{-n}(\cal{L}_\beta)$ and $\KhR_n(\cal{L}_\beta)$ are invariants of the link $\cal{L}_\beta \subset S^3$, 
and both decategorify to the (uncolored) $\sln$ Reshetikhin-Turaev invariant of $\cal{L}_\beta$, since $\bV^1\C^n_q \cong \C^n_q \cong \Sym^1\C^n_q$. 
One might expect that (perhaps up to a re-grading) these invariants are isomorphic, but this is not the case.

\begin{proposition}
In general, $\cal{H}_{-n}(\cal{L}_\beta) \ncong \KhR_n(\cal{L}_\beta)$ for an uncolored link $\cal{L}_\beta$.
\end{proposition}
\begin{proof}
It suffices to exhibit an example. 
Let $\cal{T}$ be the right-handed trefoil knot.
In Section \ref{sec:Trefoil} we show that $\dim_{q,h}(\KhR^{\vee}_2(\cal{T})) = (q+q^3)+h^2q^5+h^3q^9$, 
while $\dim_{q,h}(\cal{H}_{-2}(\cal{T})) = h^{-2}(q^7+q^5) + h^{-1}(q^9+q^7) + (q^3+q)$.
Hence, these invariants are not isomorphic, even as non-graded vector spaces 
(but do decategorify to give the same Laurent polynomial).
\end{proof}
Moreover, $\cal{H}_{-2}(\cal{L}_\beta)$ is also distinct from odd Khovanov homology \cite{ORS};
see Section \ref{sec:Hopf}.
We also stress that even in the most trivial case of $n=1$, the two homologies $\cal{H}_{-n}(\cal{L}_\beta)$ and $\KhR_n(\cal{L}_\beta)$ differ; 
the latter is always $1$-dimensional, so it appears that the former contains more information. 

\begin{remark}
It is predicted in \cite{DGR} that there should exist a \textit{reduced} $\slnegn$ link homology, 
and that this theory should be isomorphic to reduced $\sln$ Khovanov-Rozansky homology. 
It is an interesting open question to relate such a theory to our invariant $\cal{H}_{-n}(\cal{L}_\beta)$.
\end{remark}

\subsection{Other link homology theories}

In this section, we discuss a number of other link homologies (some old, some new), 
which are given by our construction for choices of $\cal{V}$ and $T$ not-yet discussed. 
Our exposition will be brief, 
as all results follow from arguments analogous to those in Sections \ref{sec:KhRsln}~--~\ref{sec:SymHomology}, 
and in \cite{QR2}.

\subsubsection{Annular theories} 

Taking $\cal{V} = \C^n, T=0$, 
the definition of $\cal{H}_{\C^n,0}^\wedge(\widehat{\beta})$ 
precisely coincides with that for annular $\sln$ Khovanov-Rozansky link homology $\cal{A}\KhR_n(\widehat{\beta})$. 
This invariant of annular links $\widehat{\beta}$ was defined by the first two named authors in \cite{QR2} (extending the $n=2$ case from \cite{APS}), 
and we recall here its pertinent properties. 
The differentials in $\llbracket \widehat{\beta} \rrbracket$ are sent by $\Upsilon_{\C^n,0}^\wedge$ to compositions of 
the maps ${_k \iota_l}$ and ${_k \pi_l}$ from \eqref{eq:SplitMerge}, which commute with the action of $\sln$ on the vector spaces 
$\Upsilon_{\C^n,0}^\wedge(\mathbf{k}) = \bV^{k_1}\C^n \otimes \cdots \otimes \bV^{k_m}\C^n$ induced via the canonical action of $\sln$ on $\C^n$. 
It follows that $\cal{A}\KhR_n(\widehat{\beta})$ carries an action of $\sln$, 
generalizing the $\slnn{2}$ action on sutured annular Khovanov homology from \cite{GLW}.
Simple computations show that this does not determine an invariant of the corresponding link $\cal{L}_\beta$ 
(\ie is not invariant under the second Markov move); 
however, there does exists a spectral sequence 
$\cal{A}\KhR_n(\widehat{\beta}) \Rightarrow \KhR_n( \cal{L}_\beta)$,
arising from the filtration by (relative) current algebra degree on the complex used to compute the latter. 

Passing from skew-symmetric to symmetric tensors, we consider $\cal{H}_{\C^n,0}^{\cal{S}}(\widehat{\beta})$, 
which is an invariant of annular links $\widehat{\beta}$ built from the $\sln$ representations $\Sym^k(\C^n)$. 
As in the non-annular case, up to an overall super-degree shift we can also obtain this invariant as
$\cal{H}_{\C^{0|n},0}^\wedge(\widehat{\beta})$, 
so we call this invariant \textit{annular $\slnegn$ link homology}, and denote it by $\cal{AH}_{-n}(\widehat{\beta})$.
Simple computations show that $\cal{AH}_{-n}(\widehat{\beta}) \ncong \cal{A}\KhR_n(\widehat{\beta})$.
Repeating the above arguments in this setting, we have the following result.

\begin{proposition}
Annular $\slnegn$ link homology is an invariant of the annular link $\widehat{\beta}$, carries an action of
$\sln$, and admits a spectral sequence $\cal{AH}_{-n}(\widehat{\beta}) \Rightarrow \cal{H}_{-n}(\cal{L}_\beta)$. 
\end{proposition}

More generally, we can consider $\cal{H}_{\cal{V},0}^\square(\widehat{\beta})$ for any choice of $\cal{V}$
(and for $\square = \cal{S}$ or $\wedge$).
Taking $\cal{V} = \C[t,\theta]/\theta^2$ produces the \textit{annular HOMFLYPT homology} $\cal{A}\HHH(\widehat{\beta})$
conjectured to exist in \cite{GNSSS}. 
As in the $\sln$ and $\slnegn$ cases, filtering the complex $C_\circ(\widehat{\beta})$ from equation \eqref{eq:HOMFLYPTcomplex} 
by current algebra degree produces a spectral sequence from annular to non-annular HOMFLYPT homology. 
Additionally, in this case we can also repeat the constructions in Section \ref{sec:Diff} below to obtain spectral sequences from 
annular HOMFLYPT homology to the annular $\sln$ and $\slnegn$ homologies. In summary, we obtain the following result.

\begin{proposition}
Annular HOMFLYPT homology is an invariant of the annular link $\widehat{\beta}$, and admits spectral sequences 
converging to each of $\HHH(\cal{L}_\beta)$, $\cal{A}\KhR_n(\widehat{\beta})$, and $\cal{AH}_{-n}(\widehat{\beta})$.
\end{proposition}

\begin{remark}
Taking $\cal{V} = \C^{m|n}$ for any choice of $m,n \in \N$, 
we obtain an annular link homology theory corresponding to $\glmn$. 
The super Howe duality from \cite{QS2} shows that this invariant similarly carries an action of $\glmn$.
Unfortunately, for general $m,n$, 
it is not clear how (or if) this should be related to the conjectural $\glmn$ homology \cite{GGS} of colored links in $S^3$. 
\end{remark}

\subsubsection{Deformations}

In \cite{RW}, the second named author and Wedrich showed that deformed, colored $\sln$ link homology 
(see \cite{Wu2} and references therein)
can be defined in the context of a deformation of $\nFoam$, in which the relation \eqref{eq:ndot} is replaced by the relation 
\[
p\left( \;
\xy
(0,0)*{
\begin{tikzpicture} [fill opacity=0.2, scale=.5]
	\draw [very thick, fill=red] (1,1) -- (-1,2) -- (-1,-1) -- (1,-2) -- cycle;
	\node [opacity=1] at (0,0) {$\bullet$};
	\node [red, opacity=1] at (-.5,1.25){\tiny$1$};
\end{tikzpicture}}; 
\endxy
\;
\right)
= 0
\]
where $p(t) \in \C[t]$ is any (monic) polynomial of degree $n$. 
Repeating the arguments in Section \ref{sec:KhRsln} in this context, 
and noting that \cite[Lemma 15]{RW} identifies the deformed $\sln$ homology of the $k$-colored unknot with $\bV^k(\C[t]/p(t))$, 
we arrive at the following.

\begin{proposition}
Let $\cal{V}_p = \C[t]/p(t)$, then for any balanced, colored braid $\beta$, $\cal{H}_{\cal{V}_p,t}^{\wedge}(\cal{L}_\beta)$ 
agrees with the deformed, colored $\sln$ link homology of $\cal{L}_\beta$.
\end{proposition}

Similarly, we can define \textit{deformed, colored $\slnegn$ link homology} by considering $\cal{H}_{\cal{V}_p,t}^{\cal{S}}(\cal{L}_\beta)$.

\section{Differentials and mirror symmetry}\label{sec:Diff}

In this section, we further study HOMFLYPT link homology. 
We construct two families of differentials that induce spectral sequences 
$\HHH(\cal{L}_\beta) \Rightarrow  \cal{H}_{-n}( \cal{L}_\beta)$ and $\HHH(\cal{L}_\beta) \Rightarrow  \KhR_n( \cal{L}_\beta)$.
In order to define the latter, we give a categorical interpretation of the ``mirror symmetry'' for HOMFLYPT polynomials detailed 
in Remark \ref{rem:Mirror}.

\subsection{The differential $d_{-n}$}\label{sec:diffnegn}

The super vector space $\C[t,\theta]/\theta^2$ admits a differential\footnote{It may seem strange notation 
to call this differential $d_{-n}$, as opposed to $d_n$, but we take this choice because the corresponding spectral sequence 
will converge to $\cal{H}_{-n}( \cal{L}_\beta)$, while $d_n$ should give a spectral sequence to $\sln$ link homology. 
From the perspective of $\HHH(\cal{L}_\beta)$, it would make more sense to call Khovanov-Rozansky homology ``$\slnegn$ link homology,''
but we've chosen not to fight history in this paper.}
$d_{-n}$ determined by
\begin{equation}\label{eq:dnegn}
d_{-n}(t) = 0 \;\; , \;\; d_{-n}(\theta) = t^n
\end{equation}
that extends to all of $\C[t,\theta]/\theta^2$ by insisting that it is a DG algebra. 
The differential $d_{-n}$ is a map of $(q,a)$-degree $(2n,-2)$, 
so the degree giving the DG structure is one-half the $a$-degree. 
We denote this degree by $\deg_D$ for the remainder of this section.
Equation \eqref{eq:dnegn} also specifies a differential on 
the super vector space $\cal{V}_\infty = qa^{-1} \C[t,\theta]/\theta^2$ giving the HOMFLYPT homology of the unknot,  
by insisting that the latter is a DG module over $(\C[t,\theta]/\theta^2,d_{-n})$.
Here, our shifts do not affect the $D$-degree.
The differential $d_{-n}$ preserves $\deg_{Q_n} := \deg_q + n\deg_a$, 
and we find that
$\mathrm{H}^\bullet(\cal{V}_\infty, d_{-n}) \cong \cal{V}_n = q^{1-n} \C[t]/t^n$, 
provided we identify the $q$-degree on the latter with the $Q_n$-degree on the former.

\begin{proposition}\label{prop:negnss}
Let $\beta$ be a balanced, colored braid. 
The differential $d_{-n}$ induces a spectral sequence with first page $\HHH(\cal{L}_\beta)$ that converges to $\cal{H}_{-n}(\cal{L}_\beta)$.
\end{proposition}

\begin{proof}
The differential $d_{-n}$ induces a differential on $\cal{V}_\infty^{\otimes k}$ that we also denote by $d_{-n}$.
The symmetric group $\mathfrak{S}_k$ acts on the super vector space $\cal{V}_\infty^{\otimes k}$,
and this action commutes\footnote{To see this, note that the super-degree $\deg_\Pi$ equals $\deg_D (\mathrm{mod} 2)$, 
and recall that $\sigma_i \in \mathfrak{S}_k$ acts via 
$\sigma_i(v_1 \otimes \cdots \otimes v_i \otimes v_{i+1} \otimes \cdots \otimes v_k) = 
(-1)^{\deg_\Pi(v_i)\deg_\Pi(v_{i+1})} v_1 \otimes \cdots \otimes v_{i+1} \otimes v_i \otimes \cdots \otimes v_k$ on 
a tensor product of super vector spaces.}
with $d_{-n}$.
It follows that $(\Sym^k\cal{V}_\infty,d_{-n})$ is a subcomplex and the inclusion and projection maps 
\begin{equation}\label{eq:inclproj}
(\Sym^k\cal{V}_\infty,d_{-n}) \hookrightarrow (\cal{V}_\infty^{\otimes k}, d_{-n}) \text{  and  } (\cal{V}_\infty^{\otimes k}, d_{-n}) \twoheadrightarrow (\Sym^k\cal{V}_\infty,d_{-n})
\end{equation}
are chain maps. 
Moreover, it is easy to see that the endomorphisms $\id^{\otimes j-1} \otimes \multt \otimes \id^{\otimes k - j }: \cal{V}_\infty^{\otimes k} \rightarrow \cal{V}_\infty^{\otimes k}$ 
also commute with $d_{-n}$, for $j=1,\ldots,k$.

It follows that $d_{-n}$ determines a differential on the complex
\[
C_\circ(\widehat{\beta}) = a^{w_\beta-i_\beta} q^{i_\beta +W_\beta - w_\beta} \HH_\bullet(\hTr(\Psi_{BS}(\llbracket \widehat{\beta} \rrbracket_\circ))) \cong 
a^{w_\beta} q^{W_\beta - w_\beta} \Upsilon^\cal{S}_{\cal{V}_\infty,t}(\llbracket \widehat{\beta} \rrbracket_\circ)
\]
that commutes with the differential whose homology gives $\HHH(\cal{L}_\beta)$; we denote this latter differential by $\del$. 
Passing to the associated double complex, we obtain two spectral sequences 
by considering the ``horizontal'' and ``vertical'' filtrations. 

Viewing $\del$ as the horizontal differential, and $d_{-n}$ are the vertical one, 
the first page of the latter spectral sequence is the homology of $C_\circ(\widehat{\beta})$ with respect 
to $\del$, hence is $\HHH(\cal{L}_\beta)$. 
The differentials on this page of the spectral sequence are those induced by $d_{-n}$, 
and since $\HHH(\cal{L}_\beta)$ is supported in only finitely many $h$- and $a$-degrees,
the spectral sequence converges to 
the homology of the total complex $\Tot(C_\circ(\widehat{\beta}),\del, d_{-n})$. 

It remains to identify the homology of the total complex with $\cal{H}_{-n}(\cal{L}_\beta)$.
Considering the horizontal filtration, we obtain a spectral sequence whose first page is the 
homology of $C_\circ(\widehat{\beta})$ with respect to $d_{-n}$. 
To compute this first page, we first note that 
$\mathrm{H}^\bullet(\cal{V}_\infty^{\otimes k}, d_{-n}) \cong \cal{V}_n^{\otimes k}$ by the K\"{u}nneth theorem.
Further, it is easy to see that $\mathrm{H}^\bullet(\Sym^k\cal{V}_\infty,d_{-n}) \cong \Sym^k\cal{V}_n$, 
and more generally the K\"{u}nneth theorem shows that 
$\mathrm{H}^\bullet(\Sym^{k_1}\cal{V}_\infty \otimes \cdots \otimes \Sym^{k_m}\cal{V}_\infty,d_{-n}) \cong \Sym^{k_1}\cal{V}_n \otimes \cdots \otimes \Sym^{k_m}\cal{V}_n$.
It follows that the terms in $\mathrm{H}^\bullet(C_\circ(\widehat{\beta}),d_{-n})$ are precisely the terms in the complex 
$q^{n w_\beta +W_\beta - w_\beta} \Upsilon^\cal{S}_{\cal{V}_n,t}(\llbracket \widehat{\beta} \rrbracket_\circ)$. 

The maps induced on $\mathrm{H}^\bullet(-, d_{-n})$ by the inclusion and projection maps in equation \eqref{eq:inclproj} 
are precisely the inclusion and projection maps $\Sym^k\cal{V}_n \hookrightarrow \cal{V}_n$ and $\cal{V}_n \twoheadrightarrow \Sym^k\cal{V}_n$, 
and, the map induced by $\multt: \cal{V}_\infty \to \cal{V}_\infty$ on $(\cal{V}_\infty, d_{-n})$ is $\multt: \cal{V}_n \to \cal{V}_n$.
This implies that the differential on the first page, which is induced on $\mathrm{H}^\bullet(C_\circ(\widehat{\beta}),d_{-n})$ by $\del$, 
is the differential on the complex $q^{n w_\beta +W_\beta - w_\beta} \Upsilon^\cal{S}_{\cal{V}_n,t}(\llbracket \widehat{\beta} \rrbracket_\circ)$. 
It follows that the second page of this spectral sequence is precisely $\cal{H}_{-n}(\cal{L}_\beta)$.
Moreover, the differentials on all pages of the spectral sequence after the first are zero, as they don't preserve the $a$-degree, 
and the terms on the first page are all concentrated in a single $a$-degree. 
It follows that $\Tot(C_\circ(\widehat{\beta}),\del, d_{-n}) \cong \cal{H}_{-n}(\cal{L}_\beta)$.
\end{proof}

\begin{remark}
In the setting of reduced $\sln$ homology, 
there is a similar spectral sequence, due to Rasmussen \cite{Ras2}. 
Recent work of Naisse-Vaz \cite{NaVaz} gives a new construction of this spectral sequence, 
and shows that it converges at the second page, 
\ie that there are differentials $d_n$ on reduced $\HHH(\cal{L})$ whose homology give reduced $\KhR_n(\cal{L})$.
We will construct such a spectral sequence in the unreduced, colored setting in Section \ref{sec:diffn} below, 
and this construction mirrors our construction in Proposition \ref{prop:negnss}. 
We hence conjecture that our $\slnegn$ spectral sequence also converges at the second page.
\end{remark}

We close this section by verifying the only outstanding part of our proof of Theorem \ref{thm:negativen} from 
Section \ref{sec:SymHomology}, that Cautis's differential $\delta_n$ agrees with $d_{-n}$ on the HOMFLYPT 
homology of concentric, labeled circles.
As detailed in Section \ref{sec:SymHomology}, in Cautis's setup the invariant of circles labeled $k_1,\ldots k_m$  is 
$\HH_\bullet(R^{[k_1,\ldots,k_m]})$.
After identifying
\[
\HH_\bullet(R^{[1,\ldots,1]}) \cong \Omega^\bullet(\C^k) = 
\bigoplus_{0 \leq l \leq k} \C[t_1,\ldots,t_k] \cdot dt_{i_1} \wedge \cdots \wedge dt_{i_l},
\]
the differential $\delta_n$ is given 
by the action of the element 
\[
\gamma = \sum_{i=1}^k t_i^n \frac{\del}{\del t_i} \in \oplus_{i=1}^k \C[t_1,\ldots,t_k] \cdot \frac{\del}{\del t_i} \cong \HH^1(R^{[1,\ldots,1]}).
\]
More generally, since $\gamma$ is symmetric, 
it determines an element $\gamma_\mathbf{k} \in \HH^1(R^{[k_1,\ldots,k_m]})$ for every $\mathbf{k} = [k_1,\ldots,k_m]$, 
and the differential $\delta_n$ is given by the action of 
$\gamma_\mathbf{k}$ on $\HH_\bullet(R^{[k_1,\ldots,k_m]}) \subseteq \HH_\bullet(R^{[1,\ldots,1]})$. 
It thus suffices to check that $\delta_n$ agrees with $d_{-n}$ on $\HH_\bullet(R^{[1,\ldots,1]})$. 
We can identify $\Omega^\bullet(\C^k)$ with the super vector space $\cal{V}_\infty^{\otimes k}$ 
by identifying basis vectors
$t_1^{p_1} \cdots t_k^{p_k} \cdot dt_{i_1} \wedge \cdots \wedge dt_{i_l}$ (with $i_r < i_{r+1}$)
with the vectors $\theta^{j_1} t^{p_1} \otimes \cdots \otimes \theta^{j_k} t^{p_k}$, where $j_i = 1$ if $i \in \{i_1,\ldots,i_l \}$, 
and is zero otherwise. After doing so, it is clear that $\delta_n$ and $d_{-n}$ agree.

\subsection{The differential $d_{n}$}\label{sec:diffn}

As mentioned in Remark \ref{rem:HHHdecat}, our construction of $\HHH(\cal{L})$ naturally categorifies the symmetric webs presentation 
of the HOMFLYPT polynomial in Proposition \ref{prop:SymPoly}, so the spectral sequence from Section \ref{sec:diffnegn} can be viewed as 
a categorical analogue of specializing $a=q^n$ in that setting. 
This begs the question: can we also realize the categorical version of the specialization in Propositions \ref{prop:WedgePoly} and \ref{prop:WedgePoly2} 
using our framework?

Indeed, we can. 
To do so, we first give a variant of the construction of $\HHH(\cal{L})$ from Section \ref{sec:TriplyGraded} that naturally categorifies the 
description of the HOMFLYPT polynomial in Proposition \ref{prop:WedgePoly2}. We then introduce a differential $d_n$ that agrees with 
$d_{-n}$ on $1$-colored circles, but differs on circles of higher color. 
This differential induces a spectral sequence from HOMFLYPT link homology to $\sln$ Khovanov-Rozansky homology, 
and in the colored case between the $\bV$-colored variants of these theories.

Recall from Remark \ref{rem:KhRonthenose} that the version of $\sln$ Khovanov-Rozansky that categorifies the 
$\sln$ link polynomial, as formulated in Section \ref{sec:Decat}, is $\KhR_n^\vee(\cal{L})$. 
We begin by defining an analogous version of HOMFLYPT link homology, which we denote $\HHH^\vee(\cal{L})$. 
Let $\vFoam^{\vee}$ denote the version of the $2$-category $\vFoam$ in which the $q$-degree has been negated.
Similarly, given a colored braid $\beta$, 
let $\llbracket \beta \rrbracket^{\vee} \in \Hot^b(\vFoam^{\vee})$ be given by negating the grading shifts in equations \eqref{eq:slnCrossing} and \eqref{eq:slnColoredCrossing}, 
\ie by replacing every instance of $q$ by $q^{-1}$. 
Repeating the arguments in Sections \ref{sec:Foams} and \ref{sec:CatUqTrace}, with all $q$-degrees negated,
shows that the homology of the complex $\Upsilon^{\square}_{\cal{V},T}(\llbracket \widehat{\beta} \rrbracket_\circ^{\vee})$ 
is an invariant of the annular link $\widehat{\beta}$ for any graded super vector space $\cal{V}$ and endomorphism $T$ of degree $-2$.

Set $\cal{V}_\infty^\vee := q^{-1}a \C[\mathtt{t},\vartheta]/\vartheta^2$,  
where $\mathtt{t}$ is an even variable with $\deg_{q,a}(\mathtt{t}) = (-2,0)$ 
and $\vartheta$ is an odd variable with $\deg_{q,a}(\vartheta) = (0,-2)$. 
Given a balanced, colored braid $\beta$, define
\[
\HHH^\vee(\cal{L}_\beta) := 
\mathrm{H}^\bullet(h^{w_\beta} a^{w_\beta} q^{w_\beta - W_\beta} \Upsilon^{\wedge}_{\cal{V}_\infty^{\vee},\mathtt{t}}(\llbracket \widehat{\beta} \rrbracket_\circ^{\vee})).
\]

\begin{proposition}
Given a balanced, colored braid $\beta$, 
$\chi\big(\HHH^\vee(\cal{L}_\beta)\big) = P_\infty(a,q;\cal{L}_\beta^\wedge)$, 
where the latter denotes the $\bV$-colored HOMFLYPT polynomial.
\end{proposition}
\begin{proof}
We first note that
\[
\dim_{q,a}\big( \bV^k \cal{V}_\infty^\vee \big) = {\displaystyle q^{-k} a^k \prod_{i=1}^k \frac{q^{2-2i} - a^{-2}}{1-q^{-2i}} = (-1)^k q^{-k} a^{-k} \prod_{i=1}^k \frac{1-a^2q^{2-2i}}{1-q^{-2i}}
= \mathsf{c}_{\wedge^k} }
\]
which follows using Remark \ref{rem:FunctorsWedge} to obtain a variant of the isomorphism in equation \eqref{eq:FundThmSymDiff}.
We then compute that
\[
\begin{aligned}
\chi\big( \HHH^\vee(\cal{L}_\beta) \big) &= 
(-1)^{w_\beta} a^{w_\beta} q^{w_\beta - W_\beta} (-1)^{\sum_{c \in \beta} k_c l_c} \langle \widehat{\beta} \rangle_\circ \big|_{\mathsf{c}_{\wedge^\bullet}} \\
&= (-1)^{\sum_{d \in \beta} k_d l_d} a^{w_\beta} (-q)^{w_\beta - W_\beta} \langle \widehat{\beta} \rangle_\circ \big|_{\mathsf{c}_{\wedge^\bullet}}
\end{aligned}
\]
Here $\sum_{c \in \beta} k_c l_c$ denotes the sum over the crossings $c$ in $\beta$ of the labels $k_c$ and $l_c$ on the involved strands, 
and $\sum_{d \in \beta} k_d l_d$ is similarly defined, but only taking the sum over crossings $d$ in which two distinct colors are present.
Since $\beta$ is balanced, the latter must be even. The result then follows from Remark \ref{rem:Coloring}.
\end{proof}

In Proposition \ref{prop:Mirror} below, we show that $\HHH^\vee(\cal{L}_\beta)$ is an invariant of the colored link $\cal{L}_\beta \subset S^3$
(up to degree shifts). Before doing so, we construct our spectral sequence to Khovanov-Rozansky homology.

\begin{proposition}\label{prop:nss}
Let $\beta$ be a balanced, colored braid. 
There exists a differential $d_{n}$ on the complex $\Upsilon^{\wedge}_{\cal{V}_\infty^{\vee},\mathtt{t}}(\llbracket \widehat{\beta} \rrbracket_\circ^{\vee})$ that
determines a spectral sequence with first page $\HHH^\vee(\cal{L}_\beta)$ and converging to $\KhR_n^\vee(\cal{L}_\beta)$.
\end{proposition}

\begin{proof}
The construction of this spectral sequence exactly parallels that in Proposition \ref{prop:negnss}, simply trading $\Sym^k$ for $\bV^k$ throughout.

First, define the differential $d_n$ on $\C[\mathtt{t},\vartheta]/\vartheta^2$ by taking $d_n(\mathtt{t}) = 0, d_n(\vartheta) = \mathtt{t}^n$, 
and extend to $\C[\mathtt{t},\vartheta]/\vartheta^2$ by demanding that it be a DG algebra. 
We then obtain a differential on $\cal{V}_\infty^\vee$ by requiring that it is a DG module. 
This differential raises the $D$-degree (which again is half the $a$-degree, before imposing the shifts defining $\cal{V}_\infty^\vee$), 
and preserves the $Q_n$-degree, which again is defined as $\deg_{Q_n} := \deg_q + n\deg_a$.

As above, there is an induced differential on the tensor product $(\cal{V}_\infty^\vee)^{\otimes k}$ commuting with the symmetric group action, 
so $(\bV^k\cal{V}_\infty^\vee,d_n)$ is a subcomplex, and the inclusion and projection maps 
\[
(\bV^k\cal{V}_\infty^\vee,d_n) \hookrightarrow ((\cal{V}_\infty^\vee)^{\otimes k}, d_n) \text{  and  } ((\cal{V}_\infty^\vee)^{\otimes k}, d_n) \twoheadrightarrow (\bV^k\cal{V}_\infty^\vee,d_n)
\]
are chain maps. Similarly, the endomorphisms 
$\id^{\otimes j-1} \otimes \mathtt{t} \otimes \id^{\otimes k - j }: (\cal{V}_\infty^\vee)^{\otimes k} \rightarrow (\cal{V}_\infty^\vee)^{\otimes k}$ 
again commute with $d_n$, so the latter determines a differential on the complex
\[
C_\circ^\vee(\widehat{\beta}) := h^{w_\beta} a^{w_\beta} q^{w_\beta - W_\beta} \Upsilon^{\wedge}_{\cal{V}_\infty^{\vee},\mathtt{t}}(\llbracket \widehat{\beta} \rrbracket_\circ^{\vee})
\]
that commutes with the differential $\del$ coming from $\llbracket \widehat{\beta} \rrbracket_\circ^{\vee}$.

Repeating the arguments in Proposition \ref{prop:negnss}, 
one of the filtrations on this double complex produces 
a spectral sequence with first page $\HHH^\vee(\cal{L}_\beta)$,
and converging to the homology of the total complex $\Tot(C^\vee_\circ(\widehat{\beta}),\del, d_n)$. 
Defining $\cal{V}_n^{\vee} = q^{n-1} \C[\mathtt{t}]/\mathtt{t}^n$, we can proceed as above and 
use the spectral sequence arising from the other filtration to identify the homology of the total complex with the homology of the complex 
$h^{w_\beta} q^{nw_\beta + w_\beta - W_\beta} \Upsilon^{\wedge}_{\cal{V}_n^{\vee},\mathtt{t}}(\llbracket \widehat{\beta} \rrbracket_\circ^{\vee})$, 
after identifying the $q$-degree on the latter with the $Q_n$-degree on the former. 
The homology of this latter complex is precisely $\KhR_n^\vee(\cal{L}_\beta)$.
\end{proof}

\begin{remark}
Our spectral sequence can be viewed as an unreduced, $\bV$-colored analogue of the Rasmussen spectral sequence.
Such a spectral sequence was also constructed in \cite{Wedrich}, using the theory of matrix factorizations.
\end{remark}

\subsection{Mirror symmetry}\label{sec:Mirror}

In the previous section, we assigned a triply-graded vector space $\HHH^\vee(\cal{L}_\beta)$ to a balanced, colored braid $\beta$ 
that decategorifies to the $\bV$-colored HOMFLYPT polynomial of the closure $\cal{L}_\beta \subset S^3$. 
In this section, we show that, up to a grading shift, this invariant is isomorphic to $\HHH(\cal{L}_\beta)$. 
This allows us to produce a spectral sequence $\HHH(\cal{L}_\beta) \Rightarrow \KhR_n(\cal{L}_\beta)$, 
and provides a categorical interpretation of the mirror symmetry property of the colored HOMFLYPT polynomial.

Recall that the HOMFLYPT polynomial
for (uncolored) links $\cal{L} \subset S^3$ 
can be defined as the unique link invariant $P_\infty(\cal{L})$ satisfying the relations:
\begin{equation}\label{eq:HOMFLYPTskein}
a^{-1}
P_\infty\left( \;
\xy
(0,0)*{
\begin{tikzpicture}[scale=.5,rotate=270]
	\draw[very thick, <-] (0,1) to [out=0,in=180] (2,0);
	\draw[very thick, <-] (0,0) to [out=0,in=225] (.9,.4);
	\draw[very thick] (1.1,.6) to [out=45,in=180] (2,1);
\end{tikzpicture}
};
\endxy \;
\right)
-
a
P_\infty\left( \;
\xy
(0,0)*{
\begin{tikzpicture}[scale=.5, rotate=270]
	\draw[very thick, <-] (0,0) to [out=0,in=180] (2,1);
	\draw[very thick, <-] (0,1) to [out=0,in=135] (.9,.6);
	\draw[very thick] (1.1,.4) to [out=315,in=180] (2,0);
\end{tikzpicture}
};
\endxy \;
\right) =
(q^{-1} - q)
P_\infty\left( \;
\xy
(0,0)*{
\begin{tikzpicture}[scale=.5, rotate=270]
	\draw[very thick, <-] (0,1) to [out=340,in=200] (2,1);
	\draw[very thick, <-] (0,0) to [out=20,in=160] (2,0);
\end{tikzpicture}
};
\endxy \;
\right)
\;\; , \;\;
P_\infty
\left(
\xy
(0,0)*{
\begin{tikzpicture}[scale=.25]
	\draw[very thick] (0,0) circle (1);
	\draw[very thick,->] (1,0) to (1,-.1);
\end{tikzpicture}
};
\endxy
\right)
= \frac{a-a^{-1}}{q-q^{-1}}
\end{equation}
In this section, 
we use the notation $P_\infty(a,q;\cal{L})$ to emphasize the variables.
Since both relations are invariant under the substitution $q \leftrightarrow -q^{-1}$, 
we have the so-called ``mirror symmetry'' 
\begin{equation}\label{eq:uncoloredMirror}
P_\infty(a,q;\cal{L})=P_\infty(a,-q^{-1};\cal{L})
\end{equation}
for the HOMFLYPT polynomial.
More generally, for a 
link $\cal{L}^\mathbf{\lambda}$ with $\# \cal{L}$ components colored by partitions $\Lambda=(\lambda_1,\ldots,\lambda_{\# \cal{L}})$,
this symmetry takes the form
\begin{equation}\label{eq:mirror}
P_\infty \left(a,q;\cal{L}^{\Lambda}\right)=P_\infty \left(a,-q^{-1};\cal{L}^{\Lambda^{\mathsf{T}}}\right)
\end{equation}
where $\Lambda^\mathsf{T}=(\lambda_1^\mathsf{T},\ldots,\lambda_{\# \cal{L}}^\mathsf{T})$ consists of the conjugates of the partitions in $\Lambda$.

The skein relation in equation \eqref{eq:HOMFLYPTskein} implies that $P_\infty(\cal{L})$ 
is an odd (respectively even) function of $q$ 
when the (uncolored) link $\cal{L}$ has an odd (resp. even) number of components $\#\cal{L}$. 
It follows that the (uncolored) mirror symmetry \eqref{eq:uncoloredMirror} can equivalently be written as
\[
P_\infty(a,q;\cal{L})= (-1)^{\# \cal{L}}P_\infty(a,q^{-1};\cal{L})
\]
which is precisely equation \eqref{eq:ourMirror} from Remark \ref{rem:Mirror} in the introduction. 
For a link with components colored by $\Lambda=(\lambda_1,\ldots,\lambda_{\# \cal{L}})$, 
equation \eqref{eq:mirror} can be re-written as 
\begin{equation}\label{eq:ourMirror2}
P_\infty(a,q;\cal{L}^\Lambda) = (-1)^{| \Lambda |}P_\infty(a,q^{-1};\cal{L}^{\Lambda^\mathsf{T}})
\end{equation}
for $| \Lambda| = \sum_{i=1}^{\#\cal{L}} | \lambda_i |$. 
This re-formulation is well-known, see \eg \cite{TVW}, and, 
in the case that $\Lambda$ consists only of $1$-row partitions, 
it follows by decategorifying our next result.

\begin{proposition}\label{prop:Mirror}
There is an isomorphism $\HHH(\cal{L}_\beta) \cong h^{-w_\beta} \Pi^{i_\beta} \HHH^\vee(\cal{L}_\beta)$ 
that preserves $a$- and $h$-degree, and negates $q$-degree. 
This isomorphism is a categorical manifestation of mirror symmetry.
\end{proposition}

\begin{proof}
Consider the isomorphism $\cal{M}:\cal{V}_\infty \to \Pi \cal{V}_\infty^\vee$ 
defined by sending $t^k \mapsto \vartheta \mathtt{t}^k$ and $\theta t^k \to \mathtt{t}^k$. 
This gives an isomorphism $\cal{V}_\infty \otimes \C^m \to \Pi \cal{V}_\infty^\vee \otimes \C^m$ commuting with the 
action of $\U(\glm[t])$,
and hence induces an isomorphism 
\[
\cal{M}: \Sym^k(\cal{V}_\infty \otimes \C^m) \xrightarrow{\cong} \Sym^k(\Pi \cal{V}_\infty^\vee \otimes \C^m) \cong \Pi^k  \bV^k(\cal{V}_\infty^\vee \otimes \C^m)
\]
of $\U(\glm[t])$-modules, for each $k$. This isomorphism preserves $a$-degree, but negates the $q$-degree.
Given a balanced, colored braid $\beta$, this gives a $q$-degree negating chain isomorphism
\begin{equation}\label{eq:Miso}
\cal{M}:
\Upsilon^{\cal{S}}_{\cal{V}_\infty,t} (a^{w_\beta} q^{W_\beta - w_\beta} \llbracket \widehat{\beta} \rrbracket_\circ ) \xrightarrow{\cong} 
\Pi^{i_\beta} \Upsilon^{\wedge}_{\cal{V}_\infty^\vee,\mathtt{t}} (a^{w_\beta} q^{w_\beta - W_\beta} \llbracket \widehat{\beta} \rrbracket^\vee_\circ )
\end{equation}
since the differentials in both complexes are given by the current algebra action.
Taking homology now gives the isomorphism 
\begin{equation}\label{eq:Miso2}
\HHH(\cal{L}_\beta) \xrightarrow{\cong} h^{-w_\beta} \Pi^{i_\beta} \HHH^\vee(\cal{L}_\beta).
\end{equation}

To see that this categorifies equation \eqref{eq:ourMirror2}, 
note that for a balanced, colored braid $\beta$, $\chi\big( \HHH(\cal{L}_\beta) \big) = P_\infty(a,q;\cal{L}^\Lambda)$, 
where $\Lambda$ consists of the $1$-row partitions corresponding to the colors of the strands in $\beta$. 
Similarly, $\chi\big( \HHH^\vee(\cal{L}_\beta) \big) = P_\infty(a,q;\cal{L}^{\Lambda^\mathsf{T}})$.
The result then holds since the parity of $i_\beta - w_\beta$ equals the parity of $\# \cal{L}_\beta$, 
and because the isomorphism in equation \eqref{eq:Miso2} negates $q$-degree.
\end{proof}

\begin{corollary}
Given a balanced, colored braid $\beta$, 
there is a spectral sequence $\HHH(\cal{L}_\beta) \Rightarrow \KhR_n(\cal{L}_\beta)$.
\end{corollary}

\begin{proof}
Conjugating the differential $d_n$ by the isomorphism in equation \eqref{eq:Miso} 
defines the relevant differential $\cal{M}^{-1} \circ d_n \circ \cal{M}$ on the complex $C_\circ(\widehat{\beta})$. 
The result then follows from Proposition \ref{prop:nss}.
\end{proof}

\begin{remark}
There is a much more interesting symmetry of HOMFLYPT link homology conjectured in \cite{GGS}, 
which categorifies the mirror symmetry of the reduced HOMFLYPT polynomial.
The latter is defined using the skein relation in equation \eqref{eq:HOMFLYPTskein}, 
but replacing the value of the unknot by $1$. 
This invariant is categorified by viewing $\HHH(\cal{L})$ as a module over the unknot, 
and tensoring over this with $\C$.
We do not address this more-sophisticated result in our present work, 
but hope that our perspective on unreduced mirror symmetry will be useful in future considerations.
\end{remark}

\section{Some computations}\label{sec:examples}

In this section, we compute a few simple examples 
that suffice to confirm various assertions about our link invariants.

\subsection{The once-stabilized unknot}

The annular closure of a crossing gives the simplest non-trivial annular knot. 
In the case of a positive crossing $\sigma_+$, we compute
\begin{equation}\label{eq:StabUnknot}
\left \llbracket
\begin{tikzpicture}[scale=.3,rotate=90, anchorbase]
\draw[gray] (1,0) circle (2.5 and 3.5);
\draw[gray] (1,0) circle (.15 and .2);
	\draw[very thick] (-.5,-1) to [out=90,in=270] (.5,1);
	\draw[very thick] (.5,-1) to [out=90,in=315] (.2,-.2);
	\draw[very thick] (-.2,.2) to [out=135,in=270] (-.5,1);
\draw[very thick, directed=.575] (-.5,1) to (-.5,1.5) to [out=90,in=180] (1,3) to [out=0,in=90] (2.5,1.5) to (2.5,-1.5) to [out=270,in=0] (1,-3) to [out=180,in=270] (-.5,-1.5) to (-.5,-1);
\draw[very thick, directed=.65] (.5,1) to (.5,1.5) to [out=90,in=180] (1,2) to [out=0,in=90] (1.5,1.5) to (1.5,-1.5) to [out=270,in=0] (1,-2) to [out=180,in=270] (.5,-1.5) to (.5,-1);
\end{tikzpicture}
\right \rrbracket
\cong q \;
\begin{tikzpicture}[scale=.25,rotate=90, anchorbase]
\draw[gray] (1,0) circle (2 and 2.5);
\draw[gray] (1,0) circle (.15 and .2);
\draw[very thick] (1,0) circle (1.5 and 2);
\draw[very thick] (1,0) circle (.75 and 1);
\end{tikzpicture}
\xrightarrow{
\begin{pmatrix} 
\xy
(0,0)*{
\begin{tikzpicture}[scale=.125]
	\draw [thick] (-1.5,0) to (0,1);
	\draw [thick] (1.5,0) to (0,1);
	\draw [thick] (0,1) to (0,3);
\end{tikzpicture}
};
\endxy 
\\ 
\xy
(0,0)*{
\begin{tikzpicture}[scale=.125]
	\draw [thick] (-1.5,0) to (0,1);
	\draw [thick] (1.5,0) to (0,1);
	\draw [thick] (0,1) to (0,3);
	\node at (-.65,.5) {\tiny$\bullet$};
\end{tikzpicture}
};
\endxy 
\end{pmatrix}
}
\uwave{
q \;
\begin{tikzpicture}[scale=.25,rotate=90, anchorbase]
\draw[gray] (1,0) circle (2 and 2.5);
\draw[gray] (1,0) circle (.15 and .2);
\draw[very thick] (1,0) circle (1 and 1.33);
\node at (0,-1.5) {\tiny$2$};
\end{tikzpicture}
\oplus
q^{-1} \;
\begin{tikzpicture}[scale=.25,rotate=90, anchorbase]
\draw[gray] (1,0) circle (2 and 2.5);
\draw[gray] (1,0) circle (.15 and .2);
\draw[very thick] (1,0) circle (1 and 1.33);
\node at (0,-1.5) {\tiny$2$};
\end{tikzpicture}
}
\end{equation}
Here, and in the following examples, we use the notation from 
Section \ref{sec:dotted} in depicting our annular foams.
Applying the functor $\Upsilon_{\C^n,0}^\cal{S}$ and computing homology, 
we find that
\[
\cal{AH}^i_{-n}(\widehat{\sigma_+}) = 
	\begin{cases}
		q {\bV^2 \C^n} &\text{if } i=-1 \\
		q^{-1} \Sym^2 \C^n &\text{if } i=0 \\
		0 &\text{else }
	\end{cases}
\]
If we instead use $\Upsilon_{\C^n,0}^\wedge$, we find that 
\[
\cal{A}\KhR^i_{n}(\widehat{\sigma_+}) = 
	\begin{cases}
		q \Sym^2 \C^n &\text{if } i=0 \\
		q^{-1} {\bV^2 \C^n} &\text{if } i=1 \\
		0 &\text{else }
	\end{cases}
\]
This confirms that neither $\cal{AH}_{-n}$ nor $\cal{A}\KhR_{n}$ is an invariant of links in $S^3$. 
Indeed, $\cal{L}_{\sigma_+}$ is an unknot, and applying either $\cal{AH}_{-n}$ or $\cal{A}\KhR_{n}$ to the annular closure of the $1$-strand braid 
(which also closes to an unknot in $S^3$) gives $\C^n$. 
It also suggests that $\cal{AH}_{-n} \ncong \cal{A}\KhR_{n}$ (although in this case they are isomorphic as non-graded vector spaces -- 
we will see an example where this is not the case below).

We quickly note that if we instead apply $\Upsilon_{\cal{V}_n,t}^\cal{S}$ to \eqref{eq:StabUnknot}, together with the relevant shifts from Theorem \ref{thm:negativen}, 
then $\cal{H}_{-n}(\cal{L}_{\sigma_+})$ is the homology of a complex that is homotopy equivalent to the complex 
\[
q^{n+1} {\bV^2\cal{V}_n} \to \uwave{q^{n-1} \Sym^2\cal{V}_n}
\]
in which the differential is injective. If we instead apply $\Upsilon_{\cal{V}_n,t}^\wedge$ and the shifts from Theorem \ref{thm:sln}, 
we find that $\KhR_n(\cal{L}_{\sigma_+})$ is the homology of a complex that is homotopy equivalent to the complex
\[
\uwave{q^{1-n} \Sym^2\cal{V}_n}  \to q^{-1-n} {\bV^2\cal{V}_n}
\]
with surjective differential. 
In both cases, 
the homology is precisely isomorphic to $\cal{V}_n$ concentrated in homological degree zero, as expected.

\subsection{The Hopf link}\label{sec:Hopf}

The Hopf link is the closure of the braid $\sigma_+^2$, and we compute
\[
\left \llbracket
\begin{tikzpicture}[scale=.3,rotate=90, anchorbase]
\draw[gray] (1,0) circle (2.5 and 3.5);
\draw[gray] (1,0) circle (.15 and .2);
	\draw[very thick] (-.5,-1.5) to [out=90,in=270] (.5,0);
	\draw[very thick] (.5,-1.5) to [out=90,in=315] (.2,-.95);
	\draw[very thick] (-.2,-.55) to [out=135,in=270] (-.5,0);
	\draw[very thick] (-.5,0) to [out=90,in=270] (.5,1.5);
	\draw[very thick] (.5,0) to [out=90,in=315] (.2,.55);
	\draw[very thick] (-.2,.95) to [out=135,in=270] (-.5,1.5);
\draw[very thick, directed=.575] (-.5,1.5) to [out=90,in=180] (1,3) to [out=0,in=90] (2.5,1.5) to (2.5,-1.5) to [out=270,in=0] (1,-3) to [out=180,in=270] (-.5,-1.5);
\draw[very thick, directed=.65] (.5,1.5) to [out=90,in=180] (1,2) to [out=0,in=90] (1.5,1.5) to (1.5,-1.5) to [out=270,in=0] (1,-2) to [out=180,in=270] (.5,-1.5);
\end{tikzpicture}
\right \rrbracket
\simeq
q^2 \;
\begin{tikzpicture}[scale=.25,rotate=90, anchorbase]
\draw[gray] (1,0) circle (2 and 2.5);
\draw[gray] (1,0) circle (.15 and .2);
\draw[very thick] (1,0) circle (1.5 and 2);
\draw[very thick] (1,0) circle (.75 and 1);
\end{tikzpicture}
\xrightarrow{\begin{pmatrix} 
\xy
(0,0)*{
\begin{tikzpicture}[scale=.125]
	\draw [thick] (-1.5,0) to (0,1);
	\draw [thick] (1.5,0) to (0,1);
	\draw [thick] (0,1) to (0,3);
\end{tikzpicture}
};
\endxy 
\\ 
\xy
(0,0)*{
\begin{tikzpicture}[scale=.125]
	\draw [thick] (-1.5,0) to (0,1);
	\draw [thick] (1.5,0) to (0,1);
	\draw [thick] (0,1) to (0,3);
	\node at (-.65,.5) {\tiny$\bullet$};
\end{tikzpicture}
};
\endxy 
\end{pmatrix}}
q^2 \;
\begin{tikzpicture}[scale=.25,rotate=90, anchorbase]
\draw[gray] (1,0) circle (2 and 2.5);
\draw[gray] (1,0) circle (.15 and .2);
\draw[very thick] (1,0) circle (1 and 1.33);
\node at (0,-1.5) {\tiny$2$};
\end{tikzpicture}
\oplus
\begin{tikzpicture}[scale=.25,rotate=90, anchorbase]
\draw[gray] (1,0) circle (2 and 2.5);
\draw[gray] (1,0) circle (.15 and .2);
\draw[very thick] (1,0) circle (1 and 1.33);
\node at (0,-1.5) {\tiny$2$};
\end{tikzpicture}
\xrightarrow{\begin{pmatrix} 0 & 0 \\ 0 & 0 \end{pmatrix}}
\uwave{
\begin{tikzpicture}[scale=.25,rotate=90, anchorbase]
\draw[gray] (1,0) circle (2 and 2.5);
\draw[gray] (1,0) circle (.15 and .2);
\draw[very thick] (1,0) circle (1 and 1.33);
\node at (0,-1.5) {\tiny$2$};
\end{tikzpicture}
\oplus
q^{-2} \;
\begin{tikzpicture}[scale=.25,rotate=90, anchorbase]
\draw[gray] (1,0) circle (2 and 2.5);
\draw[gray] (1,0) circle (.15 and .2);
\draw[very thick] (1,0) circle (1 and 1.33);
\node at (0,-1.5) {\tiny$2$};
\end{tikzpicture}
}
\]
Our results above imply that
\[
\cal{AH}^i_{-n}(\widehat{\sigma^2_+}) = 
	\begin{cases}
		q^2 {\bV^2 \C^n} &\text{if } i=-2 \\
		\Sym^2 \C^n &\text{if } i=-1 \\
		\Sym^2 \C^n \oplus q^{-2} \Sym^2 \C^n &\text{if } i=0 \\
		0 &\text{else }
	\end{cases}
\; , \;
\cal{A}\KhR^i_{n}(\widehat{\sigma^2_+}) = 
	\begin{cases}
		q^2 \Sym^2 \C^n &\text{if } i=0 \\
		{\bV^2 \C^n} &\text{if } i=1 \\
		{\bV^2 \C^n} \oplus q^{-2} {\bV^2 \C^n} &\text{if } i=2 \\
		0 &\text{else }
	\end{cases}
\]
and this confirms that, in general, $\cal{AH}_{-n}(\widehat{\beta}) \ncong \cal{A}\KhR_{n}(\widehat{\beta})$.

Similarly, the $\slnegn$ and $\sln$ homology of the Hopf link 
$\cal{L}_{\sigma^2_+} \subset S^3$ are:
\[
\cal{H}^i_{-n}(\cal{L}_{\sigma^2_+}) = 
	\begin{cases}
		q^{1+n} \cal{V}_n &\text{if } i=-1 \\
		q^{2n} \Sym^2 \cal{V}_n \oplus q^{2n-2} \Sym^2 \cal{V}_n &\text{if } i=0 \\
		0 &\text{else }
	\end{cases}
\]
and
\[
\KhR^i_{n}(\cal{L}_{\sigma^2_+}) = 
	\begin{cases}
		q^{1-n} \cal{V}_n &\text{if } i=0 \\
		q^{-2n} {\bV^2 \cal{V}_n} \oplus q^{-2-2n} {\bV^2 \cal{V}_n} &\text{if } i=2 \\
		0 &\text{else }
	\end{cases}
\]
This shows that, in general, for links $\cal{L} \subset S^3$ we have $\cal{H}_{-n}(\cal{L}) \ncong \KhR_{n}(\cal{L})$. 
Moreover, this also shows that $\cal{H}_{-2}(\cal{L})$ is generally distinct from odd Khovanov homology, 
since for the Hopf link the former has rank eight, while the latter has rank four.

\subsection{The right-handed trefoil}\label{sec:Trefoil}

We conclude by computing the Khovanov homology $\KhR_2$ and the $\glnn{-2}$ homology of 
the right-handed trefoil $\cal{T} \subset S^3$, confirming that these invariants are distinct for the simplest non-trivial knot.
This trefoil can be presented as the closure of the braid $\sigma^3_+$. 
Recall from~\cite{QR2} that the complex of annular foams assigned to 
\[
\widehat{\sigma^3_+} = 
\begin{tikzpicture}[scale=.3, rotate=90, anchorbase]
\draw[gray] (1,0) circle (2.5 and 3.5);
\draw[gray] (1,0) circle (.15 and .2);
\draw [very thick] (.5,-1.5) .. controls (.5,-1) and (-.5,-1) .. (-.5,-.5);
\draw [draw =white, double=black, very thick, double distance=1.25pt] (-.5,-1.5) .. controls (-.5,-1) and (.5,-1) .. (.5,-.5);
\draw [very thick] (.5,-.5) .. controls (.5,0) and (-.5,0) .. (-.5,.5);
\draw [draw =white, double=black, very thick, double distance=1.25pt] (-.5,-.5) .. controls (-.5,0) and (.5,0) .. (.5,.5);
\draw [very thick] (.5,.5) .. controls (.5,1) and (-.5,1) .. (-.5,1.5);
\draw [draw =white, double=black, very thick, double distance=1.25pt] (-.5,.5) .. controls (-.5,1) and (.5,1) .. (.5,1.5);
\draw[very thick, directed=.575] (-.5,1.5) to [out=90,in=180] (1,3) to [out=0,in=90] (2.5,1.5) to (2.5,-1.5) to [out=270,in=0] (1,-3) to [out=180,in=270] (-.5,-1.5);
\draw[very thick, directed=.65] (.5,1.5) to [out=90,in=180] (1,2) to [out=0,in=90] (1.5,1.5) to (1.5,-1.5) to [out=270,in=0] (1,-2) to [out=180,in=270] (.5,-1.5);
\end{tikzpicture}
\]
is given (using our present normalizations and notation) by: 
\[
\llbracket \widehat{\sigma^3_+} \rrbracket
\simeq
q^3 \;
\begin{tikzpicture}[scale=.25,rotate=90, anchorbase]
\draw[gray] (1,0) circle (2 and 2.5);
\draw[gray] (1,0) circle (.15 and .2);
\draw[very thick] (1,0) circle (1.5 and 2);
\draw[very thick] (1,0) circle (.75 and 1);
\end{tikzpicture}
\xrightarrow{A}
q^3 \;
\begin{tikzpicture}[scale=.2,rotate=90, anchorbase]
\draw[gray] (1,0) circle (2 and 2.5);
\draw[gray] (1,0) circle (.15 and .2);
\draw[very thick] (1,0) circle (1 and 1.33);
\node at (0,-1.5) {\tiny$2$};
\end{tikzpicture}
\oplus
q \;
\begin{tikzpicture}[scale=.2,rotate=90, anchorbase]
\draw[gray] (1,0) circle (2 and 2.5);
\draw[gray] (1,0) circle (.15 and .2);
\draw[very thick] (1,0) circle (1 and 1.33);
\node at (0,-1.5) {\tiny$2$};
\end{tikzpicture}
\xrightarrow{B}
q \;
\begin{tikzpicture}[scale=.2,rotate=90, anchorbase]
\draw[gray] (1,0) circle (2 and 2.5);
\draw[gray] (1,0) circle (.15 and .2);
\draw[very thick] (1,0) circle (1 and 1.33);
\node at (0,-1.5) {\tiny$2$};
\end{tikzpicture}
\oplus
q^{-1} \;
\begin{tikzpicture}[scale=.2,rotate=90, anchorbase]
\draw[gray] (1,0) circle (2 and 2.5);
\draw[gray] (1,0) circle (.15 and .2);
\draw[very thick] (1,0) circle (1 and 1.33);
\node at (0,-1.5) {\tiny$2$};
\end{tikzpicture}
\xrightarrow{C}
\uwave{
q^{-1} \;
\begin{tikzpicture}[scale=.2,rotate=90, anchorbase]
\draw[gray] (1,0) circle (2 and 2.5);
\draw[gray] (1,0) circle (.15 and .2);
\draw[very thick] (1,0) circle (1 and 1.33);
\node at (0,-1.5) {\tiny$2$};
\end{tikzpicture}
\oplus
q^{-3} \;
\begin{tikzpicture}[scale=.2,rotate=90, anchorbase]
\draw[gray] (1,0) circle (2 and 2.5);
\draw[gray] (1,0) circle (.15 and .2);
\draw[very thick] (1,0) circle (1 and 1.33);
\node at (0,-1.5) {\tiny$2$};
\end{tikzpicture}
}
\]
where the matrices specifying the differential are as follows:
\[
A = \begin{pmatrix} 
\xy
(0,0)*{
\begin{tikzpicture}[scale=.125]
	\draw [thick] (-1.5,0) to (0,1);
	\draw [thick] (1.5,0) to (0,1);
	\draw [thick] (0,1) to (0,3);
\end{tikzpicture}
};
\endxy 
\\ 
\xy
(0,0)*{
\begin{tikzpicture}[scale=.125]
	\draw [thick] (-1.5,0) to (0,1);
	\draw [thick] (1.5,0) to (0,1);
	\draw [thick] (0,1) to (0,3);
	\node at (-.65,.5) {\tiny$\bullet$};
\end{tikzpicture}
};
\endxy 
\end{pmatrix}
\;\; , \;\;
B = \begin{pmatrix} 0 & 0 \\ 0 & 0 \end{pmatrix}
\;\; , \;\;
C = \begin{pmatrix} 
{-}
\xy
(0,0)*{
\begin{tikzpicture}[scale=.2, anchorbase]
	\draw [thick] (0,0) to (0,1);
	\draw [thick] (0,1) to [out=150,in=210] (0,3);
	\draw [thick] (0,1) to [out=30,in=330] (0,3);	
	\draw [thick] (0,3) to (0,4);
	\node at (.5,2) {\tiny$\bullet$};
\end{tikzpicture}
};
\endxy 
&
\xy
(0,0)*{
\begin{tikzpicture}[scale=.2, anchorbase]
	\draw [thick] (0,0) to (0,1);
	\draw [thick] (0,1) to [out=150,in=210] (0,3);
	\draw [thick] (0,1) to [out=30,in=330] (0,3);	
	\draw [thick] (0,3) to (0,4);
\end{tikzpicture}
};
\endxy 
\\ 
{-}
\xy
(0,0)*{
\begin{tikzpicture}[scale=.2, anchorbase]
	\draw [thick] (0,0) to (0,1);
	\draw [thick] (0,1) to [out=150,in=210] (0,3);
	\draw [thick] (0,1) to [out=30,in=330] (0,3);	
	\draw [thick] (0,3) to (0,4);
	\node at (-.5,2) {\tiny$\bullet$};
	\node at (.5,2) {\tiny$\bullet$};
\end{tikzpicture}
};
\endxy 
&
\xy
(0,0)*{
\begin{tikzpicture}[scale=.2, anchorbase]
	\draw [thick] (0,0) to (0,1);
	\draw [thick] (0,1) to [out=150,in=210] (0,3);
	\draw [thick] (0,1) to [out=30,in=330] (0,3);	
	\draw [thick] (0,3) to (0,4);
	\node at (-.5,2) {\tiny$\bullet$};
\end{tikzpicture}
};
\endxy 
\end{pmatrix}
\]
Since the second entry in the first row of $C$ is equal to two times the identity,
this complex is homotopy equivalent to 
\begin{equation}\label{eq:UnivTrefoil}
q^3 \;
\begin{tikzpicture}[scale=.25,rotate=90, anchorbase]
\draw[gray] (1,0) circle (2 and 2.5);
\draw[gray] (1,0) circle (.15 and .2);
\draw[very thick] (1,0) circle (1.5 and 2);
\draw[very thick] (1,0) circle (.75 and 1);
\end{tikzpicture}
\xrightarrow{A}
q^3 \;
\begin{tikzpicture}[scale=.2,rotate=90, anchorbase]
\draw[gray] (1,0) circle (2 and 2.5);
\draw[gray] (1,0) circle (.15 and .2);
\draw[very thick] (1,0) circle (1 and 1.33);
\node at (0,-1.5) {\tiny$2$};
\end{tikzpicture}
\oplus
q \;
\begin{tikzpicture}[scale=.2,rotate=90, anchorbase]
\draw[gray] (1,0) circle (2 and 2.5);
\draw[gray] (1,0) circle (.15 and .2);
\draw[very thick] (1,0) circle (1 and 1.33);
\node at (0,-1.5) {\tiny$2$};
\end{tikzpicture}
\xrightarrow{B'}
q \;
\begin{tikzpicture}[scale=.2,rotate=90, anchorbase]
\draw[gray] (1,0) circle (2 and 2.5);
\draw[gray] (1,0) circle (.15 and .2);
\draw[very thick] (1,0) circle (1 and 1.33);
\node at (0,-1.5) {\tiny$2$};
\end{tikzpicture}
\xrightarrow{C'}
\uwave{
q^{-3} \;
\begin{tikzpicture}[scale=.2,rotate=90, anchorbase]
\draw[gray] (1,0) circle (2 and 2.5);
\draw[gray] (1,0) circle (.15 and .2);
\draw[very thick] (1,0) circle (1 and 1.33);
\node at (0,-1.5) {\tiny$2$};
\end{tikzpicture}
}
\end{equation}
where $A$ is as above, $B' = \begin{pmatrix} 0 & 0  \end{pmatrix}$, and 
\[
C' = 
\frac{1}{2}
\xy
(0,0)*{
\begin{tikzpicture}[scale=.2, anchorbase]
	\draw [thick] (0,0) to (0,1);
	\draw [thick] (0,1) to [out=150,in=210] (0,3);
	\draw [thick] (0,1) to [out=30,in=330] (0,3);	
	\draw [thick] (0,3) to (0,4);
	\node at (.5,2) {\tiny$\bullet$};
	\draw [thick] (0,4) to [out=150,in=210] (0,6);
	\draw [thick] (0,4) to [out=30,in=330] (0,6);	
	\draw [thick] (0,6) to (0,7);
	\node at (-.5,5) {\tiny$\bullet$};
\end{tikzpicture}
};
\endxy 
-
\xy
(0,0)*{
\begin{tikzpicture}[scale=.25, anchorbase]
	\draw [thick] (0,0) to (0,1);
	\draw [thick] (0,1) to [out=150,in=210] (0,3);
	\draw [thick] (0,1) to [out=30,in=330] (0,3);	
	\draw [thick] (0,3) to (0,4);
	\node at (-.5,2) {\tiny$\bullet$};
	\node at (.5,2) {\tiny$\bullet$};
\end{tikzpicture}
};
\endxy 
\]
Applying $\Upsilon^{\wedge}_{\cal{V}_2,t}$
to \eqref{eq:UnivTrefoil}, 
as well as the requisite shifts from Theorem \ref{thm:LinkHomologies},
we find that $\KhR_2(\cal{T})$ is the homology of the following complex:
\[
\uwave{ q^{-3} \cal{V}_2 \otimes \cal{V}_2 } \xrightarrow{\pi \circ \left(\begin{smallmatrix} \id \otimes \id \\ t \otimes \id \end{smallmatrix}\right)}
 q^{-3} \bV^2(\cal{V}_2) \oplus q^{-5} \bV^2(\cal{V}_2) \xrightarrow{0}
 q^{-5} \bV^2(\cal{V}_2) \xrightarrow{0}
 q^{-9} \bV^2(\cal{V}_2) .
\]
This complex is homotopy equivalent to the complex
\[
\uwave{ q^{-3} \Sym^2( \cal{V}_2 ) } \xrightarrow{}
q^{-5} \bV^2(\cal{V}_2) \xrightarrow{0}
q^{-5} \bV^2(\cal{V}_2) \xrightarrow{0}
q^{-9} \bV^2(\cal{V}_2) 
\]
in which the first map is surjective, so we find that 
the Poincare polynomial of the Khovanov homology of $\cal{T}$ is
\[
\dim_{q,h}(\KhR_2(\cal{T})) = (q^{-1}+q^{-3})+h^2q^{-5}+h^3q^{-9}
\]
and, after negating $q$-degree (as indicated in Remark \ref{rem:KhRonthenose}), this gives
\begin{equation}\label{eq:KhR2Trefoil}
\dim_{q,h}(\KhR^{\vee}_2(\cal{T})) = (q+q^{3})+h^2q^{5}+h^3q^{9}.
\end{equation}

If we instead apply $\Upsilon^{\cal{S}}_{\cal{V}_2,t}$ to \eqref{eq:UnivTrefoil}, 
as well as the shifts from Theorem \ref{thm:negativen},
we find that $\cal{H}_{-2}(\cal{T})$ is the homology of the following complex:
\[
q^9 \cal{V}_2 \otimes \cal{V}_2 \xrightarrow{\pi \circ \left(\begin{smallmatrix} \id \otimes \id \\ t \otimes \id \end{smallmatrix}\right)}
q^9 \Sym^2(\cal{V}_2) \oplus q^7 \Sym^2(\cal{V}_2) \xrightarrow{0}
q^7 \Sym^2(\cal{V}_2) \xrightarrow{-\frac{1}{2} \pi \circ (t \otimes t) \circ \iota}
\uwave{ q^3 \Sym^2(\cal{V}_2) }
\]
This is homotopy equivalent to the complex
\[
\xymatrix{
q^9 \bV^2(\cal{V}_2) \ar[r]
& q^7 \Sym^2(\cal{V}_2) \ar[r]^-{0}
& q^7 \Sym^2(\cal{V}_2) \ar[r]^-{}
&  \uwave{ q^3 \Sym^2(\cal{V}_2) }
}
\]
in which the first map is injective and the final map is non-zero (in the only possible degree). 
Thus,
\begin{equation}\label{eq:H-2Trefoil}
\dim_{q,h}(\cal{H}_{-2}(\cal{T})) = h^{-2}(q^7+q^5) + h^{-1}(q^9+q^7) + (q^3+q)
\end{equation}
Both \eqref{eq:KhR2Trefoil} and \eqref{eq:H-2Trefoil} decategorify to $q+q^3+q^5-q^9$, 
the unreduced Jones polynomial of the trefoil, as expected.

\normalem
\bibliographystyle{plain}
\bibliography{bib-skew}

%

%
\end{document}